\documentclass[a4paper]{amsart}
\usepackage[utf8]{inputenc}
\usepackage{amsmath}
\usepackage{amssymb}
\usepackage{amsthm}
\usepackage{tkz-graph}
\usepackage[backend=bibtex]{biblatex}
\usepackage{enumerate} 
\usepackage{bm} 
\usepackage{varwidth}
\usepackage{listings}
\usepackage{graphicx}
\usepackage{color}
\usepackage{pstricks}
\theoremstyle{definition}
\newtheorem{theorem}{Theorem}

\newtheorem{corollary}{Corollary}
\newtheorem{lemma}{Lemma} 
\newtheorem{definition}{Definition}

\newtheorem{property}{Property}

\newcommand{\Cc}{\mathcal{C}}

\newcommand{\calC}{\mathcal{C}}

\newcommand{\N}{\operatorname{N}}

\newcommand{\iso}{\cong}
\newcommand{\dist}{\operatorname{dist}}
\newcommand{\cupdot}{\mathbin{\mathaccent\cdot\cup}}

\title{An invariant for minimum triangle-free graphs}
\author{Oliver Kr\"uger}

\begin{document}

\maketitle

\begin{abstract}
We study the number of edges, $e(G)$, in triangle-free graphs with a prescribed
number of vertices, $n(G)$, independence number, $\alpha(G)$, and number of
cycles of length four, $\N(C_4;G)$. We in particular show that
$$3e(G) - 17n(G) + 35\alpha(G) + \N(C_4;G) \geq 0$$
for all triangle-free graphs $G$. We also characterise the graphs that satisfy
this inequality with equality.
\end{abstract}

\section{Introduction}

\subsection{Background}

The \emph{(minimum) edge numbers}, $e(3,k,n)$, are defined as the minimum number of edges in a triangle-free graph on $n$ vertices without an independent set of size $k$. These numbers, and constructions of some related graphs, have successfully been used to compute, or bound, the classical two-colour Ramsey numbers $R(3,\ell)$. In particular for $\ell = 6$ by Kalbfleisch \cite{kalbfleisch}, for $\ell = 7$ by Graver and Yackel \cite{graver-yackel} and for $\ell = 9$ by Grinstead and Roberts \cite{grinstead-roberts}. Among the useful upper bounds on the Ramsey numbers $R(3,\ell)$ that have been obtained by these considerations are those of Radziszowski and Kreher (e.g. \cite{radziszowski-kreher91}).

In particular Radziszowski and Kreher proved, in \cite{radziszowski-kreher91}, that $e(3,k+1,n) \geq 6n-13k$ for all non-negative integers $n$ and $k$. One may differently phrase that result by saying that $t(G) := e(G) - 6n(G) + 13 \alpha(G) \geq 0$ for all triangle-free simple graphs $G=(V,E)$ where $e(G) = |E|$ denotes the number of edges, $n(G) = |V|$ the number of vertices and $\alpha(G)$ the independence number of $G$. Moreover the triangle-free graphs $G$ for which $t(G) = 0$ have been classified in part by Radziszowski and Kreher in \cite{radziszowski-kreher91} and completely by Backelin in \cite{jart}. The invariant $t$ is just one in a series of invariants of a similar kind, all of which give bounds on the edge-numbers and for which there is a classification of the triangle-free graphs that satisfy them with equality. In particular we have $e(G) \geq 0$, $e(G) - n(G) + \alpha(G) \geq 0$, $e(G) - 3n(G) + 5\alpha(G) \geq 0$ and $e(G) - 5n(G) + 10\alpha(G) \geq 0$, with full classification of graphs for which we have equality (see, e.g. \cite{radziszowski-kreher91}).

In this paper we consider a related invariant, $\nu(G)$, which we define as
$$\nu(G) = 3e(G) - 17n(G) + 35 \alpha(G) + \N(C_4;G),$$
where $\N(C_k;G)$ denotes the number of cycles of length $k$ in $G$. We will,
in particular, show that $\nu(G) \geq 0$ for all triangle-free graphs $G$
(see Theorem \ref{mainthm} in Section \ref{nuzerographs}).
 
This affirmatively answers a question first considered in \cite{carc}. We also give a classification of the graphs that satisfy this inequality with equality. We will see that this bound is tight since there are (infinitely many) triangle-free graphs $G$ for which $\nu(G) = 0$. Interestingly, these graphs seem to be closely related to those for which $t(G) = 0$. In particular there are infinitely many triangle-free graphs for which $t(G) = \nu(G) = 0$.

A specialisation of the bound $\nu(G) \geq 0$ in Theorem \ref{mainthm}. Is the bound $e(C_{\leq 4},k+1,n) \geq \frac{17}{3}n - \frac{35}{3}k$, where $e(C_{\leq 4},k+1,n)$ is the number of edges in a graph containing no cycle of length at most four on $n$ vertices without an independent set of size $k + 1$.

It is not wholly unnatural to involve the quantity $\N(C_4;G)$ in the bound. Assume that $G$ is some triangle-free graph. Consider, similarly to Graver and Yackel in \cite{graver-yackel}, the following
$$2e^2(G) := \sum_{v \in V(G)} d(v)^2,$$
where $V(G)$ denotes the set of vertices in $G$ and $d_G(v) = d(G;v)$ denotes the valency of the vertex $v$ in the graph $G$, where we leave out the subscript if it is clear which graph we are considering from the context. Let $d^2(v)$, called the \emph{second valency} of a vertex $v$, be defined as
$d^2(v) = d^2(G;v) = \sum_{w \in N(v)} d(w)$, where $N_G(v)$ denotes the neighbourhood of $v$ (set of vertices adjacent to $v$) in $G$, where again $G$ is dropped from notation if clear from context.

We will by $G_v$ denote the induced subgraph of $G$ obtained by removing $v$ and all its neighbours. We let $\N(C_k;G,S)$ denote, where $S \subseteq V(G)$ is a subset of vertices of $G$, the number of cycles of length $k$ in $G$ that contains at least one vertex from $S$. If $S = \{v\}$ we will omit the use of set parentheses and write $\N(C_k;G,v)$ instead of $\N(C_k;G,\{v\})$.

We have that
$$e^2(G) - e^2(G_v) = \sum_{w \in N(v)} \left( d^2(w) + \binom{d(w)}{2} \right) - \binom{d(v)}{2} - (\N(C_4;G) - \N(C_4;G,v)).$$
Differences of these kinds are quite essential to the methods of Graver and Yackel and the quantities involved in such differences might therefore be interesting to study in relation to graphs with low edge numbers.

\subsection{Graphs with $\nu$-value zero and the main theorem}
\label{nuzerographs}

We already know of some triangle-free graphs $G$ such that $\nu(G) = 0$. We will here describe all graphs with $\nu$-value zero. That these are indeed all such graphs will be demonstrated in the conclusion of this article.

We need to define the following class of graphs (which appears in \cite{jart} and \cite{carc} as \emph{chains} denoted by $Ch_k$, in \cite{radziszowski-kreher91} as $F_k$ and in \cite{cfjones} as $H_k$). These graphs will also play an important role in our proofs. 

\begin{definition}
\label{defChk}
Let $Ch_2$ be a cycle of length five. We recursively define $Ch_{k+1}$ for $k \geq 2$. Let $x \in V(Ch_k)$ be some bivalent vertex. Let $V(Ch_{k+1}) = V(Ch_k) \cupdot \{v,w_1,w_2\}$ and $E(Ch_{k+1}) = E(Ch_k) \cup \{vw_1,vw_2,w_1x\}\cup \{w_2y; y \in N(x)\}$.
\end{definition}

It is easy to verify that $Ch_k$ is then well-defined for $k \geq 2$, i.e. up to isomorphism the result does not depend on the choice of bivalent vertex in the recursive construction. It is also easy to check that $n(Ch_k) = 3k-1$, $e(Ch_k) = 5k-5$, $\alpha(Ch_k) = k$ and $\N(C_4;Ch_k) = k-2$. Hence, $\nu(Ch_k) = 0$ for all $k \geq2$.

There are two connected 3-regular graphs with $\nu$-value 0. These have been characterised in \cite{carc}. Using the same notation as there we define the graphs $(2C_7)_{2i}$ and $W_5$ as follows. Let $V((2C_7)_{2i}) = \{a_0,a_1,\dots,a_6\}\cup \{b_0,b_1,\dots,b_6\}$ and the edges of $(2C_7)_{2i}$ be such that both $a_0,a_1,\dots,a_6$ and $b_0,b_1,\dots,b_6$ form cycles of length seven in $(2C_7)_{2i}$. Connect these two cycles by adding an edge $b_ia_{2i}$ for all $i \in \{0,1,\dots,6\}$, taking indices modulo 7.

\begin{figure}[h]
\begin{center}
\begin{tikzpicture}[scale=1,rotate=90]
  \GraphInit[vstyle=Classic]
  \renewcommand*{\VertexSmallMinSize}{4pt}
  \Vertices[Math,unit=2.2,Lpos=90]{circle}{a_0,a_1,a_2,a_3,a_4,a_5,a_6}
  \Vertices[Math,unit=1,Lpos=20,Ldist=-4]{circle}{b_0,b_4,b_1,b_5,b_2,b_6,b_3}
  \Edges(a_0,a_1,a_2,a_3,a_4,a_5,a_6,a_0,b_0,b_1,b_2,b_3,b_4,b_5,b_6,b_0)
  \Edge(a_1)(b_4)
  \Edge(a_2)(b_1)
  \Edge(a_3)(b_5)
  \Edge(a_4)(b_2)
  \Edge(a_5)(b_6)
  \Edge(a_6)(b_3)
\end{tikzpicture}
\end{center}
\caption{The graph $(2C_7)_{2i}$.}
\end{figure}
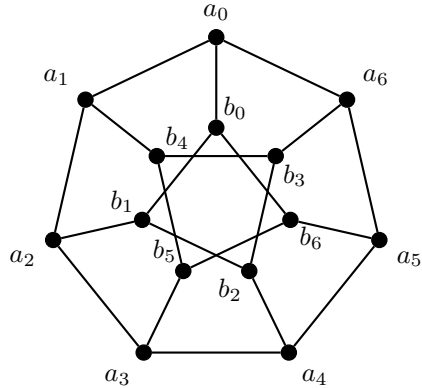
This graph is also known as a generalised Petersen graph, variously denoted $\textnormal{GP}(7,2)$ or $\textnormal{P}(7,2)$.

Let $V(W_5) = \{a_0,a_1\}\cup \{b_0,\dots,b_4\} \cup \{c_0,\dots,c_7\}$ and the edges of $W_5$ be such that $a_0a_1$ are adjacent, $b_0,\dots,b_4$ are independent and $c_0,\dots,c_7$ form a cycle of length eight. Add edges $b_ia_i$ for $i \in \{1,2,3,4\}$ taking $a_i$-indices modulo 2. Also add edges $b_ic_{2i}$ and $b_ic_{2i + 3}$ for $i \in \{1,2,3,4\}$ taking indices modulo $8$.

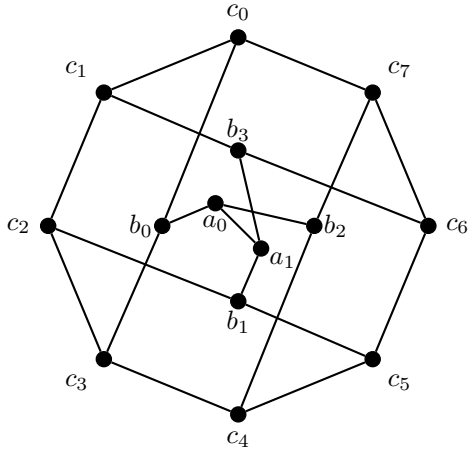
\begin{figure}[h]
\begin{center}
\begin{tikzpicture}[scale=1,rotate=90]
  \GraphInit[vstyle=Classic]
  \renewcommand*{\VertexSmallMinSize}{4pt}
  \Vertices[Math,unit=2.5,Lpos=90]{circle}{c_0,c_1,c_2,c_3,c_4,c_5,c_6,c_7}
  \Edges(c_0,c_1,c_2,c_3,c_4,c_5,c_6,c_7,c_0)
  \begin{scope}[rotate=0]
    \Vertices[Math,unit=1,Lpos=90,Ldist=-3]{circle}{b_3,b_0,b_1,b_2}
  \end{scope}
  \Vertex[Math,x=0.3,y=0.3,Lpos=-90,Ldist=-2]{a_0}
  \Vertex[Math,x=-0.3,y=-0.3,Lpos=-80,Ldist=-5]{a_1}
  \Edge(a_0)(a_1)
  \Edge(a_0)(b_0)
  \Edge(a_0)(b_2)
  \Edge(a_1)(b_1)
  \Edge(a_1)(b_3)
  \Edge(b_0)(c_0)
  \Edge(b_0)(c_3)
  \Edge(b_2)(c_7)
  \Edge(b_2)(c_4)
  \Edge(b_1)(c_2)
  \Edge(b_1)(c_5)
  \Edge(b_3)(c_1)
  \Edge(b_3)(c_6)
\end{tikzpicture}
\end{center}
\caption{The graph $W_5$.}
\end{figure}

Property \ref{no3regular} will show, after Theorem \ref{mainthm} has been established, that these two 3-regular graphs are the only 3-regular connected graphs graphs with $\nu$-value zero. This extends a result in \cite{carc} that states that these two graphs are the only two 3-regular graphs with $\nu$-value zero that neither contains cycles of length three nor of length four.

Let $BC_k$, $k \geq 4$, be a graph consisting of an induced cycle on vertices $c_1,c_2\dots,c_{2k}$ and one induced cycle on vertices $d_1,d_2\dots,d_k$. Connect the cycles by edges $d_ic_{2i-2}$ and $d_ic_{2i+1}$ for $i \in \{1,\dots,k\}$, taking indices modulo $2k$ for $c_i$s and modulo $k$ for $d_i$s. The graphs $BC_k$ have been called \emph{bicycles} (in \cite{carc} and \cite{jart}) or \emph{extended $k$-chains} (in \cite{cfjones}, denoted $E_k$) and $G_k$ (in \cite{radziszowski-kreher91}).

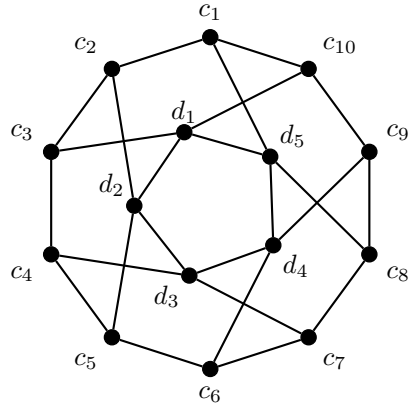
\begin{figure}[h]
\begin{center}
\begin{tikzpicture}[scale=1,rotate=90]
  \GraphInit[vstyle=Classic]
  \renewcommand*{\VertexSmallMinSize}{4pt}
  \Vertices[Math,unit=2.2,Lpos=90]{circle}{c_1,c_2,c_3,c_4,c_5,c_6,c_7,c_8,c_9,c_{10}}
  \begin{scope}[rotate=20]
  \Vertices[Math,unit=1,Lpos=90,Ldist=-3]{circle}{d_1,d_2,d_3,d_4,d_5}
  \end{scope}
  \Edges(c_1,c_2,c_3,c_4,c_5,c_6,c_7,c_8,c_9,c_{10},c_1)
  \Edges(d_1,d_2,d_3,d_4,d_5,d_1)
  \Edge(d_1)(c_3)
  \Edge(d_1)(c_{10})
  \Edge(d_2)(c_5)
  \Edge(d_2)(c_2)
  \Edge(d_3)(c_7)
  \Edge(d_3)(c_4)
  \Edge(d_4)(c_9)
  \Edge(d_4)(c_6)
  \Edge(d_5)(c_1)
  \Edge(d_5)(c_8)
\end{tikzpicture}
\end{center}
\caption{The graph $BC_5$.}
\end{figure}

Note that we have $n(BC_k) = 3k$, $e(BC_k) = 5k$. It is not difficult to show that $\alpha(BC_k) = k$. Moreover, for $k \geq 5$ we have $\N(C_4; BC_k) = k$. Hence $\nu(BC_k) = 0$. In the case $k = 4$ we have one ``extra'' cycle of length four formed by the vertices $d_1,d_2,d_3$ and $d_4$ and because of this we have $\nu(Ch_4) = 1$.

Note also that we will not be considering the empty graph, $G = (\emptyset,\emptyset)$. If one were to consider it however, then it is reasonable to define $n(G) = e(G) = \alpha(G) = \N(C_4;G) = 0$. Thus also $G$ would be a graph with $\nu$-value zero. All graphs in this paper will however be assumed to be non-empty.

We will prove the following theorem, which is the main result of this paper, in Section \ref{conclusion}. We now introduce the following notation for the family of graphs with $\nu$-value zero that have been defined in this section.
$$\mathcal{G} := \{W_5, (2C_7)_{2i}\} \cup \{Ch_k; k \geq 2\} \cup \{BC_k; k \geq 5\}.$$

\begin{theorem} \label{mainthm}
If $G$ is a triangle-free graph then $\nu(G) \geq 0$, and if $\nu(G) = 0$, with $G$ connected, then
$G \in \mathcal{G}$.
\end{theorem}

This means that the only triangle-free connected graphs with $\nu$-value zero are those that are those that we have defined in this section. Since $\nu$ is linear, in the sense that if $H_1 + H_2$ is the disjoint union of graphs $H_1$ and $H_2$ then $\nu(H_1 + H_2) = \nu(H_1) + \nu(H_2)$, we get that it is enough to classify the connected graphs with $\nu$-value zero as in Theorem \ref{mainthm}.

\subsection{Outline of the proof}

To establish Theorem \ref{mainthm} we, in Section \ref{prelnot}, introduce some preliminary results which we will need later while simultaneously establishing the bulk of the notation that will be used throughout the paper.

In Section \ref{theinvariant} we start the systematic study of the properties of the invariant $\nu$ for triangle-free graphs. In particular, in Section \ref{asssec} we derive some properties, relating to $\nu$, for graphs $G$ such that all proper subgraphs of $G$ have non-negative $\nu$-value and all proper subgraphs with $\nu$-value zero are in the family of graphs $\mathcal{G}$. When we later establish Theorem \ref{mainthm} all properties in this section will state general properties that hold for all triangle-free graphs $G$, since then the condition on the proper subgraphs always holds. The general idea is to derive some local properties for the structure of the neighbourhoods of low valency vertices in subgraphs of $G$ with low $\nu$-value.

Later, in Section \ref{mincounter} we strengthen our assumption, for a contradiction, on $G$ to say that $G$ is a minimal counter-example to satisfying Theorem \ref{mainthm}. This means that we assume that $G$ is a triangle-free graph such that $\nu(G) < 0$ or $\nu(G) = 0$ but $G \notin \mathcal{G}$, while $\nu(H) \geq 0$ with equality if and only if $H \in \mathcal{G}$ for all proper subgraphs $H$ of $G$. Unlike the results from Section \ref{theinvariant} the properties we derive in Section \ref{mincounter} do not give us general results for triangle-free graphs relating to $\nu$ since we begin by assuming something which we will show leads to a contradiction.

In Section \ref{graphswithvalencies} we will mimic the work of Radziszowski and Kreher in \cite{radziszowski-kreher91} with the use of a slight modification of their proof mentioned by Backelin in \cite{jart}. This section mostly consists of reformulating their results to make them fit into our particular context.  

\subsection{Preliminaries and notation}
\label{prelnot}

We start by stating some preliminary lemmas which we will use later. These are for the most part easy results.

We will let $G-e$ denote, for $e \in E(G)$ the graph obtained by removing the edge $e$ from the graph $G$. A graph is called \emph{edge-critical} if it has the property that for all edges $e \in E(G)$ its independence number increases as we remove it, i.e. $\alpha(G-e) > \alpha(G)$. This property is also commonly known as the graph being $\alpha$-critical (for example in \cite{berge}). It is easily verified that all the graphs $G$ such that $\nu(G) = 0$ defined in the previous section are edge-critical. In fact, it will follow from Theorem \ref{mainthm} that all triangle-free graphs $G$ such that $\nu(G) < 3$ must be edge-critical.

By abuse of notation we will often say that graphs are equal when they in fact are merely isomorphic. It should be clear from the context that we only need to consider graphs up to isomorphism in such instances. Therefore we either use the notation $H \iso G$ or $H = G$ to say that the graphs $H$ and $G$ are isomorphic.

If $S$ is a set and $k \geq 1$ we will let $\binom{S}{k}$ denote the set of all subsets of $S$ of size $k$.

A vertex $v \in V(G)$ will be said to be \emph{monovalent} (resp. \emph{bivalent}, \emph{trivalent}, \emph{tetravalent} etc.) if $d(v) = 1$ ($d(v) = 2$, $d(v) = 3$, $d(v) = 4$, etc.) in $G$.

\begin{lemma} (Lemma 2.4 in \cite{jart})
\label{ecalpha}
If $G$ is an edge-critical triangle-free graph, then $\alpha(G_v) = \alpha(G) - 1$ for all $v \in V(G)$.
\end{lemma}
\begin{proof}
It is obviously true for $v \in V(G)$ such that $d(v) = 0$. Suppose therefore that $v$ is at least monovalent.

We see that $\alpha(G_v) \leq \alpha(G) - 1$ since $S \cup \{v\}$ is an independent set of $G$ for all $v \in V(G)$ and any maximum independent set $S$ of $G_v$.

If $\alpha(G_v) + 2 \leq \alpha(G)$ then there would be a maximum independent set, $S$, of size at least $\alpha(G_v) + 3$ in $G - e$, where $e = \{v,w\}$, since $G$ is edge-critical. Then $v \in S$ and thus no $G$-neighbours of $v$ other than $w$ is in $S$. Hence $S\setminus \{v,w\}$ is an independent set of size at least $\alpha(G_v) + 1$ in $G_v$, a contradiction.
\end{proof}

For a vertex $v$ in a graph $G$ we denote the vertices at distance exactly $k$ from $v$ in $G$ by $N_{k,G}(v)$. If there is no ambiguity for the graph $G$ we will just write $N_k(v)$.
For $S \subseteq V(G)$, we shall let $G[S]$ denote the induced subgraph on $S$ and by $G\setminus S$ the induced subgraph $G[V(G) \setminus S]$. If $S = \{v\}$ contains only one element we sometimes omit the usage of set parentheses and write $G\setminus v$ for $G\setminus \{v\}$. If $T \subseteq E(G)$ then we let $G - T$ denote the graph $(V(G),E(G)\setminus T)$ where the edges in $T$ are removed from $G$. If $T = \{e\}$ we will similarly omit the usage of set parentheses. Note in particular the distinction between $G \setminus e$ and $G - e$. A set of vertices $S \subseteq V(G)$ is said to \emph{destabilise} $G$ if $\alpha(G\setminus S) < \alpha(G)$. If $S$ destabilises $G$ then the set $S$ is called a \emph{destabiliser}. If $S$ is a destabiliser of $G$ such that for all proper subsets $T \subsetneq S$ the set $T$ does not destabilise $G$, then $S$ is called a \emph{minimal destabiliser}. If $G$ has no destabilisers of size $r$ or less then $G$ is said to be \emph{$r$-stable}.

We say that a subset $S \subseteq V(G)$ is \emph{connected in $G$} if $G[S]$ is a connected graph. We will use $V_k(G)$, for $k \geq 0$, to denote the set of all vertices in $G$ of valency exactly $k$, i.e. $V_k(G) = \{v \in V(G); d(v) = k\}$.

\begin{lemma} \label{lemma:alpha}
If $G$ is a triangle-free graph, $v \in V(G)$ and $N_2(v)$ does not destabilise $G_v$ then $\alpha(G) \geq \alpha(G_v) + d(v)$.
\end{lemma}
\begin{proof}
Suppose that $N_2(v)$ does not destabilise $G_v$. Then there is an independent set, $I$, of size $\alpha(G_v)$ in $G_v$ such that $I \cap N_2(v) = \emptyset$. Since $G$ is triangle-free we have that $I \cup N(v)$ is an independent set, of size $\alpha(G_v) + d(v)$.
\end{proof}

\begin{lemma} (Lemma 2.2 in \cite{jart})
\label{ecdestab}
If $G$ is a connected edge-critical triangle-free graph, $v \in V(G)$ and $d(v) \geq 2$ then $N_2(v)$ is a destabiliser of $G_v$.
\end{lemma}
\begin{proof}
Suppose that $N_2(v)$ did not destabilise $G_v$. By Lemma \ref{lemma:alpha} there is an independent set, of size $\alpha(G_v) + d(v) \geq \alpha(G_v) + 2$, in $H$, contradicting Lemma \ref{ecalpha}. Therefore $N_2(v)$ destabilises $G$.
\end{proof}

We will use the following well-known lemma about the minimum valency of edge-critical graphs.
\begin{lemma} (see e.g. \cite[Prop.~1,~Ch.~13]{bergecommonprop}) \label{ecmindeg}
If $G$ is an edge-critical graph then $\delta(G) \geq 2$ unless $G \iso K_1$ or $G \iso K_2$.
\end{lemma}

An edge $e \in E(G)$ will be called \emph{redundant} in $G$ if $\alpha(G-e) = \alpha(G)$, otherwise $e$ is said to be \emph{critical}. Note that a graph is edge-critical precisely when it has no redundant edges.

%
%
For a set of vertices $W \subseteq V(G)$ we let $N[W]$ denote the \emph{closed neighbourhood} of the vertices in $W$, which is the set of vertices that are either in $W$ or adjacent to a vertex in $W$. If $W = \{w\}$ has size one we omit the set parentheses and write $N[v] = N[\{v\}]$. If $S$ is an independent set of vertices we let $G_S$ denote the graph $G[V(G) \setminus N[S]]$, i.e. the graph obtained by removing all the vertices in $S$, all their neighbours and edges incident to all such vertices. When we have an independent set $S = \{s_1,s_2,\dots,s_k\}$ we will drop the usage of set parentheses in the subscript and write $G_{s_1,s_2,\dots,s_k}$ for $G_{\{s_1,s_2,\dots,s_k\}}$.

\begin{lemma} \label{ecbvconn} (Lemma 2.6 of \cite{jart})
Let $G$ be an edge-critical, connected and triangle-free graph. If $v \in V(G)$ is a bivalent vertex, then $G_v$ is connected.
\end{lemma}
\begin{proof}
Let $\{w_1,w_2\} = N(v)$ and we will write $X$ for $N(w_1) \cup N(w_2) \setminus \{v\}$. Suppose that $G_v = G_1 + G_2$ and set $X_i = X \cap V(G_i)$.

Not both $X_1 \cap N(w_1)$ and $X_2 \cap N(w_2)$ can be empty since $G$ is connected. We may without loss of generality assume that $X_1 \cap N(w_1) \neq \emptyset$.

Suppose that $X_1$ destabilises $G_1$ and let $S$ be a maximum independent set of $G - \{w_1,x\}$ where $x \in X_1 \cap N(w_1)$. Then $w_1 \in S$ and therefore $v \notin S$. If $w_2 \in S$ then $|S \cap V(G_1)| \leq \alpha(G_1) - 1$, and if $w_2 \notin S$ then $|S \cap N[v]| = 1$. Hence, in either case we have
$$|S| = |S \cap V(G_1)| + |S \cap V(G_2)| + |S\cap N[v]| \leq \alpha(G_1) + \alpha(G_2) + 1 = \alpha(G),$$
contradicting that $G$ is edge-critical.

Hence $X_1$ does not destabilise $G_1$ and analogously $X_2$ does not destabilise $G_2$. But then $\alpha(G) \geq \alpha(G_1) + \alpha(G_2) + 2$, contradicting Lemma \ref{ecalpha}.
\end{proof}

We first classify the minimal destabilisers of minimum size in $Ch_k$-graphs.

\begin{lemma}
\label{Gvdestabiliser}
If $G$ is edge-critical, $S \subseteq V(G)$ destabilises $G$ and $v \in V(G)\setminus S$, then $S \cap V(G_v)$ destabilises $G_v$.
\end{lemma}
\begin{proof}
Otherwise there would be a maximum independent set $T$ of $G_v$ avoiding $S \cap V(G_v)$.  $T = \alpha(G_v) = \alpha(G) - 1$, by Lemma \ref{ecalpha}, and therefore $T \cup \{v\}$ would be a maximum independent set of $G$ avoiding $S$. This contradicts that $S$ destabilises $G$.
\end{proof}

\begin{lemma} (Lemma 6.2(b) of \cite{carc})
\label{Chkdestab3}
If $S$ destabilises $G = Ch_k$, where $k \geq 2$, then $|S| \geq 3$ with equality if and only if $S = N[v]$ for some bivalent $v \in V(G)$.
\end{lemma}
\begin{proof}
Induction on $k$. Easily verified for $k = 2,3$. Let $k \geq 4$ and suppose that the statement holds for all $G = Ch_d$, where $d < k$.

Suppose that $|S| \leq 2$. Since $k \geq 4$ there is some bivalent vertex $v \in V(G) \setminus S$ and therefore, by Lemma \ref{Gvdestabiliser}, $S \cap V(G_v)$ would destabilise $G_v \iso Ch_{k-1}$. This, however, contradicts the inductive assumption. Hence $|S| \geq 3$.

Clearly $N[v]$ destabilises $G$ for all bivalent vertices $v \in V(G)$. Hence it only remains to show that if $|S| = 3$ then $S = N[v]$ for some bivalent vertex $v$. Suppose, for a contradiction, that $|S| = 3$ but $S \neq N[v]$ for all bivalent vertices $v \in V(G)$.

We then have that $S \cap N[v] \neq \emptyset$ for all bivalent vertices $v \in V(G)$, since otherwise $S$ would destabilise $G_u \iso Ch_{k-1}$ and therefore by the induction hypothesis $S = N_{G_v}[u]$ for some, in $G_v$, bivalent vertex $u \in V(G_v)$. But $u$ is not bivalent in $G$ since $S \neq N[u]$ and therefore $u$ has distance two from $v$. Let $T$ be an independent set of size $k-2$ in $G_{v,u}$. Then $T \cup N(v)$ is an independent set in $G$ of size $k$ avoiding $S$, contradicting that $S$ destabilises $G$.

Since $|S| = 3$ and there are four bivalent vertices in $G$ there is at least one bivalent vertex that is not in $S$. Let $v \in V(G)$ be such a vertex. Then $S \cap N(v) \neq \emptyset$ by the above. By induction $S \setminus N[v]$ does not destabilise $G_v$, which contradicts Lemma \ref{Gvdestabiliser}.
\end{proof}

For minimal destabilisers of size four we will not completely classify them, but the following lemma tells us that in all but one case they are connected.

\begin{lemma} (Lemma 6.2(e) of \cite{carc})
\label{Chkdestab4}
If $S$ is a minimal destabiliser of $G = Ch_k$ such that $|S| = 4$ and $S$ is not connected in $G$, then $k = 3$ and $S = V_2(Ch_k)$.
\end{lemma}
\begin{proof}
It is easy to verify, by checking all possibilities, that the statement holds for $G = Ch_2$ and $G = Ch_3$.

Suppose that $k \geq 4$. Let $S$ be a disconnected destabiliser of $G = Ch_k$ of size four and suppose that $S \not\supseteq N[u]$ for all bivalent vertices $u \in V_2(G)$.

Suppose $V_2(G) = \{s_1,s_2,s_3,s_4\}$, where $s_1s_2,s_3s_4 \in E(G)$. Let $u$ and $v$ denote the neighbours of $s_1$ and $s_2$ that are not in $V_2(G)$, respectively. Then $G_{s_1,v} \iso Ch_{k-2}$ and there is an independent set $I$ of size $k-2$ avoiding $\{s_3,s_4\}$, by Lemma \ref{Chkdestab4}. Then $I \cup \{u,v\}$ is an independent set of size $k$ in $G$ avoiding $V_2(G)$. Hence $V_2(G)$ does not destabilise $G$. Therefore there is some bivalent vertex, $v \in V_2(G)$, such that $v \notin S$. Fix such a vertex $v \in V_2(G)$.

By Lemmas \ref{Gvdestabiliser} and \ref{Chkdestab3} we have that $|S \cap V(G_v)| \geq 3$ and therefore $|S \cap N[v]| = |S|-|S\cap V(G_v)| \leq 4-3 = 1$. In fact, $S \cap N[v] = \emptyset$ since otherwise $|S \cap N[v]| = 1$ and therefore $|S \setminus N[v]| = 3$, where $S \setminus N[v]$ destabilises $G_v$ (otherwise add $v$ to an maximum independent set of $G_v$). Hence $S \setminus N[v] = N_{G_v}[u]$ for some $u \in V_2(G_v)$ by Lemma \ref{Chkdestab4}. But $S \not\supseteq N[u]$ for all $u \in V_2(G)$ so $u$ must be trivalent in $G$. Both neighbours of $v$ in $G$ are adjacent to $N_{G_v}[u]$ and $S \cap N(v) \neq \emptyset$, thus $S$ would be connected.

$S \not\supseteq N_{G_v}[u]$ for all $u \in V_2(G_v)$ since otherwise $u \in V_2(G_v)$ but $u \notin V_2(G)$ because $S$ is minimal so $|S \setminus N[u,v]| \leq 1$ and therefore there would be some independent set $I$ in $G_{v,u}$ of size $k-2$ avoiding $S$. Because $S \cap N[v] = \emptyset$ we would get that $I \cup N(v)$ would be independent in $G$ of size $k$, contradicting that $S$ destabilises $G$.

We will now use these facts to prove, by induction, that every disconnected destabiliser of size four in $G=Ch_k$, $k \geq 4$ contains $N[u]$ for some $u \in V_2(G)$. For $k = 4$ let $v \in V_2(G)$ be such that $v \notin S$ as in the above. Then since $S \cap N[v] = \emptyset$ and $\forall u \in V_2(G_v): S \not\supseteq N_{G_v}[u]$ we must have that $S = V_2(G_v)$. However, then there is an independent set, $I$, of size $2$ in $G_{v,w} \iso Ch_{2} \iso C_5$ avoiding $S$, where $w \in V_2(G_v)$, since $|S \cap V(G_{v,w})| = 2$ and therefore $S \cap V(G_{v,w})$ does not destabilise $G_{v,w}$. Then $I \cup N(v)$ would be a maximum independent set in $G$ avoiding $S$, contradicting that $S$ is a destabiliser.

Let $k \geq 5$ and suppose that every disconnected destabiliser of size four in $Ch_{\ell}$, $4 \leq \ell < k$ contains $N_{Ch_{\ell}}[u]$ for some $u \in V_2(Ch_{\ell})$. Suppose also that $S$ is a disconnected destabiliser of size four in $G = Ch_k$. Take $v \in V_2(G)$ such that $v \notin S$ as before. Then $S \cap N[v] = \emptyset$ and $\forall u \in V_2(G_v): S \not\supseteq N_{G_v}[u]$, thus we get that $S \supseteq N_{G_v}[u] = N[u]$ for some vertex $u \in V_2(G_v) \cap V_2(G)$.
\end{proof}

Many of the graphs we will study are such that if we remove a vertex $v$ and all its neighbours from $G$, then we get a $Ch_k$-graph for some $k \geq 2$. The following lemma tells us that if $v$ is bivalent with second valency six then the local neighbourhood of $v$ in $G$ has a certain structure.

\begin{lemma}
\label{lemma:starone}
Let $G$ be such that $G_v = Ch_k$ for some $k \geq 2$ and some $v \in V(G)$ such that $d(v) = 2$ and $d^2(v) = 6$. If $N_2(v)$ destabilises $G_v$ then either $N_2(v)$ contains a pair of adjacent bivalent vertices of $G_v$ or $\N(C_4; G, N(v)) \geq 2$.
\end{lemma}
\begin{proof}
Suppose $N_2(v)$ destabilises $G_v = Ch_k$. By Lemmas \ref{Chkdestab3} and \ref{Chkdestab4} in addition to the fact $|N_2(v)| \leq 4$ we have that at least one of the following three statements must hold:
\begin{enumerate}[(i)]
\item $N_2(v) \supseteq N_{G_v}[u]$ for some $u \in V_2(G_v)$, or
\item $k = 3$ and $N_2(v) = V_2(G_v)$, or
\item $|N_2(v)| = 4$ and $N_2(v)$ is connected.
\end{enumerate}
In cases (i) and (ii) we get that $N_2(v)$ does contain a pair of adjacent bivalent vertices of $G_v$. In case (iii) it is easily checked that $\N(C_4;G,N(v)) \geq 2$ since the only three connected triangle-free graphs on four vertices (that $N_2(v)$ possibly induce in $G_v$) are $P_4, C_4$ and $K_{1,3}$.
\end{proof}

\begin{lemma} (Lemma 2.10(b) in \cite{carc})
If $e \in E(G)$ is redundant in $G$ and $x \in V(G) \setminus N[e]$, then either
\begin{enumerate}[(i)]
  \item $e$ is redundant in $G_x$, too, or
  \item $\alpha(G_x) \leq \alpha(G) - 2$.
\end{enumerate}
\end{lemma}
\begin{proof}
Suppose that $\alpha(G_x - e) = \alpha(G_x) + 1 = \alpha(G)$. Then if $I$ is a maximum independent set in $G_x - e$ we get that $I \cup \{x\}$ is an independent set of size $\alpha(G) + 1$ in $G-e$, contradicting that $e$ is redundant.

Hence, either $\alpha(G_x - e) \neq \alpha(G_x) + 1$, in which case $e$ is redundant in $G_x$ as well, or $\alpha(G_x) + 1 \neq \alpha(G)$, in which case $\alpha(G_x) \leq \alpha(G) - 2$.
\end{proof}

For subsets of vertices $A,B \subseteq V(G)$ we write $E_G(A,B)$ for the set of edges with one endpoint in $A$ and the other endpoint in $B$, i.e. $E_G(A,B) = E(G) \cap \{\{a,b\}; a \in A, b \in B\}$. The cardinality of this set will be denoted by $e_G(A,B)$. If the graph $G$ is clear from context we will drop the subscript from the notation. We will also sometimes abuse the notation by writing $E_G(H_1,H_2)$ for $E_G(V(H_1),V(H_2))$ where $H_1$ and $H_2$ are two subgraphs of $G$.

If $S$ is a collection of sets then we let $\bigcup S$ denote the union of all the sets in $S$, i.e. $\bigcup S = \bigcup_{X \in S} X$. In particular, if $H$ is an induced subgraph of $G$ then $\bigcup \{e \cap V(H); e \in E(H,G \setminus H)\}$ is the set of vertices of $H$ that are adjacent, in $G$, to some vertex outside $H$.

\begin{lemma} \label{lemma:A}
If $H$ is an induced subgraph of $G$ and $M = \bigcup \{e \cap V(H); e \in E(H, G \setminus H)\}$ does not destabilise $H$. Then every edge in $E(H,G\setminus H)$ is redundant.
\end{lemma}
\begin{proof}
Suppose that $e \in E(H, G\setminus H)$ were not redundant, then $\alpha(G-e) = \alpha(G) + 1$. Let $S$ be a maximum independent set of $G-e$. $S' = S \cap V(H)$ is independent in $H$. Since $M$ does not destabilise $H$ there is a maximum independent set $S''$ of $H$ such that $S'' \cap M = \emptyset$. It follows that $(S \setminus S') \cup S''$ is independent, in $G - e$, of size at least $\alpha(G) + 1$. But since $(S \setminus S') \cup S''$ avoids $e \cap V(H)$ the set $(S \setminus S') \cup S''$ would also be independent in $G$, a contradiction.
\end{proof}

As an immediate consequence of this lemma we have the following corollary.

\begin{corollary} \label{corollary:A}
If $H$ is an induced subgraph of $G$, $H$ is $r$-stable and $e(H,G\setminus H) \leq r$ then all edges in $E(H,G\setminus H)$ are redundant.
\end{corollary}

In addition to Lemmas \ref{Chkdestab3} and \ref{Chkdestab4} which give us a description of the small minimal destabilisers in $Ch_k$-graphs we will also encounter situations where we have graphs such as $Ch_k$ with an extra edge added between two (in $Ch_k$) bivalent vertices. We will want to say something about the destabilisers of such a graph as well. In particular we will make use of the following lemma.

\begin{lemma} \label{lemma:Chkplusedge}
Let $G = Ch_k + e$ where $e = \{\alpha,\beta\} \notin E(Ch_k)$, $k \geq 3$. Moreover suppose that both $\alpha$ and $\beta$ are bivalent in $Ch_k$. Let $S \subseteq V(G)$, $|S| = 3$, be such that $S$ contains at least two of the bivalent vertices of $Ch_k$ and if $|S \cap V_2(Ch_k)| = 2$, then $S \cap V_2(Ch_k)$ is an independent set. Then $G$ is not destabilised by $S$.
\end{lemma}
\begin{proof}
If $S \cap \{\alpha,\beta\} \neq \emptyset$ let $S' := S$. By assumption $S' \neq N_{Ch_k}[v]$ for all $v \in V_2(Ch_k)$ and therefore $S'$ does not destabilise $Ch_k$ by Lemma \ref{Chkdestab3}.

Otherwise $|S \cap V_2(Ch_k)| = 2$. The vertex in $S \setminus V_2(Ch_k)$ is adjacent to at most one bivalent vertex in $Ch_k$ and therefore either $S \cup \{\alpha\}$ or $S \cup \{\beta\}$ is a disconnected set in $Ch_k$ of size four not containing all bivalent vertices of $Ch_k$. By Lemmas \ref{Chkdestab3} and \ref{Chkdestab4} this set is not a destabiliser of $Ch_k$. In this case define $S' := S \cup \{z\}$ where $z \in \{\alpha,\beta\}$ and $S'$ does not destabilise $Ch_k$.

There is some independent set $I$ of size $\alpha(Ch_k) = k$ in $Ch_k$ such that $I \cap S' = \emptyset$. But then $I$ contains at most one endpoint of $\{\alpha,\beta\}$ and therefore $I$ is also independent in $G$. Moreover $\alpha(G) = \alpha(Ch_k) = k$ since there are maximum independent sets in $Ch_k$ avoiding $\alpha$, for instance.

Hence $I$ is a maximum independent set in $G$ such that $S' \cap I = \emptyset$, so $S \subseteq S'$ does not destabilise $G$.
\end{proof}

Let $\mathcal{H}$ be a set of graphs. A graph $G$ is called \emph{$\mathcal{H}$-avoiding} if $G$ has no subgraph that is isomorphic to any of the graphs in $\mathcal{H}$. The set of cycle graphs of lengths $3,4,\dots,k$ will be denoted by $C_{\leq k}$. The length of the shortest cycle in $G$ is called the \emph{girth} of $G$. Note that $G$ is $C_{\leq k}$-avoiding if and only if $G$ has girth at least $k+1$.

The \emph{distance} between two vertices $u,v \in V(G)$ is denoted $\dist_G(u,v)$ and is defined to be the least number of edges in a path from $u$ to $v$, or infinity if there is no such path. If the graph is clear from context we will drop the subscript from the notation.

\begin{lemma}
\label{nbrsincycle}
If $G$ is a $C_{\leq 4}$-avoiding graph which contains a $k$-cycle $C=c_1,c_2,\dots,c_k$, then for all $v \in V(G)\setminus C$ we have $|N(v) \cap C| \leq \lfloor \frac{k}{3} \rfloor$.
\end{lemma}
\begin{proof}
Let $H := G[C]$ be the induced graph on $C$. Since $G$ is $C_{\leq 4}$-avoiding we have that $\forall{u_1,u_2} \in N(v): \dist_{G\setminus v}(u_1,u_2) \geq 3$. Hence, a fortiori, $\dist_H(u_1,u_2) \geq 3$ for all $u_1,u_2 \in N(v) \cap C$ so at most a third of the vertices of $C$ can be in the neighbourhood of $v$.  The lemma follows.
\end{proof}

In particular the previous lemma gives us that if we have a cycle of length five in a subgraph of $G$ then any vertex outside the cycle can be adjacent to at most one vertex in the cycle. This fact will be used frequently in what follows.

\section{The invariant and basic properties}
\label{theinvariant}


We will prove that $\nu$ is a non-negative invariant for all triangle-free graphs $G$. Let $\calC(G)$ denote the set of connected components of the graph $G$. Note in particular that $\nu(G) = \sum_{C \in \calC(G)} \nu(C)$.

\begin{property}
\label{prop-removalcount}
Let $G$ be a triangle-free graph, then
$$\forall v \in V(G): \nu(G_v) \leq \nu(G) - 3 d^2(v) + 17 d(v) - 18 - \N(C_4;G,N(v)).$$
\end{property}
\begin{proof}
Note that $n(G_v) = n(G) - d(v) - 1$ and $e(G_v) = e(G) - d^2(v)$ (since $G$ is triangle-free). Also we have that $\alpha(G_v) \leq \alpha(G) - 1$ since any maximum independent set in $G$ must either contain $v$ or a vertex in the neighbourhood of $v$. Also, $\N(C_4;G_v) = \N(C_4;G) - \N(C_4;G,N(v))$ since any cycle of length four through $v$ also goes through one of the neighbours of $v$.

Therefore, we get that
\begin{equation*}
\begin{aligned}
  \nu(G_v) &\leq \nu(G) - 3d^2(v) + 17d(v) + 17 - 35 - \N(C_4;G,N(v)) \\
    &= \nu(G) - 3d^2(v) + 17d(v) - 18 - \N(C_4;G,N(v)).
\end{aligned}
\end{equation*}
\end{proof}

We will let $H \leq G$ denote that $H$ is a subgraph of $G$ and $H < G$ will denote that $H$ is a proper subgraph of $G$ (i.e. that $H \leq G$ but $H$ is not equal to $G$).

We will often want to determine a bound for $\nu(G_{s_1,s_2,\dots,s_k})$ in $\nu(G)$. We will then use the notation
$$\nu(G_{s_1,s_2,\dots,s_k}) \leq \nu(G) + d_1 - (c_1) + d_2 - (c_2) + \dots + d_k - (c_k)$$
to indicate that $\nu(G_{s_1,s_2,\dots,s_{\ell}}) - d_{\ell} + c_{\ell} \leq \nu(G_{s_1,s_2,\dots,s_{\ell-1}})$ and that $\N(C_4; G_{s_1,s_2,\dots,s_{\ell}}) + c_{\ell} \leq \N(C_4; G_{s_1,s_2,\dots,s_{\ell-1}})$. Intuitively this may be thought of as saying that when removing $s_1$ and all its neighbours from $G$ the $\nu$-value increases by at most $d_1 - c_1$ and we remove at least $c_1$ vertices from the graph. Then we remove $s_2$ from $G_{s_1}$ and $d_2 - c_2$ is gives a bound on the increase in $\nu$-value from $G_{s_1}$ to $G_{s_1,s_2}$ and $c_2$ indicates the least number of cycles known to be removed. It is sometimes useful to think of $G_{s_1,s_2,\dots,s_k}$ as the vertices, $s_1,s_2,\dots,s_k$, are being removed in sequence and in each step we keep track of how much the $\nu$-value and the number of cycles of length four changes.

\subsection{Properties of a graph for which all subgraphs have non-negative $\nu$-value}
\label{asssec}

From here on we assume that $G$ is a triangle-free graph such that every
graph with fewer vertices or the same number of vertices but fewer edges than $G$ has non-negative  $\nu$-value, i.e.
\begin{equation}
  \tag{A1} \label{bigass}
  \begin{array}{l}
\forall H: n(H) < n(G) \text{ or } n(H) = n(G) \wedge e(H) < e(G): \\
\quad (1) \; \nu(H) \geq 0 \text{ and } \\
\quad (2) \; \text{if } \nu(H) = 0, H \text{ connected, then } H \in \mathcal{G}.
 \end{array}
\end{equation}
Note that this means in particular that for all $H < G$ (all proper subgraphs $H$ of $G$) we assume that $\nu(H) \geq 0$ and if $\nu(H) = 0$, then $H \in \mathcal{G}$. This is how we will most often use assumption \eqref{bigass}.

We then derive some properties for $G$ and its subgraphs. Note that when Theorem \ref{mainthm} has been established all properties for $G$ that have been derived under assumption \eqref{bigass} will be shown to hold in general, for all triangle-free graphs $G$.

\begin{lemma} \label{nu-1connected}
Let $\nu(G) \leq -1$ then $G$ is connected.
\end{lemma}
\begin{proof}
If $\nu(G) \leq -1$ but $G = G' + G''$ for nonempty graphs $G'$ and $G''$ then we would have that $\nu(G') + \nu(G'') = \nu(G) \leq -1$, thus $\nu(G') \leq -1$ or $\nu(G'') \leq -1$. This would contradict assumption \eqref{bigass}.
\end{proof}

\begin{property}
\label{prop-edge-critical}
  If $H \leq G$ and $\nu(H) \leq 2$ then $H$ is edge-critical.
\end{property}
\begin{proof}
Suppose that $H \leq G$ and $\nu(H) \leq 2$ but $H$ is not edge-critical. Then there is an edge $e \in E(H)$ such that $\alpha(H-e) = \alpha(H)$. Hence,
$$\nu(H-e) = \nu(H) - 3 \leq -1,$$
which contradicts assumption \eqref{bigass}.
\end{proof}

We let $\delta(H)$ denote the minimum valency of a graph $H$ and $\Delta(H)$ the maximum valency. The complete graph on $\ell$ vertices is denoted by $K_{\ell}$.

\begin{property}
  \label{prop-mindeg1}
  If $H \leq G$ and $\nu(H) \leq 17$ then $H$ is 1-stable and $\delta(H) \geq 1$.
\end{property}
\begin{proof}
First note that
$$\nu(K_1) = 3 \cdot 0 - 1 \cdot 17 + 35 + 0 = 18.$$
Suppose that $H \leq G$ is such that $\nu(H) \leq 17$ but $\delta(H) = 0$. Let $v \in V(H)$ be a vertex of valency zero. Then we have that
$$\nu(H-v) = \nu(H) - \nu(K_1) \leq 17 - 18 = -1.$$ 
Moreover, if $S$ were a 1-destabiliser then $\nu(H \setminus S) \leq \nu(H) - 18 \leq -1$. In both situations we would get a contradiction to assumption \eqref{bigass}.
\end{proof}

\begin{property}
  \label{prop-mindeg2}
  If $H \leq G$ and $\nu(H) \leq 3$ then $\delta(H) \geq 2$.
\end{property}
\begin{proof}
  Let $H \leq G$ be such that $\nu(H) \leq 3$. Suppose that $\delta(H) < 2$, then $\delta(H) = 1$ since $\delta(H) = 0$ would contradict Property \ref{prop-mindeg1}. Let $v \in V(H)$ be a monovalent vertex such that $d^2(v) = \max\{d^2(v); v \in V(H), d(v) = 1\}$. If $d^2(v) = 1$ then $v$ belongs to a $K_2$-component of $H$, which means that $H = H' + K_2$ for some subgraph $H' < G$. But then $3 \geq \nu(H) = \nu(H') + \nu(K_2) = \nu(H') + 4$, which gives us that $\nu(H') \leq -1$ which contradicts assumption \eqref{bigass}.

Hence, $d^2(v) \geq 2$ and therefore we get by Property \ref{prop-removalcount} that
$$\nu(H_v) \leq \nu(H) - 6 + 17 - 18 - \N(C_4;H,N(v)) \leq -4,$$
again contradicting assumption \eqref{bigass}.
\end{proof}

\begin{property}
  \label{mindegleq4}
  If $H \leq G$ and $\nu(H) \leq 7$ then $\delta(H) \leq 4$.
\end{property}
\begin{proof}
  Suppose that $H \leq G$, $\nu(H) \leq 7$ but $\delta(H) = \delta \geq 5$. Then there is a vertex $v \in V(H)$ of minimum valency $d(v) = \delta $. Such a vertex must have $d^2(v) \geq \delta^2$ and therefore by Property \ref{prop-removalcount}
  $$\nu(H_v) \leq \nu(H) - 3\delta^2 + 17 \delta - 18 - \N(C_4;H,N(v)) \leq \nu(H) - 8 \leq -1,$$
since $-3\delta^2 + 17 \delta \leq 10$ for $\delta \geq 5$, contradicting assumption \eqref{bigass}.
\end{proof}

\begin{property} \label{prop1}
If $H \leq G$ and $\nu(H) \leq 3$ then $H$ is 2-stable.
\end{property}
\begin{proof}
By Property \ref{prop-mindeg2} we get that $\delta(H) \geq 2$. If $S$ is a destabiliser of size 2, then $\nu(H\setminus S) \leq \nu(H) - 9 + 34 - 35 \leq -7$, contradicting assumption \eqref{bigass}.
\end{proof}

\begin{property}
  \label{prop-no2regular}
  If $H \leq G$, $\nu(H) \leq 6$, $C \in \calC(H)$ and $C$ is 2-regular, then $C = C_5$.
\end{property}
\begin{proof}
  Suppose that $H \leq G$ with $\nu(H) \leq 6$ and $C$ is a 2-regular component of $H$. Then $C = C_m$ for some $m \geq 4$ since $G$ is triangle-free and therefore so is $H$. We must have that $\nu(C) \leq 6$ since otherwise we would get that $\nu(H-C) = \nu(H) - \nu(C) \leq -1$. But $n(C_m) = e(C_m) = m$ and $\alpha(C_m) \geq \frac{m-1}{2}$, whence $\nu(C_m) \geq 3m - 17m + 35\left(\frac{m-1}{2}\right) = \frac{7(m-5)}{2}$. However, $\frac{7(m-5)}{2} \geq 7$ for $m \geq 7$ and $\nu(C_6) = 21$. Moreover, $\nu(C_4) = 15$ so the only remaining possibility is that $m = 5$ and $C = C_5$.
\end{proof}

Inductively, by assumption \eqref{bigass}, we have that if $H < G$ has a 3-regular component, then it is one of the two 3-regular graphs $W_5$ and $(2C_7)_{2i}$ defined in Section \ref{nuzerographs}.

\begin{property}
  \label{no3regular}
  If $\nu(G) \leq 0$, $C \in \calC(G)$ and $C$ is 3-regular, then $C \in \{(2C_7)_{2i}, W_5\}$.
\end{property}
\begin{proof}
The conclusion follows immediately from assumption \eqref{bigass} if $G$ is not connected. It is therefore enough to consider when $G$ is connected. Suppose that $G \notin \{(2C_7)_{2i}, W_5\}$.

Suppose furthermore that $\N(C_4;G) = 0$. By the corollary in \cite[p.~202,v.~2015-07-16]{carc} we have that $\nu(C) \geq 1$, since $C$ is $3$-regular and not one of the two graphs $(2C_7)_{2i}$ or $W_5$. But then we would have
$$\nu(G-C) = \nu(G) - \nu(C) \leq -1,$$
contradicting assumption \eqref{bigass}.

Hence, $\N(C_4; G) \geq 1$. Let $\{a,b,c,d\}$ be the vertices of a cycle of length four in $G$. Then $\nu(G_a) \leq \nu(G) + 6 - 1 \leq 5$ and $c$ is at most monovalent in $G_a$, whence $c$ is monovalent in $G_a$ by Property \ref{prop-mindeg1}. Since $e(N(a),N(c)) \leq 1$ (otherwise we would have a triangle) we get that $\nu(G_{a,c}) \leq \nu(G_a) - 7 = -2$, contradicting assumption \eqref{bigass}.
\end{proof}

\begin{property}
  \label{prop-leq6}
  For all $H \leq G$ such that $\nu(H) \leq 6$ either $\delta(H) \geq 2$ or $H$ is edge-critical and $H = H' + K_2$ where $\nu(H') = \nu(H) - 4 \leq 2$ (whence, in particular, $\delta(H') \geq 2$ and $H'$ is edge-critical).
\end{property}
\begin{proof}
Let $H \leq G$ be such that $\nu(H) \leq 6$. By Property \ref{prop-mindeg2} we have that the assertion is true if $\nu(H) \leq 3$, therefore we may assume that $4 \leq \nu(H) \leq 6$.

If $H$ is not edge-critical, then there is an $e \in E(H)$ such that $\alpha(H - e) = \alpha(H)$, whence $\nu(H - e) \leq 6 - 3 = 3$. Thus by Property \ref{prop-mindeg2} we have that $\delta(H-e) \geq 2$, whence also $\delta(H) \geq 2$, as desired.

If on the other hand $H$ is edge-critical, then we may assume that $\delta(H) \leq 1$ since otherwise the assertion is trivially satisfied. This means that $H$ contains a component $C \in \calC(H)$ such that $C \iso K_2$ by Lemma \ref{ecmindeg} and Property \ref{prop-mindeg1}.

Therefore $H = H' + K_2$ where we get
  $$6 \geq \nu(H) = \nu(H') + \nu(K_2) \Rightarrow 2 \geq \nu(H').$$
\end{proof}

\begin{property} \label{prop2}
If $H \leq G$ and $\nu(H) \leq 6$ then either $H$ is 2-stable or $H = H' + K_2$ where $V(K_2)$ is the only 2-destabiliser.
\end{property}
\begin{proof}
By Property \ref{prop-leq6} we have that $\delta(H) \geq 2$ or $H = H' + K_2$. If $\delta(H) \geq 2$ then any destabiliser, $S$, of size 2 would be such that $\nu(H\setminus S) \leq \nu(H) - 9 + 35 - 35 \leq -4$, contradicting assumption \eqref{bigass}.

If $H = H' + K_2$ then $\nu(H') \leq \nu(H) - 4 \leq 2$ and therefore $H'$ is 2-stable by Property \ref{prop1}. Any destabiliser $S$ must contain a destabiliser of one of the components of $H$. If $|S| = 2$ then $S$ must destabilise the $K_2$-component. Hence $S = V(K_2)$.
\end{proof}

\begin{property} \label{prop3}
If $H \leq G$, $\nu(H) \leq 4$ and $\delta(H) \geq 3$ then $H$ is 3-stable.
\end{property}
\begin{proof}
Otherwise let $S$ be a destabilising set of size 3. Since $H$ is triangle-free $e(S) \leq 2$ and therefore $e(H \setminus S) \leq e(H) - 7$. Hence $\nu(H \setminus S) \leq \nu(H) - 21 + 51 - 35 \leq -1$, contradicting assumption \eqref{bigass}.
\end{proof}

\begin{property} \label{prop4}
If $H \leq G$, $\delta(H) = \delta$ and $v \in V(H)$ is a vertex of valency $d$ then either $N_2(v)$ destabilises $H_v$ or $\nu(H) \geq 3 \delta d + 18d - 17$.
\end{property}
\begin{proof}
Suppose that $N_2(v)$ did not destabilise $H_v$, then there would be some independent set, $I$, of size $\alpha(H_v)$ in $H_v$ such that $I \cap N_2(v) = \emptyset$. Since $H$ is triangle-free we would get that $I \cup N(v)$ is an independent set of size $\alpha(H_v) + d(v)$ in $H$. Therefore $\nu(H_v) \leq \nu(H) + 17 - 18d - 3\delta d$ since $\alpha(H) \geq \alpha(H_v) + 3$, $n(H) = n(H_v) - (d + 1)$ and $e(H) \geq e(H_v) + \delta d$. The assertion then follows from the assumption \eqref{bigass} on $H_v$.
\end{proof}

\begin{property}
  \label{tretton}
  If $H \leq G$ is such that $\nu(H) \leq 2$ and $C_5 \notin \Cc(H)$ then for all bivalent vertices $v \in V(H)$ we have that
\begin{enumerate}[(i)]
 \item $5 \leq d^2(v) \leq 6$ with $d^2(v) = 5$ if $\nu(H) \leq 1$,
 \item $\delta(H_v) \geq 2$ with equality if $\nu(H) \leq 1$,
 \item $d^2(v) = 5 \Rightarrow \N(C_4;H, v) = 0$, and
 \item $d^2(v) = 6 \Rightarrow \N(C_4;H,N(v)) = 0$.
\end{enumerate}
\end{property}
\begin{proof}
Let $H$ and $v \in V(H)$ be as in the premises. By Property \ref{prop-mindeg2} we have that $\delta(H) \geq 2$ and therefore $4 \leq d^2(v)$. Note that $H$ contains no 2-regular component by Property \ref{prop-no2regular} and the assumption that $C_5 \notin \calC(H)$.

Suppose that $d^2(v) = 4$. Because $H$ contains no 2-regular component there is a vertex $v' \in V(H)$ such that $d(v') = 2$, $d^2(v') \geq 5$ with $w \in N(v')$ such that $d(w)=2$ and $d^2(w) = 4$ (take $v'$ to be an endpoint in the path-component of $H[\{v \in V(H): d(v) = 2\}]$ which contains $v$). Let $x$ be the vertex in $N(w) \setminus \{v'\}$, the situation is then as illustrated in Figure \ref{fig:tretton1}.

\begin{figure}[h]
\begin{center}
  \begin{tikzpicture}[scale=0.5]
  \GraphInit[vstyle=Classic]
  \renewcommand*{\VertexSmallMinSize}{4pt}
  \Vertex[L=$v'$,Lpos=180]{v}
  \Vertex[x=2,y=1,L=$w$,Lpos=90]{w}
  \Vertex[x=2,y=-1,NoLabel]{w2}
  \Vertex[x=3,y=1,L=$x$,Lpos=90]{x}
  \Vertex[x=4,y=1,empty]{y}
  \Vertex[x=3,y=-0.4,empty]{x2}
  \Vertex[x=3,y=-1.6,empty]{x3}
  \Edge(v)(w)
  \Edge(v)(w2)
  \Edge(w)(x)
  \Edge(w2)(x2)
  \Edge(w2)(x3)
  \Edge(x)(y)
\end{tikzpicture}
\end{center}
\caption{Local structure in the neighbourhood of $v'$.}
\label{fig:tretton1}
\end{figure}
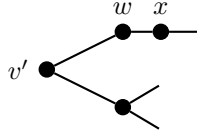

But then we would have that
$$\nu(H_{v'}) \leq \nu(H) + 1 \leq 3,$$
whence $\delta(H_{v'}) \geq 2$ by Property \ref{prop-mindeg2} which contradicts that $d(H_{v'}; x) = 1$. Hence $d^2(v) \geq 5$. Moreover, if $\nu(H) \leq 1$ then $d^2(v) \leq 5$ as well since otherwise we would get $\nu(H_v) \leq \nu(H) - 2 - \N(C_4; H, N(v)) \leq -1$, by Property \ref{prop-removalcount}, contradicting assumption \eqref{bigass}.

For the upper bound in (i) note that if $d^2(v) \geq 7$ then we would have $\nu(H_v) \leq \nu(H) - 5 \leq -3$ which contradicts assumption \eqref{bigass}.

The first assertion of (ii) follows by Property \ref{prop-mindeg2} since we have $\nu(H_v) \leq \nu(H) + 1 \leq 3$. The second assertion of (ii) since $\nu(H_v) \leq 1$ would contradict Property \ref{prop-mindeg2} because we have $\nu(H_v) \leq \nu(H) + 1 - \N(C_4;H,N(v)) \leq 2$ by Property \ref{prop-removalcount}.

For (iii), suppose that $d^2(v) = 5$. If $\N(C_4; H,v) \geq 1$ then $\N(C_4;H,w) \geq 1$ where $w$ is the bivalent neighbour of $v$. We must then have that $d^2(w) \geq 6$ by (ii). Hence, $\nu(H_{w}) \leq \nu(H) - 2 - 1 \leq -1$, which contradicts assumption \eqref{bigass}. Hence $\N(C_4; H,v) = 0$.

Finally, if $d^2(v) = 6$ then $\nu(H_v) \leq \nu(H) - 2 - \N(C_4; H,N(v)) \leq 0 - \N(C_4; H,N(v))$ whence $\N(C_4;H,N(v)) = 0$ or this would contradict assumption \eqref{bigass}. This proves (iv).
\end{proof}

The following corollary is immediate.

\begin{corollary}
\label{kor1}
Let $H \leq G$ be a connected graph such that $\nu(H) \leq 2$. Suppose that there is a bivalent vertex $v \in V(H)$ with second valency four. Then $H \iso C_5$.
\end{corollary}

\begin{property}
\label{ddestab}
Let $H \leq G$. If $\nu(H) \leq 2$ and $v \in V(H)$ is bivalent with second valency six, then $N_2(v)$ is a minimal destabiliser of $H_v$ of size four.
\end{property}
\begin{proof}
Note that $H$ is edge-critical by Property \ref{prop-edge-critical}.

$N_2(v)$ clearly has size at most four and if it would have size less than four then there would be a cycle of length four through $v$, contradicting Property \ref{tretton}.

It follows from Lemma \ref{ecdestab} that $N_2(v)$ is a destabiliser. If it were not a \emph{minimal} destabiliser then, $N_2(v) \setminus \{x\}$ would destabilise $H_v$ for some $x \in N_2(v)$. Let $H'$ be the graph obtained when we remove the edge between $x$ and $N(v)$ from $H$. By Lemma \ref{ecalpha} we then have that $\alpha(H') = \alpha(H) + 1$ and if $S$ is a maximum independent set in $H'$, then it must contain both $w_1$ and $x$.

Since $S \cap V(H'_v) = S \cap V(H_v)$ is an independent set we must have that $|S \cap V(H_v)| \leq \alpha(H_v) = \alpha(H) - 1$ and therefore $S$ contains both $w_1$ and $w_2$. But $|S \cap V(H_v)| \leq \alpha(H) - 1$ and $S \cap V(H_v) \cap (N_2(v) \setminus \{x\}) = \emptyset$, contradicting that $N_2(v) \setminus \{x\}$ is a destabiliser.
\end{proof}

\begin{property}
 \label{containsChk2}
 If $H \leq G$ is a connected graph which contains two adjacent bivalent vertices, $v_1$ and $v_2$, each of second valency five, then
\begin{enumerate}[(i)]
  \item $\nu(H) \leq 1 \Rightarrow H \iso Ch_k$ for some $k \geq 3$,
  \item $\nu(H) \leq 2 \Rightarrow \N(C_5;H, v_1) = \N(C_5;H,v_2) = 2$,
\end{enumerate}
whence, in particular, there is at least one cycle of length four through the trivalent neighbours of $v_1$ and $v_2$.
\end{property}
\begin{proof}
The assertion is trivially true for $H = K_1$. Suppose that it holds for all $J \leq G$ with fewer vertices than $H$. Moreover, let $H$ be connected with $\nu(H) \leq 2$ and $\delta(H) = 2$ but suppose that either (i) or (ii) does not hold for $H$.

We then have, a fortiori, $H \not \iso C_5 \iso Ch_2$. Therefore, by Property \ref{tretton}, we have that $\delta(H_{v_1}),\delta(H_{v_2}) \geq 2$ and that there is no cycle of length four through $v_1$ or $v_2$. Denote the trivalent neighbour of $v_i$ by $u_i$ and the two neighbours of $u_i$ that are not $v_i$ by $w_{i1}$ and $w_{i2}$ for $i \in \{1,2\}$. We then have that
\begin{equation}
 \label{wijeq3s}
d(w_{ij}) \geq 3 \quad (\forall i,j \in \{1,2\}).
\end{equation}
The neighbourhood of $v_1$ and $v_2$ in this situation has been illustrated in Figure \ref{fig:a1}.

\begin{figure}[h]
\begin{center}
\begin{tikzpicture}[scale=1]
  \GraphInit[vstyle=Classic]
  \renewcommand*{\VertexSmallMinSize}{4pt}
  \Vertex[L=$u_1$,Lpos=-90]{u1}
  \Vertex[x=1,y=0,L=$v_1$,Lpos=-90]{v1}
  \Vertex[x=2,y=0,L=$v_2$,Lpos=-90]{v2}
  \Vertex[x=3,y=0,L=$u_2$,Lpos=-90]{u2}
  \Vertex[x=-0.7,y=0.7,L=$w_{11}$,Lpos=90]{w11}
  \Vertex[x=-0.7,y=-0.7,L=$w_{12}$,Lpos=-90]{w12}
  \Vertex[x=3.7,y=0.7,L=$w_{21}$,Lpos=90]{w21}
  \Vertex[x=3.7,y=-0.7,L=$w_{22}$,Lpos=-90]{w22}
  \Vertex[x=-1.2,y=0.5,empty]{w111}
  \Vertex[x=-1.2,y=0.7,empty]{w112}
  \Vertex[x=-1.2,y=0.9,empty]{w113}
  \Vertex[x=-1.2,y=-0.5,empty]{w121}
  \Vertex[x=-1.2,y=-0.7,empty]{w122}
  \Vertex[x=-1.2,y=-0.9,empty]{w123}
  \Vertex[x=4.2,y=0.5,empty]{w211}
  \Vertex[x=4.2,y=0.7,empty]{w212}
  \Vertex[x=4.2,y=0.9,empty]{w213}
  \Vertex[x=4.2,y=-0.5,empty]{w221}
  \Vertex[x=4.2,y=-0.7,empty]{w222}
  \Vertex[x=4.2,y=-0.9,empty]{w223}
  \Edge(u1)(w11)
  \Edge(u1)(w12)
  \Edge(u2)(w21)
  \Edge(u2)(w22)
  \Edge(u1)(v1)
  \Edge(v1)(v2)
  \Edge(v2)(u2)
  \Edge(w11)(w111)
  \Edge(w11)(w112)
  \Edge[style=dotted](w11)(w113)
  \Edge(w12)(w121)
  \Edge(w12)(w122)
  \Edge[style=dotted](w12)(w123)
  \Edge(w21)(w211)
  \Edge(w21)(w212)
  \Edge[style=dotted](w21)(w213)
  \Edge(w22)(w221)
  \Edge(w22)(w222)
  \Edge[style=dotted](w22)(w223)
\end{tikzpicture}
\end{center}
\caption{The neighbourhood of $v_1$ and $v_2$ in $H$.}
\label{fig:a1}
\end{figure}
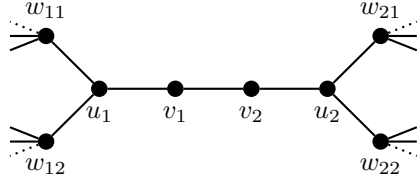

Note however that the vertices $w_{ij}$ are not, a priori, distinct. The vertices $w_{i1}$ and $w_{i2}$ are however distinct.

Note also that for any subgraph $J < H$ such that $\nu(J) \leq 1$ and $\delta(J) = 2$ every bivalent vertex in $J$ must have second valency five by Property \ref{tretton} or belong to a $Ch_2 \iso C_5$-component of $J$. Hence, inductively, any bivalent vertex in $J$ belongs to a $Ch_k$-component for some $k \geq 2$. This fact will be used repeatedly throughout this proof.

\emph{Case 1}: $N(u_1) \cap N(u_2) = \emptyset$. Then all four vertices $w_{ij}$ are distinct and there is no cycle of length five through the $v_1$ and $v_2$. Suppose that $w_{11}$ had a bivalent neighbour $y$, then $\nu(H_{u_1,v_2}) \leq \nu(H) + 9 - 10 \leq 1$ and $d(H_{u_1,v_2};y) = 1$, contradicting Property \ref{prop-mindeg2}. Hence, 
\begin{equation}
\label{wijnobivalnbrs}
w_{ij} \text{ has no bivalent neighbours } \quad (\forall i,j \in \{1,2\}).
\end{equation}

If $d(w_{ij}) \geq 4$ for some $i,j \in \{1,2\}$ then without loss of generality assume that $d(w_{11}) \geq 4$. Hence $\nu(H_{v_2}) \leq \nu(H) + 1 \leq 3$, but $d^2(H_{v_2}; u_1) \geq 7$ which would give $\nu(H_{v_2,u_2}) \leq \nu(H_{v_2}) - 5 \leq -2$, contradicting assumption \eqref{bigass}. Therefore $d(w_{ij}) = 3$ for all $i,j \in \{1,2\}$.

If $d^2(w_{11}) \geq 10$ then we get $\nu(H_{w_{11}}) \leq \nu(H) + 3 \leq 5$ and thus $v_2$ would be monovalent in $H_{w_{11}}$ by $d(H_{w_{11}}; v_1) = 1$ and Property \ref{prop-leq6}. It would then follow that $w_{11}$ is adjacent to $u_2$ contradicting our assumption that $N(u_1) \cap N(u_2) = \emptyset$. Hence, by (\ref{wijnobivalnbrs}) and an analogous argument for $w_{ij} \neq w_{11}$, we get
\begin{equation*}
d^2(w_{ij}) = 9 \quad (\forall i,j \in \{1,2\})
\end{equation*}

This means that $\nu(H_{v_1,w_{11}}) \leq \nu(H) + 1 - 2 \leq 1$ and $d(H_{v_1,w_{11}}; u_2) = 2$. Note that $d^2(H_{v_1}; w_{11}) = 6$ since the $w_{ij}$ are distinct. By Lemma \ref{ecbvconn} we get that $H_{v_1}$ is connected. It is however possible that $H_{v_1,w_{11}}$ is not since $H_{v_1}$ might not be edge-critical.

If $H_{v_1,w_{11}}$ is connected, then since $u_2$ is bivalent in $H_{v_1,w_{11}}$ we must have, by the inductive hypothesis, that $H_{v_1,w_{11}} \iso Ch_k$ for some $k \geq 2$. Also, $N_2(H_{v_1}; w_{11})$ destabilises $H_{v_1,w_{11}}$ and thus, by Lemma \ref{lemma:starone}, either $N_2(H_{v_1};w_{11})$ contains a pair of adjacent bivalent vertices of $H_{v_1,w_{11}}$ or $\N(C_4; G,N(v)) \geq 2$. The latter is not possible however since then we would get $\nu(H_{v_1,w_{11}}) \leq \nu(H) + 1 - 2 - (2) \leq -1$, contradicting assumption \eqref{bigass}. Then since $w_{12}$ and $u_2$ are bivalent in $H_{v_1}$, and not adjacent to $w_{11}$, we would have to have $w_{12}u_2 \in E(H)$, contradicting $N(u_1) \cap N(u_2) = \emptyset$.

On the other hand, if $H_{v_1,w_{11}}$ is disconnected, then every component is 2-stable by Property \ref{prop1}. Therefore $e(C,H\setminus C) \geq 3$ for all $C \in \calC(H_{v_1,w_{11}})$, but $e(N[v_1,w_{11}], H\setminus N[v_1,w_{11}]) = 6$. Hence there are exactly two components, $C_1$ and $C_2$, in $H_{v_1,w_{11}}$ each connected to the rest of the graph by three edges. Moreover $N_2(H_{v_1; w_{21}})$ destabilises one of the components, say $C_1$, since otherwise $\alpha(H_{v_1}) \geq \alpha(H_{v_1,w_{11}}) + 2$ by Lemma \ref{lemma:alpha} and then $\nu(H_{v_1,w_{11}}) \leq \nu(H_{v_1}) - 35 \leq - 32$, contradicting assumption \eqref{bigass}.

By induction, Property \ref{prop3} and Corollary \ref{corollary:A} we get that both $C_1$ and $C_2$ are $Ch_k$-components. Since $e(N(H_{v_1}; w_{11}), V(C_1)) = 3$ we get by Lemma \ref{Chkdestab3} that $N_2(H_{v_1}; w_{11}) \cap N(C_1) = N_{C_1}[x]$ for some bivalent vertex $x$ in $C_1$. Thus by the recursive construction of $Ch_k$-graphs we get that $H' := H[V(C_1) \cup N_{H_{v_1}}[w_{11}]] \iso Ch_{k+1}$ where $k \geq 2$ is such that $C_1 \iso Ch_{k}$. But $e(H',H\setminus H') = 2$, since $e(N(v_1),V(C_2)) = 2$, contradicting Corollary \ref{corollary:A}.

\emph{Case 2:} $|N(u_1) \cap N(u_2)| = 1$. Without loss of generality assume that $w := w_{11} = w_{21}$. The situation is then as illustrated in Figure \ref{fig:a2}.

\begin{figure}[h]
\begin{center}
\begin{tikzpicture}[scale=1]
  \GraphInit[vstyle=Classic]
  \renewcommand*{\VertexSmallMinSize}{4pt}
  \Vertex[L=$u_1$,Lpos=-90]{u1}
  \Vertex[x=1,y=0,L=$v_1$,Lpos=-90]{v1}
  \Vertex[x=2,y=0,L=$v_2$,Lpos=-90]{v2}
  \Vertex[x=3,y=0,L=$u_2$,Lpos=-90]{u2}
  \Vertex[x=1.5,y=1,L=$w$,Lpos=90]{w}
  \Vertex[x=-0.7,y=-0.7,L=$w_{12}$,Lpos=-90]{w12}
  \Vertex[x=3.7,y=-0.7,L=$w_{22}$,Lpos=-90]{w22}
  \Vertex[x=1.2,y=1.3,empty]{w-1}
  \Vertex[x=1.8,y=1.3,empty]{w-2}
  \Vertex[x=-1.2,y=-0.5,empty]{w121}
  \Vertex[x=-1.2,y=-0.7,empty]{w122}
  \Vertex[x=-1.2,y=-0.9,empty]{w123}
  \Vertex[x=4.2,y=-0.5,empty]{w221}
  \Vertex[x=4.2,y=-0.7,empty]{w222}
  \Vertex[x=4.2,y=-0.9,empty]{w223}
  \Edge(u1)(w)
  \Edge(u1)(w12)
  \Edge(u2)(w)
  \Edge(u2)(w22)
  \Edge(u1)(v1)
  \Edge(v1)(v2)
  \Edge(v2)(u2)
  \Edge(w12)(w121)
  \Edge(w12)(w122)
  \Edge[style=dotted](w12)(w123)
  \Edge(w22)(w221)
  \Edge(w22)(w222)
  \Edge[style=dotted](w22)(w223)
  \Edge(w)(w-1)
  \Edge[style=dotted](w)(w-2)
\end{tikzpicture}
\end{center}
\caption{The neighbourhood of $v_1$ and $v_2$ in $H$.}
\label{fig:a2}
\end{figure}
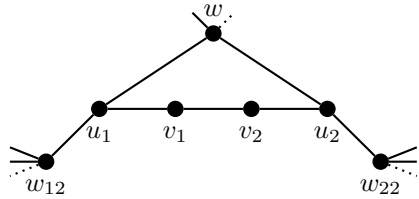

If $w$ had a bivalent neighbour, $y$, then since $\nu(H_{u_1,v_2}) \leq \nu(H) + 9 - 7 \leq 4$ we get that $y$ has a monovalent neighbour $y'$ in $H_{u_1,v_2}$ by Property \ref{prop-leq6}. But $y'$ is not adjacent to any vertex in $N(u_1) \cup N(v_2) \setminus \{w_{12}\}$, whence $y' w_{12} \in E(H)$. Then we would have that $\nu(H_{u_2,v_1,y}) \leq \nu(H) + 9 - 7 - 7 \leq -3$ since $y'$ has valency at least two in $H_{u_2,v_1}$. However, this contradicts assumption \eqref{bigass}, whence $w$ has no bivalent neighbours.

If $w_{12}$ had a bivalent neighbour, $y$, then $\nu(H_{u_1,v_2}) \leq \nu(H) + 9 - 7 \leq 4$ and therefore $y$ would have a monovalent neighbour, $y'$ in $H_{u_1,v_2}$ (by Property \ref{prop-leq6}). But $y'$ would then be adjacent to $w$ and no other vertex in $N(u_1) \cup N(v_2)$, contradicting that $w$ has no bivalent neighbour. By an completely analogous argument for $w_{22}$ we can now say that
\begin{equation}
  \label{wnobivalnbrs}
  w_{ij} \text{ has no bivalent neighbours. } \quad (\forall i,j \in \{1,2\})
\end{equation}

\emph{Subcase 2.1:} $w_{12} w_{22} \in E(H)$. If $d(w_{12}) = 3$ then $d^2(w_{12}) \geq 9$ by (\ref{wnobivalnbrs}) and therefore $\nu(H_{w_{12},v_1}) \leq \nu(H) + 6 - 7 \leq 1$ but $u_2$ is monovalent in $H_{w_{12},v_1}$, contradicting Property \ref{prop-mindeg2}. Hence, $d(w_{12}) \geq 4$ and analogously we get that $d(w_{22}) \geq 4$. If $d(w_{12}) \geq 5$ then we would get $\nu(H_{u_1,v_2}) \leq \nu(H) + 3 - 7 \leq -3$, contradicting the assumption \eqref{bigass}. By an analogous argument again for $w_{22}$ we can now conclude that $d(w_{12}) = d(w_{22}) = 4$, i.e. we have the situation as illustrated in Figure \ref{fig:a3}.

\begin{figure}
\begin{center}
\begin{tikzpicture}[scale=1]
  \GraphInit[vstyle=Classic]
  \renewcommand*{\VertexSmallMinSize}{4pt}
  \Vertex[L=$u_1$,Lpos=-90]{u1}
  \Vertex[x=1,y=0,L=$v_1$,Lpos=-90]{v1}
  \Vertex[x=2,y=0,L=$v_2$,Lpos=-90]{v2}
  \Vertex[x=3,y=0,L=$u_2$,Lpos=-90]{u2}
  \Vertex[x=1.5,y=1,L=$w$,Lpos=0]{w}
  \Vertex[x=-0.7,y=-0.7,L=$w_{12}$,Lpos=-90]{w12}
  \Vertex[x=3.7,y=-0.7,L=$w_{22}$,Lpos=-90]{w22}
  \Vertex[x=1.5,y=1.3,empty]{w-1}
  \Vertex[x=-1.2,y=-0.5,empty]{w121}
  \Vertex[x=-1.2,y=-0.7,empty]{w122}
  \Vertex[x=-1.2,y=-0.9,empty]{w123}
  \Vertex[x=4.2,y=-0.5,empty]{w221}
  \Vertex[x=4.2,y=-0.7,empty]{w222}
  \Vertex[x=4.2,y=-0.9,empty]{w223}
  \Edge(u1)(w)
  \Edge(u1)(w12)
  \Edge(u2)(w)
  \Edge(u2)(w22)
  \Edge(u1)(v1)
  \Edge(v1)(v2)
  \Edge(v2)(u2)
  \Edge(w12)(w121)
  \Edge(w12)(w123)
  \Edge(w22)(w221)
  \Edge(w22)(w223)
  \Edge(w)(w-1)
  \Edge(w12)(w22)
\end{tikzpicture}
\end{center}
\caption{The neighbourhood of $v_1$ and $v_2$ in $H$.}
\label{fig:a3}
\end{figure}
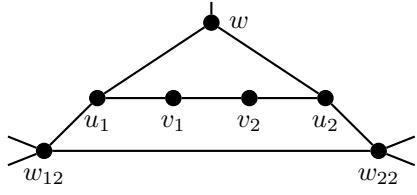

If $d(w) \geq 4$ then $\nu(H_{u_1,v_2}) \leq \nu(H) + 3 - 7 \leq -2$, contradicting assumption \eqref{bigass} . Hence, $d(w) = 3$. If $N(w_{12}) \cap N(w) = \{u_1\}$ then $\nu(H_{w_{12},v_1,u_2}) \leq \nu(H) + 11 - 7 - 7 \leq -1$, contradicting assumption \eqref{bigass}. Thus, by the previous and an analogous argument for $w_{22}$ we must have that
$$N(w_{12}) \cap N(w) \neq \{u_1\} \text{ and } N(w_{22}) \cap N(w) \neq \{u_2\}.$$
But since $d(w) = 3$ we then must have that
$N(w_{22}) \cap N(w) \setminus \{u_2\} = N(w_{12}) \cap N(w) \setminus \{u_1\}$. Let $a$ be an element of $N(w_{22}) \cap N(w) \setminus \{u_2\}$, then we would have a triangle through $w_{12},a$ and $w_{22}$, contradicting the assumption that $G$ is triangle-free.

\emph{Subcase 2.2:} $w_{12} w_{22} \notin E(H)$. Since $\nu(H_{v_i}) \leq \nu(H) + 1 \leq 3$ we must have that $d^2(H_{v_i}; u_{3-i}) \leq 7$ for $i \in \{1,2\}$. Otherwise we would get $\nu(H_{v_i,u_{3-i}}) \leq \nu(H) + 1 - 5 \leq -2$, contradicting assumption \eqref{bigass}.

Suppose that $d(w) \geq 4$, then $d(H_{v_i}; w_{3-i,2}) \leq 3$ for $i \in \{1,2\}$. But then $d(H_{v_i}; w_{3-i,2}) = 3$ by (\ref{wijeq3s}) and $d(H_{v_i}; w_{3-i,2}) = d(w_{3-i,2})$. Now note that $\nu(H_{v_2,u_1}) \leq \nu(H) + 1 - 2 \leq 1$ and $w_{22}$ is bivalent in $H_{v_2,u_1}$. Hence $w_{22} \in V(C)$ for some $C \in \calC(H_{v_2,u_1})$ such that $C \iso Ch_k$, where $k$ is at least two. Let $N := N[u_1,v_2]$. If $H_{v_2,u_1}$ contained more than one component, say some $C' \neq C$, $C' \in \calC(H_{v_2,u_1})$, then $C'$ would have at most two edges to $N$ in $H$. But $C'$ has to be 2-stable by Property \ref{prop1}, so we would have redundant edges in $E(N,H\setminus N)$ by Corollary \ref{corollary:A}.

Hence $H_{v_2,u_1} \iso Ch_k$ and $N_2(H_{v_2}; u_1)$ destabilises $H_{v_2,u_1}$, by Property \ref{prop4}.
Let $T := N_2(H_{v_2}; u_1)$. If $|T| = 4$ and $T$ is a minimal destabiliser, then either $T$ is connected in $H_{v_2,u_1}$ or $k = 3$ and $T = V_2(H_{v_2,u_1})$ by Lemma \ref{Chkdestab4}. The latter is however not possible since $w_{22}$ is bivalent in $H_{v_2}$. The former is not possible either since then $T$ would induce one of the three connected triangle-free graphs of size four ($P_4$, $K_{1,3}$ or $C_4$) in $H_{v_2,u_1}$. It is easy to see that in all three cases $\N(C_4; H_{v_2}, N(u_1)) \geq 2$ and thus $\nu(H_{v_2,u_1}) \leq \nu(H) + 1 - 2 - (2) \leq -1$, contradicting assumption \eqref{bigass}. Hence we get by Lemmas \ref{Chkdestab3} and \ref{Chkdestab4} that $N_C[x] \subseteq T$ for some $x \in V_2(H_{v_2,u_1})$.
Therefore by the recursive construction of $Ch_k$ we must have that $H_{v_2} \iso Ch_{k+1} + \{\alpha, \beta\}$ for some $\{\alpha,\beta\} \in \binom{V(Ch_{k+1})}{2} \setminus E(Ch_{k+1})$. The extra edge $\{\alpha,\beta\}$ is incident to either $w$ or $w_{12}$. In either case $\{\alpha,\beta\}$ is incident to at least one bivalent vertex in $Ch_{k+1}$, since both $w$ and $w_{12}$ are trivalent in $H_{v_2}$. But then $\{\alpha,\beta\}$ is also incident to the bivalent vertex in $V_2(Ch_{k+1}) \setminus \{u_2,w,w_{12},w_{22}\}$, since otherwise that vertex would be bivalent also in $H$, but have second valency 6, and a cycle of length four through its neighbourhood. Hence $\{\alpha,\beta\} \subseteq V_2(Ch_{k+1})$. Now, note that $S := \{u_1,w,w_{22}\} \subseteq V(H_{v_2})$ is such that $|S| = 3$ and $|S \cap V_2(Ch_{k+1})| \geq 2$ since both $u_1$ and $w_{22}$ are bivalent in $H_{v_2}$ (and therefore also in $H_{v_2,w_{11}}$). Furthermore, if $|S \cap V_2(Ch_{k+1})| = 2$ then $S \cap V_2(Ch_{k+1}) = \{u_1,w_{22}\}$ and $u_1$ is not adjacent to $w_{22}$, by assumption. By Property \ref{lemma:Chkplusedge} therefore $H_{v_2}$ is not destabilised by $S = \{u_1,w,w_{22}\} = N_2(v_2)$, contradicting Lemma \ref{ecdestab}.

Hence $d(w) = 3$. Let $y$ denote the neighbour of $w$ that is neither $u_1$ nor $u_2$. This situation is illustrated in Figure \ref{fig:a4}.

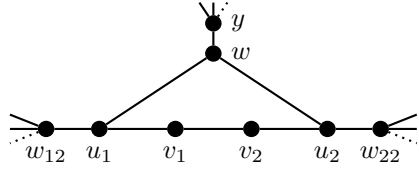
\begin{figure}[h]
\begin{center}
\begin{tikzpicture}[scale=1]
  \GraphInit[vstyle=Classic]
  \renewcommand*{\VertexSmallMinSize}{4pt}
  \Vertex[L=$u_1$,Lpos=-90]{u1}
  \Vertex[x=1,y=0,L=$v_1$,Lpos=-90]{v1}
  \Vertex[x=2,y=0,L=$v_2$,Lpos=-90]{v2}
  \Vertex[x=3,y=0,L=$u_2$,Lpos=-90]{u2}
  \Vertex[x=1.5,y=1,L=$w$,Lpos=0]{w}
  \Vertex[x=-0.7,y=0,L=$w_{12}$,Lpos=-90]{w12}
  \Vertex[x=3.7,y=0,L=$w_{22}$,Lpos=-90]{w22}
  \Vertex[Math,x=1.5,y=1.4]{y}
  \Vertex[x=-1.2,y=-0.2,empty]{w121}
  \Vertex[x=-1.2,y=0,empty]{w122}
  \Vertex[x=-1.2,y=0.2,empty]{w123}
  \Vertex[x=4.2,y=-0.2,empty]{w221}
  \Vertex[x=4.2,y=0,empty]{w222}
  \Vertex[x=4.2,y=0.2,empty]{w223}
  \Vertex[x=1.5,y=1.7,empty]{y1}
  \Vertex[x=1.3,y=1.7,empty]{y2}
  \Vertex[x=1.7,y=1.7,empty]{y3}
  \Edge(u1)(w)
  \Edge(u1)(w12)
  \Edge(u2)(w)
  \Edge(u2)(w22)
  \Edge(u1)(v1)
  \Edge(v1)(v2)
  \Edge(v2)(u2)
  \Edge[style=dotted](w12)(w121)
  \Edge(w12)(w123)
  \Edge[style=dotted](w22)(w221)
  \Edge(w22)(w223)
  \Edge(w)(y)
  \Edge(w12)(w122)
  \Edge(w22)(w222)
  \Edge(y)(y1)
  \Edge(y)(y2)
  \Edge[style=dotted](y)(y3)
\end{tikzpicture}
\end{center}
\caption{The neighbourhood of $v_1$ and $v_2$ in $H$ in Subcase 2.2.}
\label{fig:a4}
\end{figure}

\emph{Subcase 2.2.1:} $d(w_{12}) = 3$ and $d(w_{22}) \geq 4$ (or analogously $d(w_{12}) \geq 4$  and $d(w_{22}) = 3$). We have $\nu(H_{v_1,u_2}) \leq \nu(H) + 1 - 2 \leq 1$ and thus $d(w_{22}) = 4$ (otherwise $\nu(H_{v_1,u_2}) \leq -2$, contradicting assumption \eqref{bigass}). Inductively $w$ is a vertex of some $C \in \calC(H_{v_1,u_2})$ such that $C \iso Ch_k$ for some $k \geq 2$. Moreover, $H_{v_1,u_2}$ is connected since every component of $H_{v_1,u_2}$ is 2-stable and otherwise $e(C,H\setminus C) \leq 2$, contradicting Corollary \ref{corollary:A}.

Since $w_{22}$ has three neighbours in $C$ we must have that $k \geq 3$, since $H$ is triangle-free. Note that $N_2(H_{v_1}; u_2)$ destabilises $C$, but $|N_2(H_{v_1}; u_2)| \leq 4$ and is not adjacent to the, in $H_{v_1,u_2}$, bivalent $w_{12}$. Now, let $T := N_2(H_{v_1};u_2)$. If $|T| = 4$ then either $T$ is not minimal or $T$ is connected in $H_{v_1,u_2}$, by Lemma \ref{Chkdestab4}. But in the latter case $H_{v_1,u_2}[T] \in \{P_4,K_{1,3},C_4\}$ so $\N(C_4; H_{v_1}, N(u_2)) \geq 2$. However, then $\nu(H_{v_1,u_2}) \leq \nu(H) + 1 - 2 - (2) \leq -1$, contradicting assumption \eqref{bigass}.

Hence $N_2(H_{v_1}) \supseteq N_C[x]$, $x \in V_2(C)$ and also $N_2(H_{v_1}; u_2)$ contains a, in $C$, bivalent neighbour of $w_{12}$, by Lemmas \ref{Chkdestab3} and \ref{Chkdestab4}. Therefore $y$ is trivalent in $H$ with at least two cycles of length four through its neighbourhood, and second valency at least 10, which would give
$$\nu(H_{y,w_{22},v_2}) \leq \nu(H) + 3 - (2) + 1 - 7 \leq -3,$$
contradicting assumption \eqref{bigass}.

\emph{Subcase 2.2.2:} $d(w_{12}) = d(w_{22}) = 3$. By (\ref{wnobivalnbrs}) we have that $d(w_{22}) \geq 9$ and if $d^2(w_{22}) \geq 10$ we would get $\nu(H_{w_{22},v_2}) \leq \nu(H) + 3 - 7 \leq -2$, contradicting assumption \eqref{bigass}. Hence $d^2(w_{22}) = 9$ and $\nu(H_{w_{22},v_2}) \leq 1$. Similarly, we also have that $\nu(H_{w_{12},v_1}) \leq 1$. Since $w$ and $u_2$ are bivalent neighbours in $H_{w_{12},v_1}$ they, inductively, belong to some $Ch_k$-component, $C \in \calC(H_{w_{12},v_1})$ for $k \geq 2$ and also $w,u_1$ are adjacent bivalent vertices in a component $C'$ of $H_{w_{22},v_2}$ such that $C' \iso Ch_{\ell}$ for some $\ell \geq 2$. Let $\{z\} = N_C(w_{22}) \setminus \{u_2\}$. Note that $\{y\} = N_C(w)\setminus \{u_2\} = N_{C'}(w) \setminus \{u_w\}$. Furthermore, let $\{x,x'\} := N(w_{12}) \setminus \{u_1\}$. This situation is illustrated in Figure \ref{fig:a5}.

\begin{figure}[h]
\begin{center}
\begin{tikzpicture}[scale=1]
  \GraphInit[vstyle=Classic]
  \renewcommand*{\VertexSmallMinSize}{4pt}
  \Vertex[L=$u_1$,Lpos=-90]{u1}
  \Vertex[x=1,y=0,L=$v_1$,Lpos=-90]{v1}
  \Vertex[x=2,y=0,L=$v_2$,Lpos=-90]{v2}
  \Vertex[x=3,y=0,L=$u_2$,Lpos=-90]{u2}
  \Vertex[x=1.5,y=1,L=$w$,Lpos=0]{w}
  \Vertex[x=-0.7,y=0,L=$w_{12}$,Lpos=-90]{w12}
  \Vertex[x=3.7,y=0,L=$w_{22}$,Lpos=-90]{w22}
  \Vertex[Math,x=2.3,y=1.7,Lpos=180]{y}
  \Vertex[Math,x=3.2,y=1]{z}
  \Vertex[x=3.5,y=1.3,empty]{z1}
  \Vertex[Math,x=-1.2,y=-0.4,Lpos=-90]{x}
  \Vertex[x=-1.2,y=0,empty]{w122}
  \Vertex[Math,x=-1.2,y=0.4,Lpos=90]{x'}
  \Vertex[x=4.2,y=-0.2,empty]{w221}
  \Vertex[x=4.2,y=0,empty]{w222}
  \Vertex[x=4.2,y=0.2,empty]{w223}
  \Vertex[x=-1.5,y=-0.6,empty]{x1}
  \Vertex[x=-1.5,y=-0.2,empty]{x2}
  \Vertex[x=-1.5,y=0.6,empty]{xp1}
  \Vertex[x=-1.5,y=0.2,empty]{xp2}
  \Vertex[x=4.2,y=0,empty]{w222}
  \Vertex[x=2.8,y=2,empty]{y2}
  \Vertex[x=2.8,y=1.7,empty]{y3}
  \Edge(y)(z)
  \Edge(w22)(z)
  \Edge(u1)(w)
  \Edge(u1)(w12)
  \Edge(u2)(w)
  \Edge(u2)(w22)
  \Edge(u1)(v1)
  \Edge(v1)(v2)
  \Edge(v2)(u2)
  \Edge(w12)(x)
  \Edge(w12)(x')
  \Edge(w)(y)
  \Edge(w22)(w222)
  \Edge(y)(y2)
  \Edge[style=dotted](y)(y3)
  \Edge[style=dotted](z)(z1)
  \Edge(x)(x1)
  \Edge(x)(x2)
  \Edge(x')(xp1)
  \Edge(x')(xp2)
  \node at (4.5,1){$\dots$};
\end{tikzpicture}
\end{center}
\caption{The neighbourhood of $v_1$ and $v_2$ in $H$ in subcase 2.2.2.}
\label{fig:a5}
\end{figure}
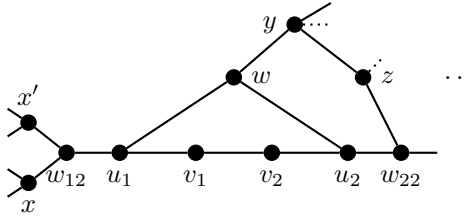

Suppose $k = 2$. By (\ref{wnobivalnbrs}) both $y$ and $z$ are at least trivalent in $H$. Therefore $y$, $z$ and $w_{22}$ are all adjacent to either $x$ or $x'$. Since $H$ is triangle-free this means that $y$ and $w_{22}$ are adjacent to the same neighbour of $w_{12}$, say $x$, and $z$ is adjacent to $x'$. But in this case $x,w_{22},z$ and $y$ forms a cycle of length four and we get $\nu(H_{w_{22},y}) \leq \nu(H) + 6 - (1) - 7 \leq 0$. However, $v_2$ is monovalent in $H_{w_{22},y}$, contradicting Property \ref{prop-mindeg2}.

Hence $k \geq 3$ and by an analogous argument to the one for $k = 2$ we may assume that $\ell \geq 3$. In particular $w_{12}$ and $w_{22}$ both have two common neighbours with $y$, since they do in $C'$ and $C$, respectively. These vertices are clearly distinct. There are also at least two cycles of length four through $y$, whence $\nu(H_{w,v_1}) \leq \nu(H) + 0 - (2) - 4 \leq - 4$, contradicting assumption \eqref{bigass}.

\emph{Subcase 2.2.3:} $d(w_{12}),d(w_{22}) \geq 4$. Then $d(w_{12}) = d(w_{22}) = 4$ since otherwise $\nu(H_{u_i,v_{3-i}}) \leq \nu(H) + 3 - 7 \leq -3$ for some $i \in \{1,2\}$, which would contradict assumption \eqref{bigass}. We then have $\nu(H_{u_i,v_{3-i}}) \leq \nu(H) + 6 - 7 \leq 1$ for both $i \in \{1,2\}$. If $u$ had a bivalent neighbour, $z$, then $z$ does not have a common neighbour with $u_1$ or $u_2$ (since if it did, then it would have second valency at least seven). Thus $\nu(H_{z,w,v_1}) \leq \nu(H) + 1 - 2 - 4 \leq -4$, contradicting assumption \eqref{bigass}. Hence $y$ has no bivalent neighbour in $H$.

Let $H' := H[N[v_1,u_2]]$, then $e(H', H\setminus H') = 5$ and therefore $H_{u_2,v_1}$ has at most one component. Otherwise there would be some component $C'$ of $H_{u_2,v_1}$ such that $e(C', H\setminus C') \leq 2$, contradicting Corollary \ref{corollary:A} since $H$ is edge-critical and $C'$ is 2-stable by Property \ref{prop1}. Analogously, $H_{v_2,u_1}$ is connected.

Suppose that $d(y) = 3$. Then $\nu(H_{u_2,v_1}) \leq 1$ and $d(H_{u_2,v_1}; y) = 2$. Hence $y$, inductively, belongs to some $Ch_k$-component, $C$, of $H_{u_2,v_1}$ for some $k \geq 2$. By the argument above $C = H_{u_2,v_1}$. Since $d(H_{u_2,v_1}; w_{12}) = 3$ we obtain that $N_2(v_1) \cup N_2(u_2)$ does not contain all bivalent vertices of $C$ (because then two adjacent ones would have to be in $N(w_{22})$, contradicting that $H$ is triangle-free). Clearly $k \geq 3$ since $w_{22}$ has three neighbours in $C$. Also, $T := N_2(H_{v_1}; u_2)$ destabilises $C$ by Property \ref{prop4}. Now, $|T| \leq 4$ and if $|T| = 4$ with $T$ connected in $H_{v_1,u_2}$ then $H_{v_1,u_2}[T]$ must be $P_3,K_{1,3}$ or $C_4$ and therefore $\N(C_4; H_{v_1}, N(u_2)) \geq 2$. However, then $\nu(H_{v_1,u_2}) \leq \nu(H) + 1 - 2 - (2) \leq -1$, contradicting the assumption \eqref{bigass}. Therefore we get by Lemmas \ref{Chkdestab3} and \ref{Chkdestab4} that $N_C[x] \subseteq N_2(H_{v_1}; u_2)$ for some, in $C$, bivalent vertex $x$.

Note that $w_{12}$ is not adjacent to $y$ since if $w_{12}y \in E(H)$ then $\nu(H_{v_2,u_1}) \leq \nu(H) + 1 - 2 - (1) \leq 0$ but $y$ would be at most monovalent in $H_{v_2,u_1}$, contradicting Property \ref{prop-mindeg2}. Analogously we can show that $w_{22}u \notin E(H)$. Since $y$ has no bivalent neighbour in $H$ but has a bivalent neighbour in $H_{v_1,u_2}$ we must have $x = y$. Hence $w_{22}$ has two common neighbours with $y$. Analogously one gets that $w_{12}$ has two common neighbours with $y$ also in $C$ since $w_{12}$ is not adjacent to $w$. This is not possible since there no cycles of length four through bivalent vertices in $C \iso Ch_k$.

Hence $d(y) \geq 4$. If $yw_{12}, yw_{22} \in E(H)$ then $\nu(H_{v_2,u_1}) \leq \nu(H) + 1 - (1) - 2 - (1) \leq -1$, contradicting assumption \eqref{bigass}. Hence at most one of $w_{12}$ and $w_{22}$ is adjacent to $y$. Suppose that $yw_{12} \in E(H)$, then $yw_{22} \notin E(H)$ and therefore $\nu(H_{v_2,u_1}) \leq + 1 - (1) - 2 \leq 0$. Moreover $d(H_{v_2,u_1}; y) = 2$, $d(H_{v_2,u_2}; w_{22}) = 3$ and $H_{v_2,u_1}$ is connected. Hence, inductively, $H_{v_2,u_1} \iso Ch_k$ for some $k \geq 3$. In particular, $y$ has a bivalent neighbour $z$ in $H_{v_2,u_1}$. If $d(z) \geq 3$ then $z$ would have to be adjacent to $w_{12}$, but that would give a triangle through $w_{12},z$ and $y$. Hence $d(z) = 2$ and $d^2(z) \geq 4+ 3$, since $z$ has a trivalent neighbour in $H_{u_1,v_2}$ because $k \geq 3$. Thus $y w_{12} \notin E(H)$ and analogously $yw_{22} \notin E(H)$.

\emph{Subcase 2.2.3.1:} $y$ has a trivalent neighbour except for $w$. Let $z$ be such a trivalent neighbour. Since $d^2(w) = 10$ we get that $\nu(H_{w,v_1}) \leq \nu(H) + 3 - 4 \leq 1$ and $z$ is bivalent in $H_{w,v_1}$. So, inductively, $z$ belongs to some $Ch_k$-component, $C$, for $k \geq 2$ of $H_{w,v_1}$. Suppose that $C \iso Ch_2$. If both neighbours, $x_1$ and $x_2$, of $z$ in $C$ were bivalent in $H$, then we would have that $\delta(H_{x_i}) = 1$, contradicting Property \ref{tretton}. Hence, $z$ has a neighbour which has lower valency in $C$ than in $H$. If $C \iso Ch_k$ for some $k \geq 3$ then the bivalent neighbour of $z$ in $C$ is not bivalent in $H$ because then it would have second valency at least six and a cycle of length four through its neighbourhood. Hence, also in this case we have that $z$ has a neighbour which has lower valency in $C$ than in $H$. Let us call such a neighbour $z'$. Since no neighbour of $z$ can be adjacent to $y$ (or we would get a triangle in $H$) we must have that $z'$ is non-adjacent to $y$ and therefore adjacent to $u_1$ or $u_2$. Hence, $z' = w_{12}$ or $z' = w_{22}$, in either case we get that $d(H;z') = 4$ but then $d(H_{w,v_1}; z') = 3$, contradicting that $z'$ is bivalent in $H_{w,v_1}$.

\emph{Subcase 2.2.3.2:} $w$ is the only trivalent neighbour of $y$. In this case we have that $\nu(H_y) \leq \nu(H) + 5 \leq 7$. If $d(H_y;w_{i2}) \geq 4$, then $\nu(H_{y,u_i,v_{3-i}}) \leq \nu(H) + 5 - 2 - 7 \leq - 2$, contradicting the assumption \eqref{bigass}. Hence, $d(H_y; w_{i2}) \leq 3$ for $i \in \{1,2\}$. If $d(H_y; w_{i2}) \leq 2$ for some $i \in \{1,2\}$, then $y$ and $w_{i2}$ would both have two common neighbours. If both $w_{12}$ and $w_{22}$ have two common neighbours with $y$, then $\nu(H_y) \leq \nu(H) + 5 - (2) \leq 5$. All vertices $w_{12},u_1,v_1,v_2,u_2$ and $w_{22}$ would be bivalent in $H_y$. Now, since $w_{12}w_{22}$ we must have that the component, $C$, containing $w_{12}$ in $H_y$ is not 2-regular by Property \ref{prop-no2regular}. Hence $C$ contains some bivalent vertex, $x$, that is an endpoint of a path of bivalent vertices in $C$ of length at least five. Say $z$ is a vertex at distance 2 from $x$ in the path of bivalent vertices, then $\nu(H_{y,x,z}) \leq \nu(H) + 5 - (2) + 1 - 7 \leq -1$, contradicting assumption \eqref{bigass}.

Hence not both $w_{12}$ and $w_{22}$ have two common neighbours with $y$. Suppose that $w_{12}$ has two common neighbours with $y$. Then $\nu(H_{y,u_2,v_1}) \leq \nu(H) + 5 - (1) + 1 - 7 \leq 0$ but $w_{12}$ would be at most monovalent in $H_{y,u_2,v_1}$, contradicting Property \ref{prop-mindeg2}. Therefore $w_{12}$ has exactly one common neighbour with $y$. Analogously $w_{22}$ also has one common neighbour with $y$.

Suppose $\beta_i$ is the common neighbour of $w_{i2}$ and $y$ for $i \in \{1,2\}$ and let $\beta_3$ be remaining neighbour of $y$ that is not $w$, as has been illustrated in Figure \ref{fig:a6}.

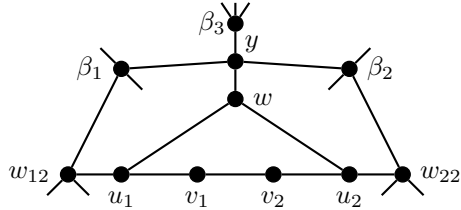
\begin{figure}[h]
\begin{center}
\begin{tikzpicture}[scale=1]
  \GraphInit[vstyle=Classic]
  \renewcommand*{\VertexSmallMinSize}{4pt}
  \Vertex[L=$u_1$,Lpos=-90]{u1}
  \Vertex[x=1,y=0,L=$v_1$,Lpos=-90]{v1}
  \Vertex[x=2,y=0,L=$v_2$,Lpos=-90]{v2}
  \Vertex[x=3,y=0,L=$u_2$,Lpos=-90]{u2}
  \Vertex[x=1.5,y=1,L=$w$,Lpos=0]{w}
  \Vertex[x=-0.7,y=0,L=$w_{12}$,Lpos=-180]{w12}
  \Vertex[x=3.7,y=0,L=$w_{22}$]{w22}
  \Vertex[x=3,y=1.4,L=$\beta_2$]{b2}
  \Vertex[x=0,y=1.4,L=$\beta_1$,Lpos=-180]{b1}
  \Vertex[Math,x=1.5,y=1.5,Lpos=45,Ldist=-3]{y}
  \Vertex[x=-1,y=-0.3,empty]{w121}
  \Vertex[x=-0.4,y=-0.3,empty]{w122}
  \Vertex[x=4,y=-0.3,empty]{w221}
  \Vertex[x=3.4,y=-0.3,empty]{w222}
  \Vertex[x=1.5,y=2,L=$\beta_3$,Lpos=-180,Ldist=-3]{b3}
  \Vertex[x=1.3,y=2.3,empty]{b31}
  \Vertex[x=1.5,y=2.3,empty]{b32}
  \Vertex[x=1.7,y=2.3,empty]{b33}
  \Vertex[x=-0.3,y=1.7,empty]{b11}
  \Vertex[x=0.3,y=1.1,empty]{b12}
  \Vertex[x=3.3,y=1.7,empty]{b21}
  \Vertex[x=2.7,y=1.1,empty]{b22}
  \Edge(b1)(b11)
  \Edge(b1)(b12)
  \Edge(b2)(b21)
  \Edge(b2)(b22)
  \Edge(b3)(b31)
  \Edge(b3)(b32)
  \Edge(b3)(b33)
  \Edge(y)(b3)
  \Edge(y)(b2)
  \Edge(y)(b1)
  \Edge(w12)(b1)
  \Edge(w22)(b2)
  \Edge(u1)(w)
  \Edge(u1)(w12)
  \Edge(u2)(w)
  \Edge(u2)(w22)
  \Edge(u1)(v1)
  \Edge(v1)(v2)
  \Edge(v2)(u2)
  \Edge(w12)(w121)
  \Edge(w12)(w122)
  \Edge(w22)(w221)
  \Edge(w22)(w222)
  \Edge(w)(y)
\end{tikzpicture}
\end{center}
\caption{The neighbourhood of $v_1$ and $v_2$ in $H$ in subcase 2.2.3.2.}
\label{fig:a6}
\end{figure}

We then have that $\nu(H_{y,u_i,v_{3-i}}) \leq \nu(H) + 5 + 1 - 7 \leq 1$ and $w_{(3-i),2}$ is bivalent in $H_{y,u_i,v_{3-i}}$ for $i \in \{1,2\}$. Therefore, inductively $w_{22}$ belongs to a $Ch_k$-component, $C$, of $H_{y,u_1,v_2}$ and $w_{12}$ belongs to a $Ch_l$-component of $H_{y,u_2,v_1}$ for some $k,l \geq 2$.

Suppose $k = 2$. Let $a_1$ and $a_4$ be the neighbours of $w_{22}$ in $C$. Also let $a_2$ and $a_3$ be the neighbours of $a_1$ and $a_4$ in $C$ that are not $w_{22}$, respectively. The situation is then like what is illustrated in Figure \ref{fig:a7}.

\begin{figure}[h]
\begin{center}
\begin{tikzpicture}[scale=1]
  \GraphInit[vstyle=Classic]
  \renewcommand*{\VertexSmallMinSize}{4pt}
  \Vertex[L=$u_1$,Lpos=-90]{u1}
  \Vertex[x=1,y=0,L=$v_1$,Lpos=-90]{v1}
  \Vertex[x=2,y=0,L=$v_2$,Lpos=-90]{v2}
  \Vertex[x=3,y=0,L=$u_2$,Lpos=-90]{u2}
  \Vertex[x=1.5,y=1,L=$w$,Lpos=0]{w}
  \Vertex[x=-0.7,y=0,L=$w_{12}$,Lpos=-180]{w12}
  \Vertex[x=3.7,y=0,L=$w_{22}$,Lpos=-90]{w22}
  \Vertex[x=3,y=1.4,L=$\beta_2$]{b2}
  \Vertex[x=0,y=1.4,L=$\beta_1$,Lpos=-180]{b1}
  \Vertex[Math,x=1.5,y=1.5,Lpos=45,Ldist=-3]{y}
  \Vertex[x=-1,y=-0.3,empty]{w121}
  \Vertex[x=-0.4,y=-0.3,empty]{w122}
  \Vertex[x=1.5,y=2,L=$\beta_3$,Lpos=-180,Ldist=-3]{b3}
  \Vertex[x=1.3,y=2.3,empty]{b31}
  \Vertex[x=1.5,y=2.3,empty]{b32}
  \Vertex[x=1.7,y=2.3,empty]{b33}
  \Vertex[x=-0.3,y=1.7,empty]{b11}
  \Vertex[x=0.3,y=1.1,empty]{b12}
  \Vertex[x=3.3,y=1.7,empty]{b21}
  \Vertex[x=2.7,y=1.1,empty]{b22}
  \Vertex[Math,x=3.7,y=0.7,Lpos=45,Ldist=-4]{a_1}
  \Vertex[Math,x=4.3,y=0.7,Ldist=-2]{a_2}
  \Vertex[Math,x=4.6,y=0.35,Ldist=-2]{a_3}
  \Vertex[Math,x=4.3,y=0,Ldist=-2]{a_4}
  \Vertex[x=3.7,y=1,empty]{a11}
  \Vertex[x=4.5,y=-0.2,empty]{a41}
  \Edge(a_1)(a11)
  \Edge(a_4)(a41)
  \Edges(w22,a_1,a_2,a_3,a_4,w22)
  \Edge(b1)(b11)
  \Edge(b1)(b12)
  \Edge(b2)(b21)
  \Edge(b2)(b22)
  \Edge(b3)(b31)
  \Edge(b3)(b32)
  \Edge(b3)(b33)
  \Edge(y)(b3)
  \Edge(y)(b2)
  \Edge(y)(b1)
  \Edge(w12)(b1)
  \Edge(w22)(b2)
  \Edge(u1)(w)
  \Edge(u1)(w12)
  \Edge(u2)(w)
  \Edge(u2)(w22)
  \Edge(u1)(v1)
  \Edge(v1)(v2)
  \Edge(v2)(u2)
  \Edge(w12)(w121)
  \Edge(w12)(w122)
  \Edge(w)(y)
\end{tikzpicture}
\end{center}
\caption{The neighbourhood of $v_1$ and $v_2$ in $H$ if $k = 2$ in subcase 2.2.3.2.}
\label{fig:a7}
\end{figure}
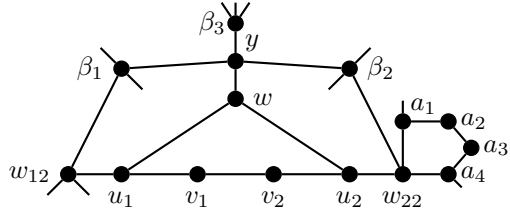

Now, since $w_{22}$ has no bivalent neighbours both $a_1$ and $a_4$ have valency at least three in $H$. Hence, $a_1$ and $a_4$ have neighbours in $\{\beta_1,\beta_2,\beta_3,w_{12}\}$.

If $d(a_1) = 3$ and $d(a_2) = 2$ then $\nu(H_{a_1}) \leq \nu(H) + 3 \leq 5$ and $a_4$ would have to be monovalent, by Property \ref{prop-leq6}, in $H_{a_1}$ as its neighbour $a_3$ is monovalent. Then $a_1$ and $a_4$ must have a common neighbour in $\{\beta_1,\beta_2,\beta_3,w_{12}\}$. But then $\nu(H_{a_1,a_3}) \leq \nu(H) + 3 - (1) - 4 \leq 0$ while $d(H_{a_1,a_3}; u_2) = 2, d^2(H_{a_1,a_3}; u_2) = 5,d(H_{a_1,a_3};v_2) = 2$ and $d^2(H_{a_1,a_3};v_2) = 4$, which is impossible since $u_2$ and $v_2$ should inductively belong to the same $Ch_{k'}$-component of $H_{a_1,a_3}$.

If, on the other hand, $d(a_1) = 3$ and $d(a_2) \geq 3$, then $\nu(H_{a_1}) \leq \nu(H) + 0 \leq 2$. If $d(a_2) \geq 4$ then $a_1$ would have second valency twelve in $H$, giving $\nu(H_{a_1}) \leq \nu(H) - 3 \leq -1$, contradicting assumption \eqref{bigass}. Hence $d(a_2) = 3$. Clearly $d(H_{a_1}; a_3) \geq 2$ and therefore $d(a_3) \geq 3$. If $d(a_3) \geq 4$ then $\nu(H_{a_2}) \leq \nu(H) - 3 \leq -1$, contradicting assumption \eqref{bigass}. Hence $d(a_2) = d(a_3) = 3$.

Now, $u_2$ is bivalent in $H_{a_2}$ and $H_{a_3}$ of second valency five but $v_2$ is bivalent of second valency four. Hence $u_2$ does not belong to a $Ch_k$-component in $H_{a_2}$ or $H_{a_3}$. In particular, $\nu(H_{a_1}), \nu(H_{a_4}) \geq 2$ since otherwise this would contradict the inductive assumption by Property \ref{tretton}. But if either $a_3$ or $a_2$ were adjacent to $\beta_2$ there would be a cycle of length four through that vertex. On the other hand if neither $a_2$ nor $a_3$ is adjacent to $\beta_2$ then all four vertices (since $H$ is triangle-free) $a_1,a_2,a_3$ and $a_4$ are adjacent to vertices in $\{w_{12},\beta_1,\beta_3\}$. It is easy to see that we then must get a cycle of length four through $a_1$ or $a_4$. Thus, independently of whether either $a_2$ or $a_3$ is adjacent to $\beta_2$, we get $\nu(H_{a_1}) \leq \nu(H) + 0 - (1)\leq 1$ or $\nu(H_{a_4}) \leq \nu(H) + 0 - (1) \leq 1$, in either case a contradiction.

Hence $d(a_1) \geq 4$ and analogously we get that $d(a_4) \geq 4$. So $\nu(H_{w_{22},v_1}) \leq \nu(H) + 5 - 7 \leq 0$. Thus neither $a_2$ nor $a_3$ is bivalent in $H$ (or they would be monovalent in $H_{w_{22},v_1}$). Therefore $e(V(C),\{\beta_1,\beta_2,\beta_3,w_{12}\}) \geq 5$, whence some vertex in $\{\beta_1,\beta_2,\beta_3,w_{12}\}$ is adjacent to two vertices in $V(C)$. These two vertices must form an independent set in $C$ and therefore are at distance two in $C$. Hence $\N(C_4; H,N(w_{22})) \geq 1$ which gives us $\nu(H_{w_{22},v_1}) \leq \nu(H) + 5 - (1) - 7 \leq - 1$, contradicting the assumption \eqref{bigass}.

Hence, $k \geq 3$ and the graph looks like in Figure \ref{fig:a8}.

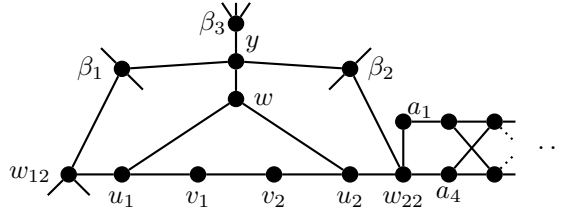
\begin{figure}[h]
\begin{center}
\begin{tikzpicture}[scale=1]
  \GraphInit[vstyle=Classic]
  \renewcommand*{\VertexSmallMinSize}{4pt}
  \Vertex[L=$u_1$,Lpos=-90]{u1}
  \Vertex[x=1,y=0,L=$v_1$,Lpos=-90]{v1}
  \Vertex[x=2,y=0,L=$v_2$,Lpos=-90]{v2}
  \Vertex[x=3,y=0,L=$u_2$,Lpos=-90]{u2}
  \Vertex[x=1.5,y=1,L=$w$,Lpos=0]{w}
  \Vertex[x=-0.7,y=0,L=$w_{12}$,Lpos=-180]{w12}
  \Vertex[x=3.7,y=0,L=$w_{22}$,Lpos=-90]{w22}
  \Vertex[x=3,y=1.4,L=$\beta_2$]{b2}
  \Vertex[x=0,y=1.4,L=$\beta_1$,Lpos=-180]{b1}
  \Vertex[Math,x=1.5,y=1.5,Lpos=45,Ldist=-3]{y}
  \Vertex[x=-1,y=-0.3,empty]{w121}
  \Vertex[x=-0.4,y=-0.3,empty]{w122}
  \Vertex[x=1.5,y=2,L=$\beta_3$,Lpos=-180,Ldist=-3]{b3}
  \Vertex[x=1.3,y=2.3,empty]{b31}
  \Vertex[x=1.5,y=2.3,empty]{b32}
  \Vertex[x=1.7,y=2.3,empty]{b33}
  \Vertex[x=-0.3,y=1.7,empty]{b11}
  \Vertex[x=0.3,y=1.1,empty]{b12}
  \Vertex[x=3.3,y=1.7,empty]{b21}
  \Vertex[x=2.7,y=1.1,empty]{b22}
  \Vertex[Math,x=3.7,y=0.7,Lpos=45,Ldist=-6]{a_1}
  \Vertex[Math,x=4.3,y=0.7,NoLabel]{a_2}
  \Vertex[Math,x=4.9,y=0.7,NoLabel]{a_3}
  \Vertex[Math,x=4.3,y=0,Ldist=-2,Lpos=-90]{a_4}
  \Vertex[x=4.9,y=0,NoLabel]{a41}
  \Vertex[x=5.2,y=0,empty]{a411}
  \Vertex[x=5.2,y=0.3,empty]{a412}
  \Edge(a41)(a411)
  \Edge[style=dotted](a41)(a412)
  \Vertex[x=5.2,y=0.7,empty]{a31}
  \Vertex[x=5.2,y=0.4,empty]{a32}
  \Edge(a_3)(a31)
  \Edge[style=dotted](a_3)(a32)
  \Edge(a41)(a_2)
  \Edge(a_4)(a41)
  \Edges(w22,a_1,a_2,a_3,a_4,w22)
  \Edge(b1)(b11)
  \Edge(b1)(b12)
  \Edge(b2)(b21)
  \Edge(b2)(b22)
  \Edge(b3)(b31)
  \Edge(b3)(b32)
  \Edge(b3)(b33)
  \Edge(y)(b3)
  \Edge(y)(b2)
  \Edge(y)(b1)
  \Edge(w12)(b1)
  \Edge(w22)(b2)
  \Edge(u1)(w)
  \Edge(u1)(w12)
  \Edge(u2)(w)
  \Edge(u2)(w22)
  \Edge(u1)(v1)
  \Edge(v1)(v2)
  \Edge(v2)(u2)
  \Edge(w12)(w121)
  \Edge(w12)(w122)
  \Edge(w)(y)
  \node at (5.7,0.35) {$\dots$};
\end{tikzpicture}
\end{center}
\caption{The neighbourhood of $v_1$ and $v_2$ in $H$ if $k \geq 3$ in subcase 2.2.3.2.}
\label{fig:a8}
\end{figure}

Let $a_1$ be the bivalent neighbour of $w_{22}$ in $C$, while $a_4$ is the trivalent neighbour of $w_{22}$ in $C$. $a_1$ is at least trivalent in $H$ and therefore adjacent to one of the vertices in $\{\beta_1,\beta_2,\beta_3,w_{12}\}$, which are all tetravalent. Then if $d(a_1) = 3$ we would have $d^2(a_1) \geq 3 + 4 + 4 = 11$, whence $\nu(H_{a_1}) \leq \nu(H) + 0 - (1) \leq 1$, since there is a cycle of length four through the trivalent neighbour of $a_1$ in $C$. However, $v_2$ would be bivalent, in $H_{a_1}$, with second valency four but not belonging to a $C_5$-component since $w$ is trivalent, contradicting Property \ref{tretton}. Hence $d(a_1) \geq 4$. Suppose that $d(a_4) = 3$. There is a cycle of length four through $a_4$ and we have $d^2(a_4) \geq 4 + 3 + 3 = 10$. We also have that $a_4 w_{12} \notin E(H)$ since $a_4$ also is trivalent in $C$. So in particular $v_1$ has second valency two in $H_{a_4,u_2}$ and thus we have $\nu(H_{a_4,u_2,v_1}) \leq \nu(H) + 3 - (1) + 1 - 7 \leq -3$, contradicting the assumption \eqref{bigass}. Hence also $d(a_4) \geq 4$. So also in the case when $k \geq 3$ we get that the two neighbours of $w_{22}$ in $C$ have valency at least four in $H$. But then we have $d^2(w_{22}) \geq 15$ and therefore $\nu(H_{w_{22},v_2}) \leq \nu(H) + 5 - (1) - 7 \leq -1$, contradicting assumption \eqref{bigass}.

\emph{Case 3:} $|N(u_1) \cap N(u_2)| = 2$. If $\nu(H) = 2$ we are done because we have two cycles of length five through $v_1$ and $v_2$. Suppose therefore that $\nu(H) \leq 1$. Let $w$ and $w'$ denote the two common neighbours of $u_1$ and $u_2$. Since they form, together with $u_1$ and $u_2$, a cycle of length four we get by Property \ref{tretton} that neither $w$ nor $w'$ have valency two. We then have the situation which is illustrated in Figure \ref{fig:a9}.

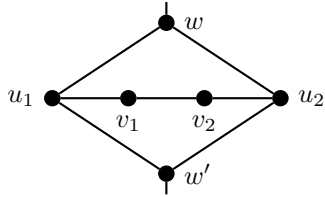
\begin{figure}[h]
\begin{center}
\begin{tikzpicture}[scale=1]
  \GraphInit[vstyle=Classic]
  \renewcommand*{\VertexSmallMinSize}{4pt}
  \Vertex[L=$u_1$,Lpos=-180]{u1}
  \Vertex[x=1,y=0,L=$v_1$,Lpos=-90]{v1}
  \Vertex[x=2,y=0,L=$v_2$,Lpos=-90]{v2}
  \Vertex[x=3,y=0,L=$u_2$,Lpos=0]{u2}
  \Vertex[x=1.5,y=1,L=$w$,Lpos=0]{w}
  \Vertex[x=1.5,y=1.3,empty]{w-1}
  \Vertex[x=1.5,y=-1,L=$w'$,Lpos=0]{wp}
  \Vertex[x=1.5,y=-1.3,empty]{wp-1}
  \Edge(u1)(w)
  \Edge(u2)(w)
  \Edge(u1)(wp)
  \Edge(u2)(wp)
  \Edge(u1)(v1)
  \Edge(v1)(v2)
  \Edge(v2)(u2)
  \Edge(w)(w-1)
  \Edge(wp)(wp-1)
\end{tikzpicture}
\end{center}
\caption{The neighbourhood of $v_1$ and $v_2$ in $H$ in case 3.}
\label{fig:a9}
\end{figure}

Now, since $\nu(H_{v_1}) \leq \nu(H) + 1 - (1) \leq 1$, and $u_2$ is bivalent in $H_{v_1}$, we inductively get that $u_2$ belongs to some $Ch_l$-component of $H_{v_1}$, for some $l \geq 2$. But then $H \iso Ch_{l+1}$ by the recursive Definition \ref{defChk}, here $u_2$ corresponds to ``$x$'' in the definition and $v_1$,$v_2$ and $u_1$ correspond ``$v$'',''$w_1$'' and ``$w_2$'', respectively. However, this contradicts the assumption that $H \not \iso Ch_k$ for all $k \geq 2$.
\end{proof}

As a simple corollary of Property \ref{containsChk2} we have the following property.

\begin{property}
\label{containsChk1}
Let $H \leq G$ be connected with $\delta(H) = 2$. If $\nu(H) \leq 1$ then $H \iso Ch_k$ for some $k \geq 2$. 
\end{property}
\begin{proof}
By Property \ref{tretton} either $H \iso C_5 = Ch_2$ or there is a vertex $v_1 \in V(H)$ such that $d^2(v) = 5$. Hence, by the first part of Property \ref{containsChk2} we get that $H \iso Ch_k$ for some $k \geq 3$.
\end{proof}

\begin{property}
  \label{mindeg3}
  If $\nu(G) \leq -1$ or $\nu(G) = 0$ but $G \notin \mathcal{G}$, then $\delta(G) \geq 3$.
\end{property}
\begin{proof}
  Suppose that $G$ satisfies the premises. By Property \ref{prop-mindeg2} we have $\delta(G) \geq 2$ and by Lemma \ref{nu-1connected} $G$ is connected. Suppose that $\delta(G) = 2$, then by Property \ref{containsChk1} we have that $G \iso Ch_k$ for some $k \geq 2$, contradicting that $\nu(Ch_k) = 0$ for all $k \geq 2$. Hence, $\delta(G) \geq 3$.
\end{proof}

We now show that in subgraphs $H \leq G$, with $\nu(H) \leq 2$, bivalent vertices with bivalent neighbours appear in ``balanced pairs'', i.e. any two adjacent bivalent vertices have the same second valency.

\begin{property}
\label{balancedpair}
Let $H \leq G$ be such that $\nu(H) \leq 2$ and $v \in V(H)$ be bivalent with a bivalent neighbour $u \in N(v)$ then $d^2(v) = d^2(u)$.
\end{property}
\begin{proof}
The assertion is trivially true if $v$ belongs to a $C_5$-component of $H$. Therefore we may assume that $v \in V(H)$ has exactly one bivalent neighbour. By Property \ref{tretton}(i) we then have that $d^2(v),d^2(u) \in \{5,6\}$. For symmetry reasons it is enough to show that if $d^2(v) = 6$ then $d^2(u) \geq 6$.

Let $w$ be the tetravalent neighbour of $v$ and $x$ the neighbour of $u$ that is not $v$. Suppose moreover, for a contradiction, that $x$ is trivalent. Then $x$ is bivalent in $H_v$ and, since $\nu(H_v) \leq \nu(H) - 2 \leq 0$, $x \in C \iso Ch_k \in \calC(H_v)$ for some $k \geq 2$ (by assumption \eqref{bigass}).

By Property \ref{ddestab} $N_2(v)$ is a minimal destabiliser of size four of $C$.
$N_2$ is not connected in $H_v$ since otherwise it would induce a $P_4$, $K_{1,3}$ or $C_4$. In either of these three cases we would get $\N(C_4; H, N(v)) \geq 2$ and then $\nu(H_v) \leq \nu(H) - 2 - (2) \leq -2$, contradicting assumption \eqref{bigass}.
Thus by Lemma \ref{Chkdestab4} we must have that $k = 3$ and $N_2(v) = V_2(C)$. But then $w$ is adjacent to two adjacent vertices in $C$, contradicting that $G$ is triangle-free.
\end{proof}

\begin{property}
\label{Hsndval}
If $H \leq G$, $C_5 \notin \calC(H), \nu(H) \leq 2$ and there are no cycles of length four through any vertices of valency less than four in $H$, then every bivalent vertex of $H$ has second valency six.
\end{property}
\begin{proof}
Since $C_5 \notin \calC(H)$ there are no bivalent vertices of second valency four by Corollary \ref{kor1}. Assume that $v \in V(H)$ is bivalent with second valency five. Then the bivalent neighbour of $v$ would have second valency five by Property \ref{balancedpair}. Therefore, by Property \ref{containsChk2}, there is some cycle of length four through the trivalent neighbour of $v$, contradicting the assumption.
\end{proof}

We will later, under the extra assumption that $\nu(G) \leq -1$, show that $G$ does not contain any cycles of length four. In that situation the above property gives us that every bivalent vertex in $H \leq G$ has second valency four or six.

\begin{property}
\label{bivaldist}
If $H \leq G$, $C_5 \notin \calC(H)$, $\nu(H) \leq 2$ and there are no cycles of length four through vertices of valency less than four in $H$, then for any pair $u,w$ of distinct bivalent vertices of $H$; $\dist(u,w) \in \{1,3\}$.
\end{property}
\begin{proof}
Both $u$ and $v$ have second valency six by Property \ref{Hsndval}. If $\dist(u,w) \geq 4$ then the removal of $u$ and its neighbours from $H_v$ does not affect the valency or second valency of $w$, whence $\nu(G_{v,u,w}) \leq \nu(H) + 3 - 2 - 2 \leq -2$, contradicting assumption \eqref{bigass}.

Hence $\dist(u,w) \leq 3$. If $\dist(u,w) = 2$ then $\nu(H_{v,u}) \leq \nu(G) + 3 - 2 \leq 0$ but $w$ would be at most monovalent in $G_{v,u}$, contradicting Property \ref{prop-mindeg2}.
\end{proof}

\begin{property}
\label{onlytwo24}
If $H \leq G$, $C_5 \notin \calC(H)$, $\nu(H) \leq 2$, there are no cycles of length four through vertices of valency less than four in $H$, and $u \in V(H)$ is a bivalent vertex with a bivalent neighbour, then $u$ and the bivalent neighbour of $u$ are the only two bivalent vertices of $H$.
\end{property}
\begin{proof}
Suppose for a contradiction that $u_1$ and $u_2$ are adjacent bivalent vertices in $H$ and that there is a third bivalent vertex $x$ in $H$. Then since $\dist(x,u_1),\dist(x,u_2) \in \{1,3\}$ (by Property \ref{bivaldist}) we must have that $\dist(x,u_1) = \dist(x,u_2) = 3$ because $H$ is triangle-free.

By Property \ref{Hsndval} we must also have $d^2(H; x) = 6$ and therefore $\nu(H_{x}) \leq \nu(H) - 2 \leq 0$. Since $u_1$ and $u_2$ are adjacent bivalent vertices in $H_{x}$ they belong to some $Ch_k$-component, $C$, of $H_{x}$ for some $k \geq 2$. The local structure in the graph then looks like in Figure \ref{fig:b1}.

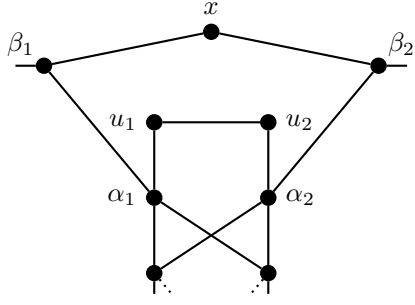
\begin{figure}[h]
\begin{center}
\begin{tikzpicture}[scale=1]
  \GraphInit[vstyle=Classic]
  \renewcommand*{\VertexSmallMinSize}{4pt}
  \Vertex[Math,x=0,y=1.2,Lpos=90]{x}
  \Vertex[Math,x=-0.75,y=0,Lpos=180]{u_1}
  \Vertex[Math,x=0.75,y=0]{u_2}
  \Vertex[x=-0.75,y=-1,Lpos=180,L=$\alpha_1$]{a1}
  \Vertex[x=0.75,y=-1,L=$\alpha_2$]{a2}
  \Vertex[x=-2.2,y=0.75,L=$\beta_1$,Lpos=135,Ldist=-3]{b1}
  \Vertex[x=-2.6,y=0.75,empty]{b11}
  \Vertex[x=2.2,y=0.75,L=$\beta_2$,Lpos=45,Ldist=-3]{b2}
  \Vertex[x=2.6,y=0.75,empty]{b21}
  \Vertex[x=-0.75,y=-2,NoLabel]{a11}
  \Vertex[x=0.75,y=-2,NoLabel]{a21}
  \Vertex[x=-0.75,y=-2.3,empty]{al1}
  \Vertex[x=-0.5,y=-2.3,empty]{al2}
  \Vertex[x=0.75,y=-2.3,empty]{ar1}
  \Vertex[x=0.5,y=-2.3,empty]{ar2}
  \Edge(a11)(al1)
  \Edge[style=dotted](a11)(al2)
  \Edge(a21)(ar1)
  \Edge[style=dotted](a21)(ar2)
  \Edge(a1)(a11)
  \Edge(a2)(a21)
  \Edge(a1)(a21)
  \Edge(a2)(a11)
  \Edge(b2)(b21)
  \Edge(b1)(b11)
  \Edge(b2)(a2)
  \Edge(b1)(a1)
  \Edge(x)(b1)
  \Edge(x)(b2)
  \Edge(u_1)(u_2)
  \Edge(u_1)(a1)
  \Edge(u_2)(a2)
\end{tikzpicture}
\end{center}
\caption{The neighbourhood of $x$ in $H$.}
\label{fig:b1}
\end{figure}

Let $\alpha_1$ and $\alpha_2$ be the tetravalent neighbours of $u_1$ and $u_2$ in $H$, respectively. Then $\alpha_1$ and $\alpha_2$ are trivalent in $H_{x}$ since otherwise $x$ would have to have two common neighbours with $\alpha_i$ (for some $i \in \{1,2\}$) in $H$ and we would get $\nu(H_{x}) \leq \nu(H) - 2 - (1) \leq -1$, contradicting assumption \eqref{bigass}. Hence $k \geq 3$. Thus there is a cycle of length four through $\alpha_1$ and $\N(C_4;H,N(u_1)) > 0$ which gives $\nu(H_{u_1}) \leq \nu(H) + 3 - 2 - (1) \leq -1$, contradicting assumption \eqref{bigass}.
\end{proof}

We now define another class of graphs, which will turn out to be useful. These graphs have have $\nu$-value two. It will later be shown that this class consists of all graphs with $\nu$-value two among the connected triangle-free graphs with minimum valency two.

\begin{definition}
Let $a_1,a_2,b_1$ and $b_2$ denote the four bivalent vertices of $G = Ch_3$, where $a_1a_2,b_1b_2 \in E(G)$. We define the \emph{shackled chain} $SCh_1$ by letting $V(SCh_1) = V(G) \cupdot \{v,w_1,w_2\}$ and $E(SCh_1) = E(G) \cup \{vw_1,vw_2,w_1a_1,w_1b_1,w_2a_2,w_2b_2\}$. For $k \geq 2$ we define $SCh_k$ recursively as follows. Let $a$ be any bivalent vertex in $SCh_{k-1}$ with neighbours $b_1$ and $b_2$. We then set $V(SCh_k) = V(SCh_{k-1}) \cupdot \{v,w_1,w_2\}$ and $E(SCh_k) = E(SCh_{k-1}) \cup \{vw_1,vw_2,w_1a,w_2b_1,w_2b_2\}$.
\end{definition}

For example in Figure \ref{figsch} the graphs $SCh_1$ and $SCh_2$ are shown.
\begin{figure}[h]
\begin{center}
\begin{tikzpicture}[scale=0.75]
  \GraphInit[vstyle=Classic]
  \renewcommand*{\VertexSmallMinSize}{4pt}
  \Vertex[Math,x=0,y=0,Lpos=180]{b_2}
  \Vertex[Math,NoLabel,x=0,y=1]{d}
  \Vertex[Math,NoLabel,x=0,y=2]{c}
  \Vertex[Math,x=0,y=3,Lpos=180]{a_1}
  \Vertex[Math,x=1,y=0]{b_1}
  \Vertex[Math,NoLabel,x=1,y=1]{f}
  \Vertex[Math,NoLabel,x=1,y=2]{e}
  \Vertex[Math,x=1,y=3]{a_2}
  \Vertex[Math,x=3,y=0.5]{w_2}
  \Vertex[Math,x=3,y=2.5]{w_1}
  \Vertex[Math,x=3.5,y=1.5]{v}
  \Edges(a_1,c,d,b_2,b_1,f,e,a_2,a_1)
  \Edge(e)(d)
  \Edge(c)(f)
  \Edge(v)(w_1)
  \Edge(v)(w_2)
  \Edge[style={bend right}](w_1)(a_1)
  \Edge(w_1)(b_1)
  \Edge(w_2)(a_2)
  \Edge[style={bend left}](w_2)(b_2)
\end{tikzpicture} $\; \;$ \begin{tikzpicture}[scale=0.75]
  \GraphInit[vstyle=Classic]
  \renewcommand*{\VertexSmallMinSize}{4pt}
  \Vertex[Math,x=0,y=0,NoLabel]{ob_2}
  \Vertex[Math,NoLabel,x=0,y=1]{od}
  \Vertex[Math,NoLabel,x=0,y=2]{oc}
  \Vertex[Math,x=0,y=3,NoLabel]{oa_1}
  \Vertex[Math,x=1,y=0,NoLabel]{ob_1}
  \Vertex[Math,NoLabel,x=1,y=1]{of}
  \Vertex[Math,NoLabel,x=1,y=2]{oe}
  \Vertex[Math,x=1,y=3,NoLabel]{oa_2}
  \Vertex[Math,x=3,y=0.5,Lpos=-90]{b_2}
  \Vertex[Math,x=3,y=2.5,Lpos=90]{b_1}
  \Vertex[Math,x=3.5,y=1.5,Lpos=180]{a}
  \Vertex[Math,x=4.5,y=0.5]{w_2}
  \Vertex[Math,x=4.5,y=2.5]{w_1}
  \Vertex[Math,x=5.5,y=1.5]{v}
  \Edges(oa_1,oc,od,ob_2,ob_1,of,oe,oa_2,oa_1)
  \Edge(oe)(od)
  \Edge(oc)(of)
  \Edge(a)(b_1)
  \Edge(a)(b_2)
  \Edge[style={bend right}](b_1)(oa_1)
  \Edge(b_1)(ob_1)
  \Edge(b_2)(oa_2)
  \Edge[style={bend left}](b_2)(ob_2)
  \Edge[style={bend right}](w_2)(b_1)
  \Edge(w_2)(b_2)
  \Edge(w_1)(a)
  \Edges(w_2,v,w_1)
\end{tikzpicture}
\end{center}
\caption{The smallest shackled chains $SCh_1$ (left) and $SCh_2$ (right).}
\label{figsch}
\end{figure}
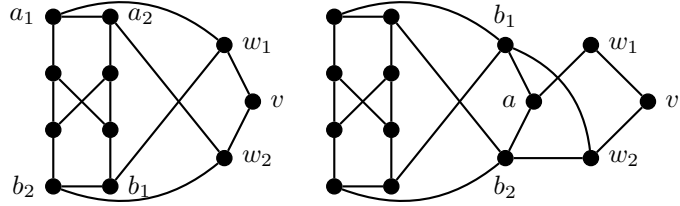

It is not difficult to prove that the definition does not depend (up to isomorphism) on the choice of bivalent vertex in the recursive construction. It is also not hard to see that $n(SCh_k) = 3k + 8$, $e(SCh_k) = 5k + 11$, $\alpha(SCh_k) = k + 3$ and $\N(C_4;SCh_k) = k$ for all $k \geq 1$. Therefore we get that $\nu(SCh_k) = 2$ for all $k \geq 1$.

We begin by showing that the chains and the shackled chains are the only connected graphs with $\nu$-value no more than two and bivalent vertices.

\begin{property}
 \label{classprop1}
  If $H \leq G$ is connected, $\nu(H) \leq 2$ and $\delta(H) = 2$, then $H \iso Ch_k$ for some $k \geq 2$ or $H \iso SCh_{\ell}$ for some $\ell \geq 1$.
\end{property}
\begin{proof}
The assertion is trivial for small graphs, i.e. for $H \leq G$ such that $n(H) = 1$. Assume that for all $J \leq G$ such that $n(J) < n(H)$ the assertion holds. Fix some bivalent vertex $v \in V(H)$. By Properties \ref{prop-mindeg2} and \ref{tretton} we have $d^2(v) \in \{4,5,6\}$.

If $d^2(v) = 4$, then we are done by Corollary \ref{kor1}. If $d^2(v) = 6$ then $\nu(H_v) \leq \nu(H) - 2  \leq 0$ and $\N(C_4; N(v)) = 0$, since otherwise $\nu(H_v) \leq -1$. We then have $|N_2(v)| = 4$ vertices at distance two from $v$. If $N_2(v)$ were connected, then it would induce one of the graphs $C_4$, $K_{1,3}$ or $P_4$, giving $\N(C_4; N(v)) \geq 1$ in all cases. Hence $N_2(v)$ must be a disconnected minimal destabiliser (by Property \ref{prop3} and since a non-minimal disconnected destabiliser of $Ch_k$ ($k \geq 2$) would contains a bivalent vertex and its two neighbours, by Lemma \ref{Chkdestab3}).

It is easily checked that neither $W_5$ nor $(2C_7)_{2i}$ have destabilisers of size four that are disconnected. Also $BC_k$ ($k \geq 5$) has only connected minimal destabilisers of size four (see e.g. \cite[Lemma 6.3(e)]{carc}). The only possibility is that $\delta(H_v) = 2$ and $H_v \iso Ch_k$ for some $k \geq 2$. But then, by Lemma \ref{Chkdestab4}, $k = 3$ and $N_2(v) = V_2(Ch_k)$, which means precisely that $H \iso SCh_1$.

Finally, if $d^2(v) = 5$ then $\N(C_4;N(v)) \geq 1$ by Property \ref{containsChk2}. Hence, $\nu(H_v) \leq \nu(H) + 1 - (1) \leq 2$. The bivalent neighbour of $v$ has a trivalent neighbour by Property \ref{balancedpair}. Thus $\delta(H_v) = 2$ and $H_v \iso Ch_k$ for some $k \geq 2$ or $H_v \iso SCh_\ell$ for some $\ell \geq 1$. By the recursive constructions of $Ch_k$ and $SCh_\ell$ we get that either $H \iso Ch_{k+1}$ or $H \iso SCh_{\ell + 1}$.
\end{proof}

\begin{property}
  \label{classprop2}
  If $H \leq G$ is connected, $\nu(H) \leq 0$, $\delta(H) = 3$ and there is a trivalent vertex $v \in V(H)$ such that $\N(C_4; v) > 0$, then $H \iso BC_k$ for some $k \geq 5$.
\end{property}
\begin{proof}
Suppose that $a \in V(H)$ is a trivalent vertex such that $\N(C_4; a) > 0$, say $\{a,b,c,d\} \subseteq V(H)$ forms a cycle of length four in $H$, where $ab,ad \in E(H)$. Since $\delta(H) = 3$ we have that $d^2(a) \geq 9$. Thus $\nu(H_a) \leq \nu(H) + 6 - (1) \leq 5$.

By Property \ref{prop-leq6} the vertex $c$ lies in a $K_2$-component of $H_a$ if $c$ is trivalent in $H$. In this case say that $t$ is the other vertex of that $K_2$-component. Then $t$ must have at least two common neighbours with $a$ (since $d(t) \geq 3$), whence $N(t) \cap N(c) \neq \emptyset$, contradicting $H$ being triangle-free.

Hence there are no 4-cycles in $H$ with trivalent vertices opposite to each other (i.e. at distance two in the cycle). If both $b$ and $d$ were at least tetravalent, then $d^2(a) \geq 11$. This would give $\nu(H_a) \leq \nu(H) + 0 - (1) \leq -1$, contradicting assumption \eqref{bigass}. Suppose, on the other hand, that $d(a) = d(b) = 3$ but $d(d), d(c) \geq 4$. Then $d(d) = d(c) = 4$ by a similar argument to the previous (either $\nu(H_a) \leq -1$ or $\nu(H_b) \leq -1$, in both cases contradicting assumption \eqref{bigass}).

If $v_1 \in V(H)$ is a trivalent vertex such that $\N(C_4;v_1) > 0$ we must therefore have that the local structure in the neighbourhood of $v_1$ is as in Figure \ref{lnfig1}. We let $v_2$ be the trivalent neighbour of $v_1$ in a 4-cycle. Let $u_i$ be the tetravalent neighbour of $v_i$ in the cycle, and define $w_i$ to be the trivalent neighbour of $v_i$ in $N(v_i) \setminus \{u_i,v_{3-i}\}$.
\begin{figure}[h]
\begin{center}
\begin{tikzpicture}[scale=1]
  \GraphInit[vstyle=Classic]
  \renewcommand*{\VertexSmallMinSize}{4pt}
  \Vertex[Math,x=0,y=0,Lpos=-90]{w_1}
  \Vertex[Math,x=1,y=0,Lpos=-90]{v_1}
  \Vertex[Math,x=2,y=0,Lpos=-90]{v_2}
  \Vertex[Math,x=3,y=0,Lpos=-90]{w_2}
  \Vertex[Math,x=1,y=1,Lpos=-90]{u_2}
  \Vertex[Math,x=2,y=1,Lpos=-90]{u_1}
  \Vertex[x=2,y=1.5,empty]{u11}
  \Vertex[x=2.5,y=1,empty]{u12}
  \Vertex[x=0.5,y=1,empty]{u21}
  \Vertex[x=1,y=1.5,empty]{u22}
  \Vertex[x=-0.3,y=0.3,empty]{w11}
  \Vertex[x=-0.3,y=-0.3,empty]{w12}
  \Vertex[x=3.3,y=0.3,empty]{w21}
  \Vertex[x=3.3,y=-0.3,empty]{w22}
  \Edges(w_1,v_1,u_1,u_2,v_2,w_2)
  \Edge(v_1)(v_2)
  \Edge(u_1)(u11)
  \Edge(u_1)(u12)
  \Edge(u_2)(u21)
  \Edge(u_2)(u22)
  \Edge(w_1)(w11)
  \Edge(w_1)(w12)
  \Edge(w_2)(w21)
  \Edge(w_2)(w22)
\end{tikzpicture}
\end{center}
\caption{The local structure around the trivalent vertex $v_1$ in $H$.}
\label{lnfig1}
\end{figure}

Suppose that $d^2(u_1) \geq 15$. Then $\nu(H_{u_1}) \leq \nu(H) + 5 - (1) \leq 4$ and $d(H_{u_1}; v_2) = 1$. Thus $d(H_{u_1}; w_2) = 1$ by Property \ref{prop-leq6}. Therefore $u_1$ and $w_2$ has two common neighbours. Since $u_1,v_1 \notin N(w_2)$ we get $|N(w_2) \cap N(u_1)| = |N(u_1) \setminus \{v_1,u_1\}| = 2$ and $\N(C_4; u_1) \geq 2$. But then $\nu(H_{u_1}) \leq \nu(H) + 5 - (2) \leq 3$, and $d(H_{u_1}; v_2) = 1$, contradicting Property \ref{prop-mindeg2}.

Hence, $d^2(u_1) \leq 14$ and therefore, in particular, $u_1$ has at least two trivalent neighbours. Analogously we get that $d^2(u_2) \leq 14$.

Let $x_i,y_i \in N(u_i) \setminus \{v_i,u_{3-i}\}$ be such that $d(x_i) \leq d(y_i)$ for $i \in [2]$. Then, in particular, $d(x_i) = 3$.

Suppose that $d(y_1) = 3$, then $\nu(H_{v_1}) \leq \nu(H) + 3 - (1) \leq 2$. If $H_{v_1}$ is connected then $H_{v_1} \iso C_5$ by Property \ref{classprop1} and since there are more than four bivalent vertices in $H_{v_1}$ (five or six depending on wether $w_1$ has two or three trivalent neighbours). Then $u_1$ would have three neighbours in a cycle of length five, contradicting $H$ being triangle-free.

Hence, $H_{v_1}$ is not connected. By Property \ref{classprop1} and since $|N_2(v_1)| \leq 6$ we have no more than two components, whence $|\Cc(H_{v_1})| = 2$.

None of the components of $H_{v_1}$ can contain all the bivalent vertices since then it would contains at least five bivalent vertices, and then we would have a $C_5$-component in which $u_1$ has three neighbours. Therefore both components have $\nu$-value at most two and minimum valency at most two, as well. So by Property \ref{classprop1} we have that each of the two components contain 1,2,4 or 5 bivalent vertices. Hence the distribution of bivalent vertices among the two components is either $1 + 5$ (in which case $H_{v_1} \iso SCh_1 + C_5$) or $2 + 4$ (in which case $H_{v_1} \iso SCh_{\geq 2} + Ch_{\geq 3}$). However, $H_{v_1} \iso SCh_1 + C_5$ is not possible since then $u_1$ would have two neighbours in a $C_5$-component, giving $\N(C_4; N(v_1)) \geq 2$. This would in turn yield $\nu(H_{v_1}) \leq 1$, but this contradicts $\nu(SCh_1 + C_5) = 2$. On the other hand, also $H_{v_1} \iso SCh_{\geq 2} + Ch_{\geq 3}$ is impossible because the four bivalent vertices of the $Ch_{\geq 3}$-component must all be in $N_2(v_1)$. Hence, $|N_2(v_1) \cap V(SCh_{\geq 2})| \leq 2$, but $SCh_{\geq 2}$ is 2-stable by Property \ref{prop1}. This would therefore give redundant edges in $H$ by Lemma \ref{lemma:A}, and therefore contradict Property \ref{prop-edge-critical}.

This shows us that $d(y_1) \geq 4$, and since $d^2(u_1) \leq 14$ we get that $d(y_1) = 4$. Analogously, $d(y_2) = 4$. So the situation looks like in Figure \ref{lnfig2}.
\begin{figure}[h]
\begin{center}
\begin{tikzpicture}[scale=1]
  \GraphInit[vstyle=Classic]
  \renewcommand*{\VertexSmallMinSize}{4pt}
  \Vertex[Math,x=0,y=0,Lpos=-90]{w_1}
  \Vertex[Math,x=1,y=0,Lpos=-90]{v_1}
  \Vertex[Math,x=2,y=0,Lpos=-90]{v_2}
  \Vertex[Math,x=3,y=0,Lpos=-90]{w_2}
  \Vertex[Math,x=1,y=1,Lpos=-90]{u_2}
  \Vertex[Math,x=2,y=1,Lpos=-90]{u_1}
  \Vertex[Math,x=2,y=1.5]{y_1}
  \Vertex[Math,x=3,y=1,Lpos=-90]{x_1}
  \Vertex[Math,x=0,y=1,Lpos=-90]{x_2}
  \Vertex[Math,x=1,y=1.5]{y_2}
  \Vertex[x=-0.3,y=0.3,empty]{w11}
  \Vertex[x=-0.3,y=-0.3,empty]{w12}
  \Vertex[x=3.3,y=0.3,empty]{w21}
  \Vertex[x=3.3,y=-0.3,empty]{w22}
  \Vertex[x=-0.3,y=0.7,empty]{x21}
  \Vertex[x=-0.3,y=1.3,empty]{x22}
  \Vertex[x=3.3,y=0.7,empty]{x11}
  \Vertex[x=3.3,y=1.3,empty]{x12}
  \Vertex[x=1,y=1.8,empty]{y21}
  \Vertex[x=0.7,y=1.8,empty]{y22}
  \Vertex[x=1.3,y=1.8,empty]{y23}
  \Vertex[x=2,y=1.8,empty]{y11}
  \Vertex[x=1.7,y=1.8,empty]{y12}
  \Vertex[x=2.3,y=1.8,empty]{y13}
  \Edges(w_1,v_1,u_1,u_2,v_2,w_2)
  \Edge(v_1)(v_2)
  \Edge(u_1)(x_1)
  \Edge(u_1)(y_1)
  \Edge(u_2)(x_2)
  \Edge(u_2)(y_2)
  \Edge(w_1)(w11)
  \Edge(w_1)(w12)
  \Edge(w_2)(w21)
  \Edge(w_2)(w22)
  \Edge(x_1)(x11)
  \Edge(x_1)(x12)
  \Edge(x_2)(x21)
  \Edge(x_2)(x22)
  \Edge(y_2)(y21)
  \Edge(y_2)(y22)
  \Edge(y_2)(y23)
  \Edge(y_1)(y11)
  \Edge(y_1)(y12)
  \Edge(y_1)(y13)
\end{tikzpicture}
\end{center}
\caption{The local structure around the trivalent vertex $v_1$ in $H$.}
\label{lnfig2}
\end{figure}

Note that in $H_{v_1}$ the vertices $w_2$, $x_1$ and $u_2$ are bivalent, and moreover so is at least one of the vertices in $N(w_1) \setminus \{v_1\}$.

Suppose that $e(\{x_2,y_2\}, N[v_1]) = 0$. Then $d(H_{v_1}; u_2) = 2$, $d^2(H_{v_1}; u_2) = 3 + 4 = 7$, which is impossible since then $\nu(H_{v_1,u_2}) \leq \nu(H) + 3 - (1) - 5 \leq - 3$. Hence, $e(\{x_2,y_2\},N[v_1]) \geq 1$. Neither $y_2$ nor $x_2$ is adjacent to $v_1,u_1$ or $v_2$. Thus, we must have that $y_2w_1 \in E(H)$ or $x_2w_1 \in E(H)$.

Now, suppose that $e(\{x_2,y_2\},w_1) = 1$. We then have that $\nu(H_{v_1,u_2}) \leq \nu(H) + 3 - (1) - 2 \leq 0$ since $d^2(H_{v_1}; u_2) = 6$. Also, $d(H_{v_1,u_2}; x_1) = 2$ and therefore $x_2 \in C \iso Ch_k \in \Cc(H_{v_1,u_2})$ for some $k \geq 2$.

If $k = 2$ then the five vertices in $N(v_1)\cup N(u_2)$ that may be adjacent to some vertex of $V(C)$ are $w_1,v_2,u_1,y_2$ and $x_2$. If any of them were adjacent to more than one vertex in $V(C)$ we would have $\N(C_4;N(v_1) \cup N(u_2)) \geq 2$, giving $\nu(H_{v_1,u_2}) \leq -1$, contradicting assumption \eqref{bigass}. Hence $e(x, V(C)) = 1$ for all $x \in \{w_1,v_2,u_1,y_2,x_2\}$, and therefore $w_2 \in V(C)$. Say that $\{\alpha,\beta\} = V(C) \setminus N(x_1)$. We now have $\nu(H_{x_1}) \leq \nu(H) + 3 - (1) \leq 2$ and $d(H_{x_1}; \alpha) = d(H_{x_1}; \beta) = 2$. If $d^2(H_{x_1}; \alpha) = d^2(H_{x_1}; \beta) = 2 + 3$, then there is a cycle of length four through the trivalent neighbours of $\alpha$ and $\beta$ in $H_{x_1}$ by Property \ref{containsChk2}. We then would have that $\N(C_4;N(v_1) \cup N(u_2)) \geq 2$, which would yield $\nu(H_{v_1,u_2}) \leq -1$, contradicting the inductive assumption. Thus, both $\alpha$ and $\beta$ have second valency $2 + 4$ (by Property \ref{balancedpair}), in $H_{x_1}$. Both $\alpha$ and $\beta$ are then adjacent to $y_2$ since $y_2$ and $u_1$ are the only tetravalent vertices among $w_1,v_2,u_1,y_2,x_2$ and $u_1 \notin V(H_{x_1})$. This, however, contradicts $H$ being triangle-free.

Hence, $k \geq 3$. If $e(N[v_1], N(x_1) \setminus \{u_1\}) = 0$, then $d^2(H_{v_1};x_1) = 6$ which implies that $\nu(H_{v_1,x_1}) \leq \nu(H) + 3 - (1) - 2 - (1) \leq -1$. Therefore we must have that $e(N[v_1],N(x_1) \setminus \{u_1\}) \geq 1$. Now, $d^2(x_1) = 4 + 3 + 3$ and $\N(C_4; N(x_1)) \geq 2$, so $\nu(H_{x_1}) \leq \nu(H) + 3 - (2) \leq 1$. Moreover, $d(H_{x_1}; v_1) = 2$ which implies that $v_1 \in C' \iso Ch_\ell \in \Cc(H_{x_1})$ for some $\ell \geq 2$. Therefore at least one of the vertices $w_1$ and $v_2$ is adjacent to $N(x_1) \setminus \{u_1\}$. If not both of them are, then $\ell \geq 3$ and the trivalent neighbour of $v_1$ in $H_{x_1}$ has a 4-cycle through it. This would give us that $\N(C_4;N(v_1)) \geq 2$, whence $\nu(H_{v_1,u_2}) \leq \nu(H) + 3 - (2) - 2 \leq -1$, contradicting assumption \eqref{bigass}. We must therefore have that both $w_1$ and $v_2$ are adjacent to $N(c_1) \setminus \{x_1\}$, which gives us that $C' \iso C_5$. But $d(x_1) = 3$, so at least two vertices in $N(x_1)$ has two neighbours in $V(C')$.

If $y_2 \in V(C')$ then since $y_2$ is tetravalent all three neighbours of $x_1$ would have to have two neighbours in $C'$, giving $\nu(H_{x_1}) \leq \nu(H) + 3 - (3) \leq 0$, and $x_2 \in N(x_1)$, $u_2 \in V(C')$. Both $y_2$ and $u_2$ are tetravalent in $H$ but bivalent in $H_{x_1}$, so $e(C', H \setminus C') \geq 7$, and therefore at least one vertex of $N(x_1)$ would have to be adjacent to three vertices in $V(C')$. This contradicts $H$ being triangle-free.

Hence we must have $y_2 \notin V(C')$. This means, in particular, that $w_1x_2 \in E(H)$ and $V(C') = \{v_1,v_2,u_2,x_2,w_1\}$. However, this would make $y_2 \in N(x_1)$, contradicting $d^2(x_1) = 3 + 3 + 4$.

Since all cases when $e(\{x_2,y_2\},w_1) = 1$ lead to contradictions we must have that $e(\{x_2,y_2\},w_1) = 2$ and analogously we get $e(\{x_1,y_1\},w_2) = 2$. The situation is therefore as in Figure \ref{lnfig3}.
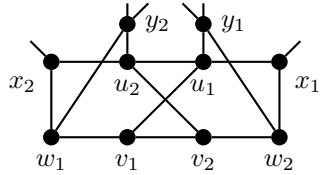
\begin{figure}[h]
\begin{center}
\begin{tikzpicture}[scale=1]
  \GraphInit[vstyle=Classic]
  \renewcommand*{\VertexSmallMinSize}{4pt}
  \Vertex[Math,x=0,y=0,Lpos=-90]{w_1}
  \Vertex[Math,x=1,y=0,Lpos=-90]{v_1}
  \Vertex[Math,x=2,y=0,Lpos=-90]{v_2}
  \Vertex[Math,x=3,y=0,Lpos=-90]{w_2}
  \Vertex[Math,x=1,y=1,Lpos=-90]{u_2}
  \Vertex[Math,x=2,y=1,Lpos=-90]{u_1}
  \Vertex[Math,x=2,y=1.5]{y_1}
  \Vertex[Math,x=3,y=1,Lpos=-45]{x_1}
  \Vertex[Math,x=0,y=1,Lpos=-135]{x_2}
  \Vertex[Math,x=1,y=1.5]{y_2}
  \Vertex[x=-0.3,y=1.3,empty]{x22}
  \Vertex[x=3.3,y=1.3,empty]{x12}
  \Vertex[x=1,y=1.8,empty]{y21}
  \Vertex[x=1.3,y=1.8,empty]{y23}
  \Vertex[x=2,y=1.8,empty]{y11}
  \Vertex[x=1.7,y=1.8,empty]{y12}
  \Edges(w_1,v_1,u_1,u_2,v_2,w_2)
  \Edge(v_1)(v_2)
  \Edge(u_1)(x_1)
  \Edge(u_1)(y_1)
  \Edge(u_2)(x_2)
  \Edge(u_2)(y_2)
  \Edge(w_1)(x_2)
  \Edge(w_2)(x_1)
  \Edge(w_2)(y_1)
  \Edge(x_1)(x12)
  \Edge(y_2)(w_1)
  \Edge(x_2)(x22)
  \Edge(y_2)(y21)
  \Edge(y_2)(y23)
  \Edge(y_1)(y11)
  \Edge(y_1)(y12)
\end{tikzpicture}
\end{center}
\caption{The local structure around the trivalent vertex $v_1$ in $H$.}
\label{lnfig3}
\end{figure}

Note that in particular in this situation we have that $\N(C_4;N(v_1)) = 3$ and $\N(C_4; u_2) \geq 2$. This means that we get $\nu(H_{v_1}) \leq \nu(H) + 3 - (3) \leq 0$. If $H_{v_1}$ were disconnected it would have to contain either at least eight bivalent vertices, or at least four bivalent vertices and a 3-stable component, in either case contradicting $|N_2(v_2)| = 6$.

Since $(x_1,w_2)$ and $(x_2,u_2)$ are pairs of adjacent bivalent vertices this makes $H$ being formed from some $Ch_k$-component by adding the vertices $v_1,v_2,w_1$ and $u_2$ and edges as illustrated in Figure \ref{lnfig4}.
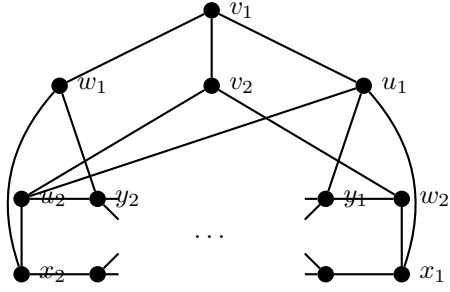
\begin{figure}[h]
\begin{center}
\begin{tikzpicture}[scale=1]
  \GraphInit[vstyle=Classic]
  \renewcommand*{\VertexSmallMinSize}{4pt}
  \Vertex[Math,x=0,y=0]{x_2}
  \Vertex[Math,x=1,y=0,NoLabel]{a}
  \Vertex[x=1.3,y=0,empty]{a2}
  \Vertex[x=1.3,y=0.3,empty]{a1}
  \Vertex[Math,x=4,y=0,NoLabel]{b}
  \Vertex[x=3.7,y=0,empty]{b2}
  \Vertex[x=3.7,y=0.3,empty]{b1}
  \Vertex[Math,x=5,y=0]{x_1}
  \Vertex[Math,x=0,y=1]{u_2}
  \Vertex[Math,x=1,y=1]{y_2}
  \Vertex[x=1.3,y=1,empty]{y21}
  \Vertex[x=1.3,y=0.7,empty]{y22}
  \Vertex[Math,x=4,y=1]{y_1}
  \Vertex[x=3.7,y=1,empty]{y11}
  \Vertex[x=3.7,y=0.7,empty]{y12}
  \Vertex[Math,x=5,y=1]{w_2}
  \Vertex[Math,x=0.5,y=2.5]{w_1}
  \Vertex[Math,x=2.5,y=2.5]{v_2}
  \Vertex[Math,x=4.5,y=2.5]{u_1}
  \Vertex[Math,x=2.5,y=3.5]{v_1}
  \Edges(a1,a,x_2,u_2,y_2,w_1,v_1,u_1,y_1,w_2,x_1,b,b1)
  \Edge(a2)(a)
  \Edge(b2)(b)
  \Edge(v_1)(v_2)
  \Edge(v_2)(u_2)
  \Edge(v_2)(w_2)
  \Edges(y21,y_2,y22)
  \Edges(y11,y_1,y12)
  \Edge(u_1)(u_2)
  \Edge[style={bend right}](w_1)(x_2)
  \Edge[style={bend left}](u_1)(x_1)
  \node at (2.5,0.5) {$\dots$};
\end{tikzpicture}
\end{center}
\caption{How $v_1$ and $v_2$'s neighbourhood looks in relation to the $Ch_k$-component in $H_{v_1}$.}
\label{lnfig4}
\end{figure}

It is easily seen that $k \geq 4$ since $H$ should be triangle-free and $\N(C_4;N(v_1)) = 3$. This makes $H \iso BC_k$ for some $k \geq 5$ since $(BC_k)_v \iso Ch_{k-1}$ where $v$ is any trivalent vertex (connected by edges as indicated in Figure \ref{lnfig4}).
\end{proof}

We will now show that connected subgraphs of $G$ with $\nu$-value zero and minimum valency three either have no cycles of length four or are isomorphic to $BC_k$s. We begin by showing that there cannot be cycles of length four through vertices of low valency. We later use this to exclude such cycles altogether. We begin with neighbourhoods of vertices of valency three.

\begin{property}
 \label{classprop3}
If $H \leq G$ is connected, $\nu(H) \leq 0$, $\delta(H) = 3$ and $H \not \iso BC_k$ for all $k \geq 5$, then $\N(C_4;N(v)) = 0$ for all trivalent vertices $v$ of $H$.
\end{property}
\begin{proof}
Suppose otherwise, i.e. there is some trivalent vertex $v \in V(H)$ such that $\N(C_4;N(v)) > 0$. By Property \ref{classprop2} the vertex $v$ must have at least one non-trivalent neighbour. Hence, $d^2(v) \geq 10$, and clearly $d^2(v) \leq 10$ since otherwise $\nu(H_v) \leq \nu(H) + 0 - (1) \leq -1$, contradicting assumption \eqref{bigass}.

We therefore have $\nu(H_v) \leq \nu(H) + 3 - (1) \leq 2$. Let $w_1,w_2$ denote the two trivalent neighbours of $v$. If we had that $d^2(w_i) = 11$ for both $i \in [2]$ then we would get $\nu(H_{w_1}) \leq \nu(H) + 0 \leq 0$ and $d(H_{w_1}; w_2) = 2$. Thus, $w_2$ would belong to a $Ch_k$-component of $H_{w_1}$ for some $k \geq 2$. It is clear that $k \neq 2$ since otherwise the two tetravalent neighbours of $w_1$ would have to have two common neighbours with $w_2$. On the other hand, for $k \geq 3$ we have that $\N(C_4; N(w_1)) \geq 1$ and we would then get $\nu(H_{w_1}) \leq -1$, contradicting assumption \eqref{bigass}.

Hence, not both $w_1$ and $w_2$ can have second valency eleven. There is therefore bivalent vertices in $H_v$. If there is only one bivalent then, by Property \ref{classprop1}, we must have that $SCh_1$ is a component of $H_v$. But $SCh_1$ contains a cycle of length four containing only trivalent vertices. All four of these vertices must have valency at least four in $H$ by Property \ref{classprop2}. There are only three neighbours of $v$ in $H$, so at least one of them is adjacent to two of the trivalent vertices in the cycle of length four. But then we get $\N(C_4;N(v)) \geq 2$, which would give $\nu(H_v) \leq 1$. However $\nu(SCh_1) = 2$ so we would have to have another component of $H_v$ with negative $\nu$-value, contradicting assumption \eqref{bigass}.

Thus there are at least two bivalent vertices of $H_v$ and therefore, by Property \ref{classprop1}, we must have either a $C_5$ in $\Cc(H_v)$ or at least eight vertices with different valency in $H$ and $H_v$. The latter is impossible since $|N_2(v)| = 7$. The former is impossible since we would get that at least one of $w_1$ and $w_2$ has two neighbours in a $C_5$-component. This would yield a 4-cycle through a trivalent vertex, contradicting Property \ref{classprop2}.
\end{proof}

\begin{property}
 \label{classprop4}
 If $H \leq G$ is connected, $\nu(H) \leq 0$, $\delta(H) \geq 3$, $H \not \iso BC_k$ for all $k \geq 5$, then $\N(C_4;v) = 0$ for all tetravalent vertices $v$ of $H$.
\end{property}
\begin{proof}
  Suppose, to the contrary, that there is a cycle, $C$, of length four in $H$ containing a tetravalent vertex $v$. Note that all the vertices in $C$ have valency four since otherwise there would be a tetravalent vertex $u \in V(C)$ with $d^2(u) \geq 17$ by Property \ref{classprop3}. This is not possible since then $\nu(H_u) \leq \nu(H) - 1 - (1) \leq -2$, contradicting assumption \eqref{bigass}.

Hence, $d^2(v) = 16$ and therefore $\nu(H_v) \leq \nu(H) + 2 - (1) \leq 1$. Let $x \in C$ be the vertex at distance two from $v$ in $C$. Since $x$ has two common neighbours with $v$ we have $d(H_v;x) = 2$. One of the two vertices in $N(x) \setminus V(C)$ has two neighbours in $N(x) \setminus V(C)$, and the other has at least one. This follows from $d^2(H_v;x) \leq 5$ (by Property \ref{containsChk1}) and from $H$ being triangle-free. Thus $\N(C_4; N(v)) \geq 3$, which implies that $\nu(H_v) \leq \nu(H) + 2 - (3) \leq -1$, contradicting assumption \eqref{bigass}.
\end{proof}

\begin{property}
 \label{classprop5}
If $H \leq G$ is connected, $\nu(H) \leq 0$ then $\N(C_4;v) = 0$ for all pentavalent vertices $v$ of $H$.
\end{property}
\begin{proof}
Suppose otherwise, i.e. that there is a cycle of length four, $C$ on $\{c_1,c_2,c_3,c_4\}$ where $c_1c_3 \notin E(H)$, containing a pentavalent vertex, say $c_1$. Since $BC_k,Ch_\ell$ contains no pentavalent vertices we have $H \not \iso BC_k$ and $H \not \iso Ch_\ell$ for all $k \geq 5$ and all $\ell \geq 2$.

Clearly $\delta(H) \geq 3$ then by Properties \ref{prop-mindeg2} and \ref{containsChk1}. All vertices in $C$ are at least pentavalent by Property \ref{classprop4}. Since $c_1$ has no trivalent neighbours (by Property \ref{classprop3}) and at least two neighbours of valency at least five we have $d^2(v) \geq 2\cdot 5 + 3 \cdot 4 = 22$. This gives $\nu(H_v) \leq \nu(H) + 1 - (1) \leq 0$. Hence, in particular, all the vertices of $C$ are pentavalent with three tetravalent and two pentavalent neighbours.

Let $M_i = N(c_i) \setminus V(C)$ for $i \in [4]$. We have that $e(M_i,M_{i+2}) \geq 2$ for $i \in [2]$ since otherwise $c_{i+2}$ would be trivalent with second valency at least eleven in $H_{c_i}$. We would moreover have that $\nu(H_{c_i}) \leq 0$. Thus by assumption \eqref{bigass} we get that all components of $H_{c_i}$ are in $\mathcal{G}$, but in all of the graphs of $\mathcal{G}$ the trivalent vertices all have second valency at most ten.

Each of the vertices in $M_i$ (for $i \in [4]$) must have at least one trivalent neighbour, since otherwise there would be some vertex $m \in M_i$ such that $d^2(m) \geq 17$. This gives us a contradiction to assumption \eqref{bigass} by considering $H_m$. Thus, a fortiori, $H_{c_1}$ contains some bivalent vertices, and therefore at least four bivalent vertices by assumption \eqref{bigass}. Note also that $|N_2(c_1)| \leq 16$.

All trivalent vertices, except possibly $c_3$, in $H_{c_1}$ must belong to $W_5$- or $(2C_7)_{2i}$-components since otherwise there would be a cycle of length four through such a vertex. That vertex would then have to be at least pentavalent in $H$ by Properties \ref{classprop2} and \ref{classprop4}. This would yield $\N(C_4;N(c_1)) \geq 2$, whence $\nu(H_{c_1}) \leq -1$, contradicting assumption \eqref{bigass}.

Hence, $\Cc(H_{c_1})$ consists of $C_5$s, $W_5$s and $(2C_7)_{2i}$s, at least one of which is a $C_5$. Now, since $d(H_{c_1};x) = 3$ for all $x \in M_2 \cup M_4 \cup \{c_3\}$, at least one of the vertices in $M_1$ must have two neighbours in the $C_5$-component. This gives us that $\N(C_4; N(c_1)) \geq 2$, and therefore $\nu(H_{c_1}) \leq -1$, contradicting assumption \eqref{bigass}.
\end{proof}

We can now conclude this study of subgraphs with $\nu$-value zero and cycles of length four by the following property. It claims that the only way to have cycles of length four in a graph with $\nu$-value zero is to be one of the known graphs with 4-cycles in $\mathcal{G}$.

\begin{property}
 \label{classprop6}
If $H \leq G$ is connected, $\nu(H) \leq 0$ then exactly one of the following three statements holds.
\begin{enumerate}[(i)]
 \item \label{i1} $H \iso Ch_k$ for some $k \geq 2$.
 \item \label{i2} $H \iso BC_k$ for some $k \geq 5$.
 \item \label{i3} $\N(C_4; H) = 0$ and $H \not \iso C_5$.
\end{enumerate}
\end{property}
\begin{proof}
Suppose that neither (\ref{i1}) nor (\ref{i2}) is true. Then we must have $\delta(H) \geq 3$ by Properties \ref{prop-mindeg2} and \ref{containsChk1}.

We will prove that no cycle of length four contains a vertex of valency $d$ for all $d \geq 3$ by induction on $d$. This holds for $d = 3, 4, 5$ by Properties \ref{classprop2}, \ref{classprop4} and \ref{classprop5}. Suppose therefore that $d \geq 6$ and that there are no cycles of length four through any vertices with valency less than $d$.

Now, suppose that $v \in V(H)$ has valency $d$ and is such that $\N(C_4; v) > 0$. Then $d^2(v) \geq 2 \cdot d + (d - 2) \cdot 4$ since its two neighbours in the cycle of length four would have valency at least $d$ and by Property \ref{classprop3} the remaining $d-2$ neighbours have valency at least four. Then by Property \ref{prop-removalcount} we get $\nu(H_v) \leq \nu(H) - 3d^2(v) + 17d - 18 - (1) \leq \nu(H) - d + 5 \leq -1$, contradicting assumption \eqref{bigass}. The conclusion now follows by induction.
\end{proof}

\begin{property}
  \label{classprop7}
  If $H \leq G$ is connected, $\nu(H) = 3$, $\alpha(H_2) > 1$ and $\N(C_4; H) = 0$, then there is an edge $e \in E(H)$ such that $H - e = C_5 + H'$, where $H' \in \{C_5,W_5,(2C_7)_{2i}\}$.
\end{property}
\begin{proof}
Suppose first that $H$ is not edge-critical, then there is an edge $e \in E(H)$ such that $\nu(H - e) = 0$. $H - e$ contains vertices that are at most bivalent so it contains a $C_5$-component. It can clearly not be the only component. Hence $H - e = C_5 + H'$ where $\nu(H') = 0$, and the conclusion follows from assumption \eqref{bigass}.

On the other hand, if $H$ is edge-critical we get the following. Let $v \in V(H)$ be some bivalent vertex. We must have that $d^2(v) \leq 5$, otherwise $\delta(H_v) \leq 2$ (since $\alpha(H_2) > 1$) and $|\Cc(H_v)| = 1$ (by Lemma \ref{ecbvconn}) which would imply that $H_v \iso C_5$, because $\nu(H_v) \leq 1$. We would then get several cycles of length four through the neighbourhood of $v$, yielding a contradiction to assumption \eqref{bigass}.

Since $H$ is not 2-regular there is a bivalent vertex $v_1$ in $H$ with a trivalent neighbour. Suppose the bivalent neighbour of $v_1$, say $v_2$, does not have second valency five. Let $\{u\} = N(v_2) \setminus \{v_1\}$, $\{t_1\} = N(v_1) \setminus \{v_2\}$ and $\{u',t_2\} = N(t_1) \setminus \{v_1\}$. See Figure \ref{cp7p1} for an illustration of the situation.
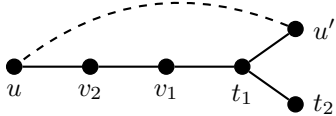
\begin{figure}[h]
\begin{center}
\begin{tikzpicture}[scale=1]
  \GraphInit[vstyle=Classic]
  \renewcommand*{\VertexSmallMinSize}{4pt}
  \Vertex[Math,x=0,y=0,Lpos=-90]{u}
  \Vertex[Math,x=1,y=0,Lpos=-90]{v_2}
  \Vertex[Math,x=2,y=0,Lpos=-90]{v_1}
  \Vertex[Math,x=3,y=0,Lpos=-90]{t_1}
  \Vertex[Math,x=3.7,y=-0.5]{t_2}
  \Vertex[Math,x=3.7,y=0.5]{u'}
  \Edges(u,v_2,v_1,t_1,u')
  \Edge(t_1)(t_2)
  \Edge[style={dashed,bend right}](u')(u)
\end{tikzpicture}
\end{center}
\caption{The situation in the neighbourhood of $v_1$ and $v_2$ in the case that $v_2$ does not have second valency 5.}
\label{cp7p1}
\end{figure}

Now, $u$ belongs to a $P_2$-component of $H_{v_1}$ by Property 9, so the edge $t_1t_2 \in E(H)$ would be redundant since the component, $C$, containing $t_2$ in $H_{v_1}$ would have $\nu(H) = 0$ and thus be 2-stable by Property \ref{prop1}. This contradicts the edge-criticality of $H$.

Hence, we have that $H_2$ only consists of $P_2$-components. Furthermore $H_2$ has at least two such components since $\alpha(H_2) > 1$. Say that $a_1$ and $a_2$ forms another such component and give names to the vertices in their neighbourhoods according to Figure \ref{cp7p2}. Note that the vertices $x_i$ and the vertices $c_i$ need not necessarily be distinct.
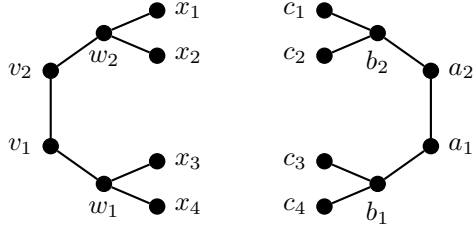
\begin{figure}[h]
\begin{center}
\begin{tikzpicture}[scale=1]
  \GraphInit[vstyle=Classic]
  \renewcommand*{\VertexSmallMinSize}{4pt}
  \Vertex[Math,x=0,y=0,Lpos=-180]{v_1}
  \Vertex[Math,x=0,y=1,Lpos=-180]{v_2}
  \Vertex[Math,x=0.7,y=-0.5,Lpos=-90]{w_1}
  \Vertex[Math,x=0.7,y=1.5,Lpos=-90]{w_2}
  \Vertex[Math,x=1.4,y=1.8]{x_1}
  \Vertex[Math,x=1.4,y=1.2]{x_2}
  \Vertex[Math,x=1.4,y=-0.2]{x_3}
  \Vertex[Math,x=1.4,y=-0.8]{x_4}
  \Edges(x_4,w_1,v_1,v_2,w_2,x_1)
  \Edge(w_2)(x_2)
  \Edge(w_1)(x_3)
  \Vertex[Math,x=5,y=0]{a_1}
  \Vertex[Math,x=5,y=1]{a_2}
  \Vertex[Math,x=4.3,y=-0.5,Lpos=-90]{b_1}
  \Vertex[Math,x=4.3,y=1.5,Lpos=-90]{b_2}
  \Vertex[Math,x=3.6,y=1.8,Lpos=-180]{c_1}
  \Vertex[Math,x=3.6,y=1.2,Lpos=-180]{c_2}
  \Vertex[Math,x=3.6,y=-0.2,Lpos=-180]{c_3}
  \Vertex[Math,x=3.6,y=-0.8,Lpos=-180]{c_4}
  \Edges(c_4,b_1,a_1,a_2,b_2,c_1)
  \Edge(b_2)(c_2)
  \Edge(b_1)(c_3)
\end{tikzpicture}
\end{center}
\caption{Names of the vertices in the $H$-neighbourhoods of the two $P_2$-components of $H_3$.}
\label{cp7p2}
\end{figure}

If $d(x_1) = 2$ then $d(H_{v_2}; x_1) \leq 1$ so $x_1$ would belong to a $P_2$-component of $H_{v_2}$, while $x_2$ and $w_1$ would belong to components that have $\nu$-value zero and therefore belong to $\mathcal{G}$ by assumption \eqref{bigass}. This is not possible however, since by Property \ref{prop1} these components are 2-stable.

Hence,
\begin{equation}
\label{newstar}
d(x_i),d(c_i) \geq 3 \quad (\forall i \in [4]).
\end{equation}

Suppose that $d^2(H_{v_1}; w_2) \geq 6$. Then $\nu(H_{v_1,w_2}) \leq 2$ so $a_1$ and $a_2$ belong to a $C_5$-component, say $C$, of $H_{v_1,w_2}$. All vertices of $C$ except for $a_1$ and $a_2$ have valency at least three in $H$ and so each of the vertices $w_1$, $x_1$ and $x_2$ must have a neighbour in $V(C)$. This leaves three other edges in $E(N[\{v_1,w_2\}], V(H) \setminus N[\{v_1,w_2\}])$, so if $C' \in \Cc(H_{v_1,w_2}) \setminus \{C\}$ we must have that $\delta(C') \leq 2$ or otherwise $C'$ would be 3-stable by Property \ref{prop3}. But then also $C' \iso C_5$, whence $H_{v_1,w_2} \iso C_5 + C_5$, which gives $\nu(H_{v_1,w_2}) = 0$. This would however require that $\N(C_4;N(v_1) \cup N(w_2)) \geq 2$, which contradicts $\N(C_4;H) = 0$.

Hence, $d^2(H_{v_1}; w_2) \leq 5$ and analogously $d^2(H_{v_2}; w_1) \leq 5$.

Now, by \eqref{newstar} there must therefore be a cycle of length five through $v_1$ and $v_2$. Analogously we get one through $a_1$ and $a_2$. We may therefore assume that the situation looks like illustrated in Figure \ref{cp7p3}.
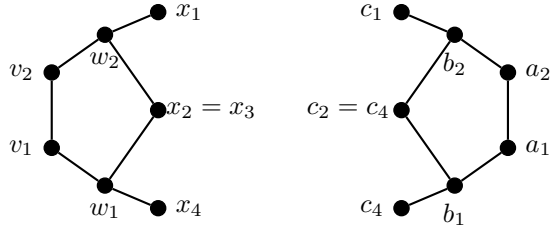
\begin{figure}[h]
\begin{center}
\begin{tikzpicture}[scale=1]
  \GraphInit[vstyle=Classic]
  \renewcommand*{\VertexSmallMinSize}{4pt}
  \Vertex[Math,x=0,y=0,Lpos=-180]{v_1}
  \Vertex[Math,x=0,y=1,Lpos=-180]{v_2}
  \Vertex[Math,x=0.7,y=-0.5,Lpos=-90]{w_1}
  \Vertex[Math,x=0.7,y=1.5,Lpos=-90]{w_2}
  \Vertex[Math,x=1.4,y=1.8]{x_1}
  \Vertex[Math,x=1.4,y=0.5,NoLabel]{x_2}
  \Vertex[Math,x=1.4,y=-0.8]{x_4}
  \Edges(x_4,w_1,v_1,v_2,w_2,x_1)
  \Edge(w_2)(x_2)
  \Edge(w_1)(x_2)
  \Vertex[Math,x=6,y=0]{a_1}
  \Vertex[Math,x=6,y=1]{a_2}
  \Vertex[Math,x=5.3,y=-0.5,Lpos=-90]{b_1}
  \Vertex[Math,x=5.3,y=1.5,Lpos=-90]{b_2}
  \Vertex[Math,x=4.6,y=1.8,Lpos=-180]{c_1}
  \Vertex[Math,x=4.6,y=0.5,NoLabel]{c_2}
  \Vertex[Math,x=4.6,y=-0.8,Lpos=-180]{c_4}
  \Edges(c_4,b_1,a_1,a_2,b_2,c_1)
  \Edge(b_2)(c_2)
  \Edge(b_1)(c_2)
  \node at (2.1,0.5){$x_2 = x_3$};
  \node at (3.9,0.5){$c_2 = c_4$};
\end{tikzpicture}
\end{center}
\caption{The $H$-neighbourhoods of the two $P_2$-components of $H_3$.}
\label{cp7p3}
\end{figure}

Clearly $x_1 \neq x_4$ since we have no cycles of length four. We also have $d(x_1) = 3$ since $d^2(H_{v_1}; w_2) \leq 5$. Analogously we have $d(x_4) = d(c_1) = d(c_4) = 3$. In a similar manner we see that $d(x_2) = d(c_2) = 3$.

We have that $d^2(x_1) \leq 9$ since otherwise $\nu(H_{x_1,v_2}) \leq \nu(H) + 3 - 7 \leq -1$, contradicting assumption \eqref{bigass}. Suppose that $d^2(x_1) = 9$. We then instead get that $\nu(H_{x_1,v_2}) = \nu(H) + 6 - 7 = 2$. The vertices $a_1$ and $a_2$ are at most bivalent in $H_{x_1,v_2}$ and must therefore belong to some $C_5$-component, $C$, of $H_{x_1,v_2}$. The same goes for the vertex $w_1$ as well. If there are two $C_5$-components in $H_{x_1,v_2}$ then these are the only two components since $e(N[\{x_1,v_2\}], V(H) \setminus N[\{x_1,v_2\}]) \leq 6$. But then $\nu(H_{x_1,v_2}) = 0$, which would give us that $\N(C_4; N(a) \cup N(v_2)) = 2$, contradicting that $H$ contains no cycles of length four. Thus $a_1,a_2$ and $w_1$ must belong to the same $C_5$-component, i.e. we must have the situation illustrated in Figure \ref{cp7p4} (note in particular that $x_1$ and $x_4$ must have a common neighbour).

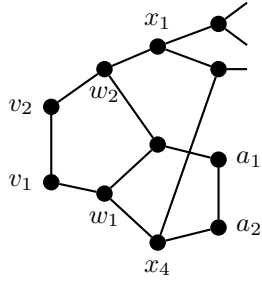
\begin{figure}[h]
\begin{center}
\begin{tikzpicture}[scale=1]
  \GraphInit[vstyle=Classic]
  \renewcommand*{\VertexSmallMinSize}{4pt}
  \Vertex[Math,x=0,y=0,Lpos=-180]{v_1}
  \Vertex[Math,x=0,y=1,Lpos=-180]{v_2}
  \Vertex[Math,x=0.7,y=-0.15,Lpos=-90]{w_1}
  \Vertex[Math,x=0.7,y=1.5,Lpos=-90]{w_2}
  \Vertex[Math,x=1.4,y=1.8,Lpos=90]{x_1}
  \Vertex[Math,x=2.2,y=2.1,NoLabel]{x11}
  \Vertex[Math,x=2.6,y=2.4,empty]{x111}
  \Vertex[Math,x=2.6,y=1.8,empty]{x112}
  \Vertex[Math,x=2.2,y=1.5,NoLabel]{x12}
  \Vertex[Math,x=2.6,y=1.5,empty]{x121}
  \Vertex[Math,x=1.4,y=0.5,NoLabel]{x_2}
  \Vertex[Math,x=1.4,y=-0.8,Lpos=-90]{x_4}
  \Edges(x_4,w_1,v_1,v_2,w_2,x_1)
  \Edge(w_2)(x_2)
  \Edge(w_1)(x_2)
  \Vertex[Math,x=2.2,y=0.3]{a_1}
  \Vertex[Math,x=2.2,y=-0.6]{a_2}
  \Edges(x_2,a_1,a_2,x_4)
  \Edge(x_1)(x11)
  \Edge(x_1)(x12)
  \Edge(x_4)(x12)
  \Edges(x111,x11,x112)
  \Edge(x12)(x121)
\end{tikzpicture}
\end{center}
\caption{The neighbourhood of $v_1,v_2,a_1,a_2$ in the case that $d^2(x_1) = 9$.}
\label{cp7p4}
\end{figure}

The three remaining edges in $E(N[\{x_1,v_2\}], V(H) \setminus N[\{x_1,v_2\}])$ must go to a component $C'$ such that $\delta(C') \leq 2$ by edge-criticality and Property \ref{prop3}. But then $C' \iso C_5$ and we would have two $C_5$-components, which we already have seen to be impossible.

Hence, $d^2(x_1) \leq 8$. This means in particular that $x_1$ has some bivalent neighbour, $t_1$, which in turn must have a bivalent neighbour $t_2$, in $H$. $t_1$ and $t_2$ are both monovalent in $H_{v_1,w_2}$ (by Property \ref{prop-leq6}) since $\nu(H_{v_1,w_2}) = \nu(H) + 1 + 1 = 5$. Clearly $t_2w_1 \notin E(H)$ so we must have that $t_2x_2 \in E(H)$. Since $d(x_2) = 3$ we get that $H_{v_1} \setminus \{t_1,t_2\}$ is a graph with $\nu$-value one. There are then only two edges in $E(T,H \setminus T)$ where $T = \{v_1,v_2,w_1,w_2,x_2,x_1,t_1,t_2\}$. This is not possible however since $H_{v_1} \setminus \{t_1,t_2\}$ is 2-stable (by Property \ref{prop1}).
\end{proof}

\begin{corollary}
\label{classcor7}
Let $H \leq G$ be connected, $\nu(H) \leq 0$ and $\N(C_4; H) = 0$. If $v \in V(H)$ is trivalent with $d^2(v) \geq 10$, then either $\delta(H_v) \geq 3$ or $\alpha((H_v)_2) = 1$.
\end{corollary}
\begin{proof}
  If $d^2(v) \geq 11$, then $\nu(H_v) \leq \nu(H) + 0 \leq 0$. By assumption \eqref{bigass} we would then get that either $\delta(H_v) \geq 3$ or $H_v$ contains $C_5$-components. It is however easily seen that $H_v$ can not contain $C_5$-components since $\N(C_4;H) = 0$.
  
  Hence, $d^2(v) = 10$.
  We must have that $H_v$ contains a least one bivalent vertex. If $\nu(H) < 0$ then $\nu(H_v) \leq 2$ and the conclusion follows form Property \ref{classprop1}. On the other hand if $\nu(H) = 0$, then $\nu(H) = 3$ and by Property \ref{classprop7} $H_v$ has at least four bivalent vertices in a cycle of length five if $\alpha((H_v)_2) > 1$. But if this were the case then one of $v$'s three neighbours would have to be adjacent to two vertices in a cycle of length five. This would contradict that $\N(C_4; H) = 0$.
\end{proof}

\begin{property}
\label{classprop8}
Suppose that $G$ is connected, $\nu(G) \leq 0$ and $\delta(G) = 3$. If $G \notin \{BC_k; k \geq 5\} \cup \{W_5, (2C_7)_{2i}\}$ then $G_3$ is 2-regular.
\end{property}
\begin{proof}
  Firstly, $\N(C_4; G) = 0$ by Proposition \ref{classprop6}. $G$ is not 3-regular by Property \ref{no3regular}.
  Note that $\delta(G_3) \geq 1$ since otherwise there would be a trivalent vertex $v$ in $G$ with second valency at least twelve. This would however give $\nu(G_v) \leq \nu(G) - 3 \leq -3$, contradicting assumption \eqref{bigass}.

  Now, suppose that $v \in V(G_3)$ is monovalent in $G_3$, then we instead get $\nu(G_v) \leq 0$. By assumption \eqref{bigass} we therefore have that all components of $G_v$ are in $\{C_5,W_5,$ $(2C_7)_{2i}\}$. We can easily see that it is not possible to have $C_5$-components in $G_v$ since then one of $v$'s neighbours would have to be adjacent to two vertices in that $C_5$-component, which would give a cycle of length four. So all components of $G_v$ are isomorphic to one of the two 3-regular components $W_5$ and $(2C_7)_{2i}$. The trivalent neighbour, $u$, of $v$ in $G$ would then have two tetravalent neighbours. Let $w_{11}$, $w_{12}$ denote the two tetravalent neighbours of $v$, and $w_{21}$, $w_{22}$ the two tetravalent neighbours of $u$. Analogously as for $G_v$ we can show that all components of $G_u$ must belong to $\{W_5,(2C_7)_{2i}\}$. This means, in particular, that the neighbours of $w_{ij}$ that are not $u$ or $v$ must be tetravalent in $G$ for all $i,j \in [2]$. But then we must have that all three vertices in $N(w_{11}) \setminus \{v\}$ are adjacent to $w_{21}$ or $w_{22}$. This would however yield a cycle of length four through $w_{11}$, contradicting $\N(C_4;G) = 0$.
  
  Hence we must have $\delta(G_3) \geq 2$. Suppose now that there is a bivalent vertex $v$ in $G_3$ with a trivalent neighbour, $u$. By Corollary \ref{classcor7} we must have that $\alpha((G_v)_2) = 1$, which would make the two neighbours of $u$ that are not $v$ adjacent, contradicting that $G$ is triangle-free.
\end{proof}

\begin{property}
  \label{classprop9}
  If $G$ is connected, $\nu(G) \leq 0$, $\N(C_4;G) = 0$, $\delta(G) = 3 \neq \Delta(G)$, then $G_3 \iso \left(\frac{|V(G_3)|}{5}\right)\cdot C_5$.
\end{property}
\begin{proof}
  By Property \ref{classprop8} the induced graph $G_3$ consists of 2-regular components. Clearly $G_3$ does not contain any $C_4$-components. If $G_3$ were to contain a cycle of length six or more, then let $v$ be a vertex of that cycle. We would then have $\delta(G_v) \leq 2$ with $\alpha((G_v)_2) > 1$, contradicting Corollary \ref{classcor7}.
\end{proof}

\begin{property}
  \label{classprop10}
  Suppose that $G$ is connected and $\nu(G) \leq 0$. If $G \notin \mathcal{G}$ then $G$ is 4-regular.
\end{property}
\begin{proof}
  Note that $\N(C_4; G) = 0$ by Property \ref{classprop6}. We have that $\delta(G) \geq 3$ by Property \ref{prop-mindeg2}. Moreover, we can conclude that $\Delta(G) \neq 3$ from Property \ref{no3regular}. Thus, Property \ref{classprop9} implies that $G_3$ consists only of $C_5$-components (possibly none).

  Suppose that $V(G_3)$ is non-empty, i.e. that $\delta(G) = 3$. We then have a cycle of length five on the vertices $v_1,v_2,v_3,v_4,v_5$ in $V(G_3)$. The neighbours of these vertices that are not in the cycle, say $w_i$ is adjacent to $v_i$, must all be tetravalent in $G$. The situation is therefore as illustrated in Figure \ref{cp10p1}.

\begin{figure}[h]
\begin{center}
\begin{tikzpicture}[scale=1]
  \GraphInit[vstyle=Classic]
  \renewcommand*{\VertexSmallMinSize}{4pt}
  \Vertices[Math,unit=1,Lpos=-90,Ldist=-1]{circle}{v_1,v_2,v_3,v_4,v_5}
  \Vertices[Math,unit=2,Lpos=90]{circle}{w_1,w_2,w_3,w_4,w_5}
  \begin{scope}[rotate=-5]
    \Vertices[Math,unit=2.4,empty]{circle}{u35,u7,u14,u21,u28}
  \end{scope}
  \begin{scope}[rotate=0]
    \Vertices[Math,unit=2.4,empty]{circle}{u1,u8,u15,u22,u29}
  \end{scope}
  \begin{scope}[rotate=5]
    \Vertices[Math,unit=2.4,empty]{circle}{u2,u9,u16,u23,u30}
  \end{scope}
  \Edges(v_1,v_2,v_3,v_4,v_5,v_1)
  \Edge(v_1)(w_1)
  \Edge(v_2)(w_2)
  \Edge(v_3)(w_3)
  \Edge(v_4)(w_4)
  \Edge(v_5)(w_5)
  \Edge(w_1)(u35)
  \Edge(w_1)(u1)
  \Edge(w_1)(u2)
  \Edge(w_2)(u7)
  \Edge(w_2)(u8)
  \Edge(w_2)(u9)
  \Edge(w_3)(u14)
  \Edge(w_3)(u15)
  \Edge(w_3)(u16)
  \Edge(w_4)(u21)
  \Edge(w_4)(u22)
  \Edge(w_4)(u23)
  \Edge(w_5)(u28)
  \Edge(w_5)(u29)
  \Edge(w_5)(u30)
\end{tikzpicture}
\end{center}
\caption{The neighbourhood of the $C_5$ from $G_3$ in $G$.}
\label{cp10p1}
\end{figure}
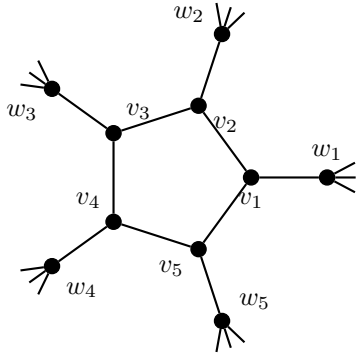

Clearly, $\dist(w_i,w_{i+1}) \geq 2$ for all $i \in [5]$ (taking indices modulo 5). On the other hand if $\dist(w_1,w_2) = 3$, then consider the graph $G' := G \setminus (\{v_i; i \in [5]\} \cup \{w_4\}) + w_1w_2$. We must have that $G'$ is triangle-free, $n(G') = n(G) - 6$, $e(G') = e(G) - 12$ and $\N(C_4;G') \leq 3$ since $e(N(w_1)\setminus v_1, N(w_2) \setminus v_2) \leq 3$. Moreover, $\alpha(G') < \alpha(G)$ since any independent set of $G'$ contains at most one of $w_1$ and $w_2$ and may therefore be extended to an independent set of size one more in $G$.

Similarly, any independent set of size $\alpha(G) - 1$ in $G \setminus \{v_i; i \in [5]\}$ must contain at least three consecutive $w_j$'s, or could otherwise be extended to an independent set of size $\alpha(G) + 1$ in $G$. The same goes for $G'' := G \setminus \{v_i; i \in [5]\} + w_1w_2$ since $\alpha(G'') \leq \alpha(G \setminus \{v_i; i \in [5]\})$. So in particular any independent set in $G''$ of size $\alpha(G) - 1$ must contain $w_4$, since it contains at most one of $w_1$ and $w_2$. Hence, $\nu(G') \leq \alpha(G) - 2$, which gives $\nu(G') \leq \nu(G) - 3\cdot 12 + 17 \cdot 6 - 35 \cdot 2 + (3) \leq -1$, contradicting assumption \eqref{bigass}.

Therefore, $\dist(w_1,w_2) = 2$ and by analogous arguments we have $\dist(w_i,w_{i+1}) = 2$ for all $i \in [5]$.

Suppose that $w_1w_3 \notin E(G)$ and that $\delta(G_{v_1,v_3}) \leq 2$. Then we would have $\nu(G_{v_1,v_3}) \leq \nu(G) + 3 - 2 \leq 1$, so $G_{v_1,v_3}$ contains a $C_5$-component. $w_1$ and $w_3$ are adjacent to at most one vertex in such a component, each. The remaining three vertices must then be $w_2,w_4$ and $w_5$. These vertices, however, are at least tetravalent in $G$, so all of them would have to be adjacent to $w_1$ or $w_2$, which would give a cycle of length three or four. This, and analogous arguments, gives us that $w_iw_{i+2} \notin E(G)$ implies $\delta(G_{v_i,v_{i+2}}) \geq 3$.

We must have that either $w_iw_{i+2},w_iw_{i-2} \in E(G)$ or $w_iw_{i+2}, w_iw_{i-2} \notin E(G)$ for all $i \in [5]$, since otherwise we may without loss of generality assume that $w_1w_3 \notin E(G)$ but $w_1w_4 \in E(G)$. In this case we would get $\nu(G_{v_1,v_3}) \leq 1$ but this would make $w_4$ bivalent in $G_{v_1,v_3}$, contradicting the previous.

Suppose that $w_iw_{i+2} \in E(G)$ for some $i \in [5]$. Then we would have to have $w_jw_{j+2} \in E(G)$ for all $j \in [5]$, i.e. the local structure is as illustrated in Figure \ref{cp10p2}.
\begin{figure}[h]
\begin{center}
\begin{tikzpicture}[scale=1]
  \GraphInit[vstyle=Classic]
  \renewcommand*{\VertexSmallMinSize}{4pt}
  \Vertices[Math,unit=1,Lpos=-180,Ldist=-2]{circle}{v_1,v_2,v_3,v_4,v_5}
  \Vertices[Math,unit=2,Lpos=90]{circle}{w_1,w_2,w_3,w_4,w_5}
  \Vertices[Math,unit=2.4,empty]{circle}{u1,u8,u15,u22,u29}
  \Edges(v_1,v_2,v_3,v_4,v_5,v_1)
  \Edge(v_1)(w_1)
  \Edge(v_2)(w_2)
  \Edge(v_3)(w_3)
  \Edge(v_4)(w_4)
  \Edge(v_5)(w_5)
  \Edge(w_1)(u1)
  \Edge(w_2)(u8)
  \Edge(w_3)(u15)
  \Edge(w_4)(u22)
  \Edge(w_5)(u29)
  \Edge[style={bend right}](w_1)(w_3)
  \Edge[style={bend right}](w_3)(w_5)
  \Edge[style={bend right}](w_5)(w_2)
  \Edge[style={bend right}](w_2)(w_4)
  \Edge[style={bend right}](w_4)(w_1)
\end{tikzpicture}
\end{center}
\caption{The neighbourhood of the $C_5$ from $G_3$ in $G$ with edges between $w_i$s.}
\label{cp10p2}
\end{figure}
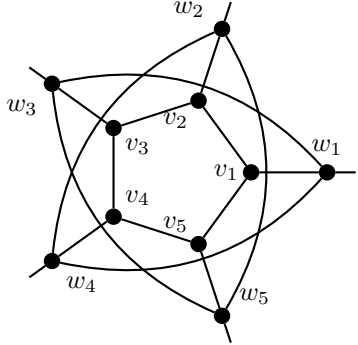

In this case it is impossible to have three consecutive $w_i$s in an independent set of $G\setminus \{v_i; i \in [5]\}$. This means that $\alpha(G \setminus \{v_i; i \in [5]\}) \leq \alpha(G) - 2$, which would make all the $v_iw_i$-edges redundant. Therefore we must have that $w_iw_{i+2} \notin E(G)$ for all $i \in [5]$.

If $d^2(w_i) \geq 16$ for some $i \in [5]$ then $v_i$ is the only trivalent neighbour of $w_i$. We would then have $\nu(G_{w_i}) \leq \nu(G) + 2 \leq 2$, with $v_{i-2}$ and $v_{i+2}$ bivalent. It is easy to see that this is not possible since then they would have to belong to a $C_5$-component. We must therefore have that $d^2(w_i) = 3 + 4 + 4 + 4 = 15$ for all $i \in [5]$.

Suppose that $d^2(G_{w_1,v_2}; v_4) = 6$. Then $\nu(G_{w_1,v_2,v_4}) \leq 1$, with $w_5$ bivalent (since $N(w_1) \cap N(w_5) \neq \emptyset$ and $N(v_4) \cap N(w_5) = \{v_5\}$). This implies that $w_5$ belongs to a $C_5$-component, $C$, of $G_{w_1,v_2,v_4}$. By the previous, however, $V(C)$ contains at least three vertices that are tetravalent in $G$. Hence, $e(C,N(w_1) \cup N(v_2) \cup N(v_4) \setminus \{v_1\}) \geq 8$, but $|N(w_1) \cup N(v_3) \cup N(v_4) \setminus \{v_1\}| \leq 7$, which would therefore give a cycle of length four in $G$. Hence, $d^2(G_{w_1,v_2}; v_4) \leq 5$. Since $e(w_4,N(v_2)) = 0$ we get $N(w_1) \cap N(w_4) \neq \emptyset$. Completely analogously we may show that $N(w_i) \cap N(w_{i+2}) \neq \emptyset$ for all $i \in [5]$.

Now, note that $G_3 = \emptyset$, since $\nu(G_{v_1,v_3}) \leq \nu(G) + 3 - 2 \leq 1$, and the common neighbour of $w_1$ and $w_3$ would have valency 2 in $G_{v_1,v_3}$, contradicting that $w_iw_{i+2} \notin E(G)$ implies $\delta(G_{v_i,v_{i+2}}) \geq 3$ for all $i \in [5]$.

Thus $\delta(G) \geq 4$ and $G$ is 4-regular by Property \ref{mindegleq4} and since $d^2(v) \leq 16$ for all $v \in V(G_4)$.
\end{proof}

\section{Properties of a minimal counterexample}
\label{mincounter} 

We will from here on assume, for a contradiction, that $G$ is a minimal counterexample to the assertion in Theorem \ref{mainthm}, i.e. we assume that either $\nu(G) < 0$ or $\nu(G) = 0$ but $G \notin \mathcal{G}$ but still assuming \eqref{bigass}. If $\nu(G) < 0$ then we in fact have $\nu(G) \leq -1$ since $\nu$ only takes integer values. In particular, $G$ must be a connected graph. Note that this is just a strengthening of assumption \eqref{bigass}, so all the properties for $G$ derived thus far holds also under this stronger assumption on $G$.

Note that by Property \ref{classprop6} we must have $\N(C_4;G) = 0$. Also, $G$ is 4-regular by Property \ref{classprop10}.

\subsection{Graphs with valencies three and four and $\nu$-value two}
\label{graphswithvalencies}

A lot of things in this section are quite close to the works of Radziszowski and Kreher in \cite{radziszowski-kreher91} and the slight modification by Backelin in \cite{jart}. Some of the following results are just reformulations of their results in our particular context.

\begin{lemma}
\label{monovalentsH3}
If $H$ is a graph such that $\delta(H) = 3$, $\Delta(H) \leq 4$, $\nu(H) \leq 2$ and $H \leq G$, then $\delta(H_3) \geq 1$, and every $v \in V(H)$ such that $d(H_3;v) = 1$ belongs to a $K_2$-component of $H_3$ and has no trivalent vertices at distance two from $v$ in $H$.
\end{lemma}
\begin{proof}
We prove this by induction on the number of vertices of $H$. For $n(H) = 1$ the collection of graphs satisfying the properties in the premise is empty, whence the assertion trivially holds. Therefore suppose that $H$ satisfies the premises, $n(H) > 1$ and that the assertion holds for all graphs on fewer vertices than $H$.

Since $\nu(H) \leq 2$ we have that every trivalent vertex $v$ in $H$ has second valency at most eleven, since otherwise $\nu(H_v) \leq \nu(H) - 3 \leq -1$. Hence $\delta(H_3) \geq 1$.

Suppose that $v \in V(H_3)$ with $d(H_3;v)=1$, then $d^2(H;v) = 11$. Note that $\nu(H_v) \leq \nu(H) + 0 = 2$ and therefore we have, by Property \ref{classprop1}, that any bivalent vertex of $H_v$ would lie in a $C_5$-component. However, if there were a $C_5$-component, $C$, in $H_v$ then we would need to have that all five vertices of $C$ is adjacent to some vertex in $N(v)$. This would however give a cycle of length four through some of the vertices in $N(v)$. Thus there are no bivalent vertices in $H_v$ and therefore no trivalent vertices at distance two from $v$ in $H$. Hence if $v \in V(H_3)$ is such that $d(H_3; v) = 1$ then all vertices at distance two from $v$ are tetravalent. This means that $v$ belongs to a $K_2$-component of $H_3$.
\end{proof}

\begin{lemma}
  \label{isittrue}
If $H \leq G$ is such that $\delta(H) = 2$, $v \in V(H_2)$ and $\nu(H) \leq 5$ then one of the following holds
\begin{enumerate}[(i)]
\item $d^2(H; v) \geq 5$.
\item $v \in V(C)$ where $C_5 \iso C \in \calC(H)$.
\item $\exists e \in E(H): v \in V(C)$ where $C_5 \iso C \in \calC(H-e)$.
\end{enumerate}
\end{lemma}
\begin{proof}
Let $H$ and $v$ be as in the premises and suppose that $d^2(H;v) \leq 4$ and that $v$ does not lie in a $C_5$-component of $H$. We then want to show that (iii) holds. By Property \ref{prop-no2regular} $H$ contains no 2-regular components and therefore neither does $H_2$.

Since $\delta(H) = 2$ we get that $d^2(H;v) = 4$, whence $v$ has two bivalent neighbours, $u$ and $w$. If neither $u$ nor $w$ has a bivalent neighbour apart from $v$, then $\nu(H_u) \leq \nu(H) + 1 \leq 6$. Since $e(N(u),N(w))\leq 1$ the monovalent vertex $w$ has second valency at least two in $H_u$, whence $\nu(H_{u,w}) \leq \nu(H_u) - 7 \leq -1$, contradicting assumption \eqref{bigass}.

Hence, at least one of $u$ and $w$ has a bivalent neighbour and $v$ belongs to a path component, $P$, of $H_2$ of length at least four. Let $a,b \in V(P)$ be the endpoints of the path component $P$. If $d^2(H;a) \geq 6$ then $\nu(H_a) \leq \nu(H) - 2 \leq 3$ but there would be a monovalent vertex (the vertex at distance two from $a$ in $P$) in $H_a$, contradicting Property \ref{prop-mindeg2}. Hence, $d^2(H;a) = 5$ and analogously we obtain $d^2(H;b) = 5$. Therefore $\nu(H_a) \leq \nu(H) + 1 \leq 6$ and the vertex at distance two from $a$ in $P$ must belong to a $K_2$-component of $H_a$ by Property \ref{prop-leq6}. This is only possible if $b$ is at distance three from $a$, in $P$, and the trivalent neighbours of $a$ and $b$, say $x$, coincide. The edge $e = xy \in E(H)$ that is neither incident to $a$ nor $b$ is then such that (iii) holds.
\end{proof}

In the remainder of this section we assume $H = G_v$ for some vertex $v \in V(G)$. In particular we then have that $H_3$ is a graph on 12 vertices since $G$ is 4-regular by Property \ref{classprop10}, and by Lemma \ref{monovalentsH3} we have that the minimum valency in $H_3$ is one and every monovalent vertex belongs to some $K_2$-component. Every vertex of $H$ which is not trivalent is tetravalent. 

\begin{lemma}
  \label{nofivecycles}
$H_3$ contains no cycles of length five.
\end{lemma}
\begin{proof}
Since $G$ is $4$-regular we have that each vertex of $V(H_3)$ is adjacent to one of the neighbours of $v$ in $G$. If there were a cycle of length five in $H_3$ then since $|N_G(v)| = 4$ one of the neighbours of $v$ would have to be adjacent to at least two vertices in the cycle, contradicting Lemma \ref{nbrsincycle}.
\end{proof}

Note that Lemmas \ref{monovalentsH3} and \ref{nofivecycles} corresponds to \cite[Lemma 5.2.3]{radziszowski-kreher91}.

\begin{lemma}
  \label{nosixcycles} (Analogue of \cite[Lemma 5.2.4]{radziszowski-kreher91})
$H_3$ contains no cycles of length six.
\end{lemma}
\begin{proof}
Since $G$ is $4$-regular we have that $\nu(H) \leq \nu(G) + 2 \leq 2$.

Suppose that $\{c_1,c_2,c_3,c_4,c_5,c_6\} \subseteq V(H_3)$ is a cycle of length six in $H_3$. Suppose one of the $c_i$ has a tetravalent neighbour in $H$, without loss of generality assume then that $d^2(H;c_1) = 10$. Then $\nu(H_{c_1}) \leq \nu(H) + 3 \leq 5$. If $c_1$ and $c_4$ have a common neighbour then, by Lemma \ref{isittrue}, either there is an edge $e \in E(H_{c_1})$: $c_4 \in V(C)$, where $C_5 \iso C \in \calC(H_{c_1}-e)$, or $c_4 \in V(C)$ where $C \iso C_5 \in \calC(H_{c_1})$. But then the four or five, in $H_{c_1}$, bivalent vertices of $V(C)$ has to be adjacent to the three neighbours of $c_1$. Accordingly there is a vertex with more than one neighbour in a cycle of length five, contradicting Lemma \ref{nbrsincycle}.

Hence, $c_3$ and $c_5$ are bivalent in $H_{c_1}$ with at least one trivalent neighbour, $c_4$. If $c_3$ has second valency at least six in $H_{c_1}$ then $\nu(H_{c_1,c_3}) \leq \nu(H_{c_1}) -2 \leq 3$, but $c_5$ is then monovalent in $H_{c_1,c_3}$ contradicting Property \ref{prop-mindeg2}. Hence $d^2(H_{c_1}; c_3) \leq 5$ and therefore $d(H_{c_1}; x_3) = 2$ where $x_3 \in N_H(c_3) \setminus \{c_2,c_4\}$. Since $H$ contains no cycle of length four, $x_3$ has at most one common neighbour with $c_1$, and therefore $x_3$ has to be trivalent in $H$. Similarly, if $x_5$ is such that $x_5 \in N_H(c_5) \setminus \{c_4,c_6\}$ then $x_5$ is trivalent in $H$ and has a common neighbour with $c_1$. Since $d^2(H_{c_1};c_3) = 5$ we get $\nu(H_{c_1,c_3}) \leq \nu(H_{c_1}) + 1 \leq 6$ and since $c_5$ is monovalent in $H_{c_1,c_3}$ we get, by Property \ref{prop-leq6}, that $x_5$ also is monovalent in $H_{c_1,c_3}$. Hence, $x_3x_5 \in E(H)$ and therefore we have a cycle $c_3c_4c_5x_5x_3$ of length five in $H_3$, contradicting Lemma \ref{nofivecycles}.

Thus all of the vertices $c_i$ of the 6-cycle have second valency nine. Clearly $c_1$ and $c_4$ have no common neighbour or we would get a cycle of length five in $H_3$, contradicting Lemma \ref{nofivecycles}. Hence $\nu(H_{c_1}) \leq \nu(H) + 6 \leq 8$ and $c_3$ is bivalent in $H_{c_1}$ with $d^2(H_{c_1}; c_3) \geq 6$ (since $c_1$ and $c_3$ do not have adjacent neighbours or we would get a cycle of length five or less in $H_3$), whence $\nu(H_{c_1,c_3}) \leq \nu(H_{c_1}) -2 \leq 6$. But $e(N(c_1) \cup N(c_3), N(c_5)) = 0$ or we would get a cycle of length five or less in $H_3$. Therefore $c_5$ is monovalent in $H_{c_1,c_3}$ but has second valency three, giving $\nu(H_{c_1,c_3,c_5}) \leq \nu(H_{c_1,c_3}) - 8 \leq -2$, contradicting assumption \eqref{bigass}.
\end{proof}

\begin{lemma}
  \label{analog525} (Analogue of \cite[Lemma 5.2.5]{radziszowski-kreher91})
Let $x,y \in V(H_3)$ be bivalent in $H_3$. Furthermore suppose $N_H(x) = \{t,x_1,x_2\}$ and $N_H(y) = \{t,y_1,y_2\}$ where $t,x_1,y_1 \in V(H_3)$. Then $x_1y_2,y_1x_2 \in E(H)$.
\end{lemma}
\begin{proof}
We will show that $e(N(x),N(y)) \geq 2$, since then $e(N(x),N(y))=2$, or we would get cycles of length four or less. Since $x_1y_1 \notin E(H)$, by Lemma \ref{nofivecycles}, the only possibility is then that $x_1y_2,y_1x_2 \in E(N(x),N(y))$.

By Lemma \ref{monovalentsH3} both $x_1$ and $y_1$ has another neighbour in $H_3$ except for $x$ and $y$, let those be $x_3 \in V(H_3)$ and $y_3 \in V(H_3)$, respectively. Note that $x_3$ and $y_3$ are distinct by Lemma \ref{nosixcycles}.

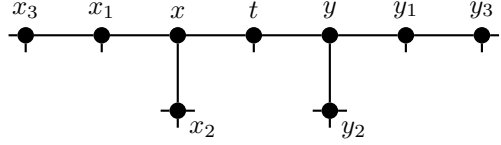
\begin{figure}[h]
\begin{center}
\begin{tikzpicture}[scale=0.5]
  \GraphInit[vstyle=Classic]
  \renewcommand*{\VertexSmallMinSize}{4pt}
  \Vertex[x=0,y=0,L=$x_3$,Lpos=90]{x3}
  \Vertex[x=2,y=0,L=$x_1$,Lpos=90]{x1}
  \Vertex[x=4,y=0,L=$x$,Lpos=90]{x}
  \Vertex[x=6,y=0,L=$t$,Lpos=90]{t}
  \Vertex[x=8,y=0,L=$y$,Lpos=90]{y}
  \Vertex[x=10,y=0,L=$y_1$,Lpos=90]{y1}
  \Vertex[x=12,y=0,L=$y_3$,Lpos=90]{y3}
  \Vertex[x=4,y=-2,L=$x_2$,Lpos=-45,Ldist=-2]{x2}
  \Vertex[x=8,y=-2,L=$y_2$,Lpos=-45,Ldist=-2]{y2}
  \Edge(x3)(x1)
  \Edge(x1)(x)
  \Edge(x)(t)
  \Edge(t)(y)
  \Edge(y)(y1)
  \Edge(y1)(y3)
  \Edge(x)(x2)
  \Edge(y)(y2)

  \Vertex[x=-0.5,y=0,empty]{x31}
  \Vertex[x=0,y=-0.5,empty]{x32}
  \Vertex[x=2,y=-0.5,empty]{x11}
  \Vertex[x=3.5,y=-2,empty]{x21}
  \Vertex[x=4.5,y=-2,empty]{x22}
  \Vertex[x=4,y=-2.5,empty]{x23}
  \Vertex[x=6,y=-0.5,empty]{t1}
  \Vertex[x=7.5,y=-2,empty]{y21}
  \Vertex[x=8.5,y=-2,empty]{y22}
  \Vertex[x=8,y=-2.5,empty]{y23}
  \Vertex[x=10,y=-0.5,empty]{y11}
  \Vertex[x=12,y=-0.5,empty]{y31}
  \Vertex[x=12.5,y=0,empty]{y32}
  \Edge(x31)(x3)
  \Edge(x32)(x3)
  \Edge(x11)(x1)
  \Edge(x21)(x2)
  \Edge(x22)(x2)
  \Edge(x23)(x2)
  \Edge(t1)(t)
  \Edge(y21)(y2)
  \Edge(y22)(y2)
  \Edge(y23)(y2)
  \Edge(y11)(y1)
  \Edge(y31)(y3)
  \Edge(y32)(y3)
\end{tikzpicture}
\end{center}
\caption{The neighbourhood of $x$ and $y$ in $H$.}
\end{figure}

If $e(N(x),N(y)) = 0$, then $\nu(H_{x,y}) \leq 0$ with $x_3$ and $y_3$ bivalent. Thus, $x_3$ and $y_3$ belong to a $C_5$-component of $H_{x,y}$. Each of the five vertices of $N(x) \cup N(y)$ has a single neighbour in the $C_5$-component, which would give a cycle of length five in $H_3$, contradicting Lemma \ref{nofivecycles}.

If $e(N(x),N(y)) = 1$ then we get $d^2(H_x; y) = 6$ and therefore $\nu(H_{x,y}) \leq \nu(H) + 3 - 2 \leq 3$. Note that $x_3$ and $y_3$ are bivalent in $H_{x,y}$. If $\alpha((H_{x,y})_2) > 1$, then (by Properties \ref{classprop1} and \ref{classprop7}) there is a component $C \in \Cc(H_{x,y})$ such that either $C \iso C_5$ or there is an $e \in E(C)$ with the property that $C-e = C_5 + C'$, where $C' \in \{C_5, W_5,(2C_7)_{2i}\}$. By the same reasoning as we used for excluding the $e(N(x),N(y))=0$-case we can see that $C \not \iso C_5$. In the other case, all of the vertices of the $C_5$ in $C-e$ must be adjacent to one (and exactly one) of the vertices in $N(x) \cup N(y)$, and vice versa. The reason for this is that we have no bivalent vertices in $H$ and no cycles of length five in $H_3$. There are six additional edges incident to $N(x) \cup N(y)$. The vertices $x_3,x_1,x,t,y,y_1,y_3$ and two more vertices of the $C_5$-part of $C-e$ are trivalent in $H$. Thus, $|V(C') \cap V(H_3)| \leq 3$, since we have in total twelve trivalent vertices, and the nine listed previously are not in $V(C')$. But $|V((C')_3)| = 14$, so $|V(C_3)| = 14$ and therefore $|V(C_3) \cap V(C')| = 13$. Clearly it is then impossible to have $|V(C') \cap V(H_3)| \leq 3$, which is a contradiction.

Thus, $\alpha((H_{x,y})_2) = 1$. Therefore $y_3x_3 \in E(H)$ and $y_3,x_3$ are the only two bivalent vertices of $H_{x,y}$. We have also $d^2(x_1), d^2(y_1) \geq 10$. Observe that $t$ and $y_3$ are non-adjacent bivalent vertices in $H_{x_1}$. This would however contradict Corollary \ref{classcor7}.

Hence, $e(N(x),N(y)) \geq 2$, as desired.
\end{proof}

\begin{lemma} \label{analog526}
(Analogue of \cite[Lemma 5.2.6]{radziszowski-kreher91})
If $t \in V(H)$ is trivalent in $H_3$ then it has two bivalent and one trivalent neighbour in $H_3$.
\end{lemma}
\begin{proof}
Recall that the number of trivalent vertices in $H$ is 12.

That $t \in V(H)$ is trivalent in $H_3$ means that $d^2_H(t) = 3 + 3 + 3$. Let $W = \{w_1,w_2,w_3\} = N(v)$ and suppose that at least two of the vertices $w_1,w_2,w_3$ have second valency nine. Then at least five out of the six distinct vertices at distance two from $t$ are trivalent in $H$, and each has at least one other trivalent neighbour not in $W$. Since there are no cycles of length five or six in $H_3$ by Lemmas \ref{nofivecycles} and \ref{nosixcycles} there is then at least five trivalent vertices at distance three from $t$.

Hence, there is one trivalent at distance 0 from $t$, three at distance $1$, at least five at distance $2$ and at least five at distance $3$, whence there are at least $14$ trivalent vertices in $H$, contradicting that there are 12 such vertices.

Therefore, at most one of the vertices $w_1,w_2$ and $w_3$ has second valency nine. If none of them had second valency nine then each of them has a trivalent and a tetravalent neighbour except for $t$. Let $x_1$ be the trivalent neighbour of $w_1$ that is not $t$, moreover let $y_2$ and $y_3$ be the tetravalent neighbours of $w_2$ and $w_3$, respectively. By applying Lemma \ref{analog525} (for $(x,y) = (w_1,w_2), (w_1,w_3)$) we get that $x_1y_2,x_1y_3 \in E(H)$. But then $w_1$ is the only trivalent neighbour of $x_1$, contradicting Lemma \ref{monovalentsH3}.

Hence, exactly one of the vertices $w_1,w_2$ and $w_3$ has second valency nine which completes the proof.
\end{proof}

We define two graphs on twelve vertices, $S_1$ and $S_2$, just as in \cite{radziszowski-kreher91}.

\begin{definition}
Denote the vertices of $C_{12}$ by $\{c_0,c_1,\dots,c_{11}\}$ so that $c_0c_1 \dots c_{11}$ forms the cycle of length 12. The graph $S_1$ is formed by adding the edge $c_0c_6$ to $C_{12}$ and the graph $S_2$ is formed by adding the two edges $c_0c_6$ and $c_3c_9$ to the $C_{12}$. 
\end{definition}

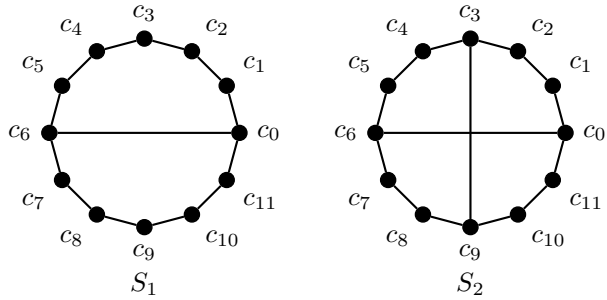
\begin{figure}[h]
\begin{center}
\begin{tikzpicture}[scale=0.5]
  \GraphInit[vstyle=Classic]
  \renewcommand*{\VertexSmallMinSize}{4pt}
  \SetGraphUnit{2.5}
  \Vertices[Math]{circle}{c_0,c_1,c_2,c_3,c_4,c_5,c_6,c_7,c_8,c_9,c_{10},c_{11}}
  \Edges(c_0,c_1,c_2,c_3,c_4,c_5,c_6,c_7,c_8,c_9,c_{10},c_{11},c_0)
  \Edge(c_0)(c_6)
  \node at (0,-4) {$S_1$};
\end{tikzpicture} \quad \begin{tikzpicture}[scale=0.5]
  \GraphInit[vstyle=Classic]
  \renewcommand*{\VertexSmallMinSize}{4pt}
  \SetGraphUnit{2.5}
  \Vertices[Math]{circle}{c_0,c_1,c_2,c_3,c_4,c_5,c_6,c_7,c_8,c_9,c_{10},c_{11}}
  \Edges(c_0,c_1,c_2,c_3,c_4,c_5,c_6,c_7,c_8,c_9,c_{10},c_{11},c_0)
  \Edge(c_0)(c_6)
  \Edge(c_3)(c_9)
  \node at (0,-4) {$S_2$};
\end{tikzpicture}
\end{center}
\caption{The two graphs $S_1$ and $S_2$.}
\end{figure}

\begin{lemma} \label{analog527}
(Analogue of \cite[Lemma 5.2.7]{radziszowski-kreher91}) If $C \in \calC(H_3)$, then \\
$C \in \{K_2,C_8,C_{10},C_{12},S_1,S_2\}$.
\end{lemma}
\begin{proof}
Recall that $n(H_3) = 12$. If $C \in \calC(H_3)$ then $\delta(C) \geq 2$ unless $C = K_2$ by Lemma \ref{monovalentsH3}. Suppose therefore that $\delta(C) \geq 2$, then $n(C) \geq 7$ because $C$ contains no cycles of length five or six, by Lemmas \ref{nofivecycles} and \ref{nosixcycles}. Since the only possible component in $H_3$ with no more than six vertices is $K_2$ we get that $C$ has an even number of vertices.

Then either $C$ is $2$-regular and then a cycle of length eight, ten or twelve, or it has at least one vertex of valency three. Suppose therefore that $C$ has a trivalent vertex, then it has at least two, and every trivalent vertex is paired with one other adjacent trivalent vertex, by Lemma \ref{analog526}.

If there are two trivalent vertices, then the only possible graph is $C = S_1$ and if there are four trivalent vertices the only possibility is $C = S_2$. It is not possible to have more than four trivalent vertices in $H_3$.
\end{proof}

\begin{lemma} \label{analog528} (Analogue of \cite[Lemma 5.2.8]{radziszowski-kreher91})
Let $C \in \calC(H_3)$, then $C \notin \{C_8,C_{10},C_{12}\}$.
\end{lemma}
\begin{proof}
Suppose that $C \in \calC(H_3)$ and $C = C_k$ where $k \in \{8,10,12\}$. Let the vertices of $C_k$ be cyclically labelled by $c_1,c_2,\dots,c_k$. By Lemma \ref{analog525} we have that two vertices at distance three in the cycle must have a common tetravalent neighbour in $H$.

If $k=8$ then $c_1$ and $c_4$ have a common tetravalent neighbour, but so does $c_4$ and $c_7$. Since the vertices of $V(C)$ only have one tetravalent neighbour each we then get that all three vertices $c_1,c_4$ and $c_7$ have a common tetravalent neighbour, contradicting Lemma \ref{nbrsincycle}.

If $k=10$ then the following pairs of vertices have common tetravalent neighbours: $(c_1,c_4),(c_4,c_7)$ and $(c_7,c_{10})$. But then $c_1,c_4,c_7,c_{10}$ would all be adjacent to the same tetravalent vertex, again contradicting Lemma \ref{nbrsincycle}.

Finally, if $k = 12$ then we get in the same manner as above that $c_i,c_{i+3},c_{i+6}$ and $c_{i+9}$ have a common tetravalent neighbour, $v_i$, for $i \in \{1,2,3\}$. Since each of the three tetravalent vertices $v_1,v_2$ and $v_3$ have four neighbours in $V(C)$ they form a component $C'$ of $H$. However, $C'$ is a graph with 24 edges, 15 vertices and independence number 6 and therefore, $\nu(C') = 27$, contradicting $\nu(H) \leq 2$ since $H-C' < G$ and $G$ is assumed to be a minimal counterexample to Theorem \ref{mainthm}.
\end{proof}

\begin{lemma} \label{analog529} (Analogue of \cite[Lemma 5.2.9]{radziszowski-kreher91})
Let $C\in \calC(H_3)$, then $C \notin \{S_1,S_2\}$.
\end{lemma}
\begin{proof}
For $C \in \{S_1,S_2\}$, then by applying Lemma \ref{analog525} repeatedly we get that $c_i,c_{i+3},c_{i+6}$ and $c_{i+9}$ have a common tetravalent neighbour, $v_i$, for $i \in \{1,2\}$. If $C = S_2$, then $v_1$, $v_2$ and $V(C)$ form a component $C'$ of $H$. This component would have 14 vertices, 22 edges and independence number 5, whence $\nu(C') = 3$, contradicting that $\nu(H) \leq 2$ and $\nu(H-C') \geq 0$.

If $C = S_1$, then $c_3$ and $c_9$ have a (possibly common) tetravalent neighbour that is not $v_1$ or $v_2$. Let $S = \{c_1,c_3,c_5,c_7,c_9,c_{11}\}$. If $N(c_3) \cap N(c_9) = \emptyset$, then $n(H_S) = n(H) - 16$, $e(H_S) = e(H)-29$ and $\alpha(H_S) \leq \alpha(H) - 6$. Hence $\nu(H_S) \leq \nu(H) - 25 \leq -23$, contradicting assumption \eqref{bigass}. On the other hand if $N(c_3) \cap N(c_9) \neq \emptyset$ then $n(H_S) = n(H) - 15$, $e(H_S) = e(H) - 25$ and $\alpha(H_I) \leq \alpha(H) - 6$, whence $\nu(H_S) \leq \nu(H) - 30$ also contradicting assumption \eqref{bigass}.
\end{proof}

Note that by Lemmas \ref{analog527}, \ref{analog528} and \ref{analog529} we get that $H_3 \iso 6K_2$, since $H_3$ has twelve vertices. 

\begin{lemma}
\label{nosixcyclesG}
$G$ contains no cycles of length six.
\end{lemma}
\begin{proof}
Suppose that the vertices $\{c_1,c_2,\dots,c_6\} \subseteq V(G)$ formed a cycle of length six (labelled cyclically in order). Let $H = G_{c_1}$, then since $H_3 \iso 6K_2$ we get that $d(H_3;c_3) = 1$. However, $c_5$ would also be trivalent in $H$ and at distance two from $c_3$, contradicting Lemma \ref{monovalentsH3}.
\end{proof}

\begin{lemma}
\label{incidentfivecycles}
Through every pair of incident edges in $G$ there is a cycle of length five.
\end{lemma}
\begin{proof}
Suppose otherwise and let $ux,xv \in E(G)$ be a pair of incident edges through which there is no cycle of length five. We have $d^2(u) = 16$ and therefore $\nu(G_u) \leq \nu(G) + 2 \leq 2$. Since there is no cycle of length five through $u,x,v$ we have that $e(N(u),N(v)) = 0$ and therefore $d^2(G_u;v) = 12$. This would however give $\nu(G_{u,v}) \leq \nu(G_v) - 3 \leq -1$, contradicting assumption \eqref{bigass}.
\end{proof}

\begin{lemma}
\label{fiveshare}
Two cycles of length five in $G$ share at most one edge.
\end{lemma}
\begin{proof}
Suppose otherwise, then the two cycles, $C_1$ and $C_2$, of length five would share exactly two consecutive edges or we would get a cycle of length four or less in $G$. Suppose therefore that the edges shared are $x_1v \in E(G)$ and $vx_2 \in E(G)$. Let $x_3$ and $x_4$ be the two remaining neighbours of $v$, $H = G_v$ and $X_i = N(x_i)\setminus \{v\}$. Because of Lemma \ref{incidentfivecycles} there are also cycles of length five through the pairs of incident edges $(x_2v,vx_3)$ and $(x_2v,vx_4)$, whence $|E(X_2,X_3)|,|E(X_2,X_4)| \geq 1$. But since two of the vertices in $X_2$ belong to $C_1$ or $C_2$, and as such get paired with a vertex in $X_1$ we must have that there are edges from $X_3$ and $X_4$ to the same vertex in $X_2$, which then is bivalent in $H_3$, contradicting that $H_3 \iso 6K_2$.
\end{proof}

Now we are ready to prove a lemma that will contradict Lemma \ref{nosixcyclesG} and complete the proof of the main theorem (Theorem \ref{mainthm}).

\begin{lemma} \label{sixcyclesG} (Analogous of an argument in the proof of \cite[Theorem 3]{jart})
$G$ contains at least one cycle of length six.
\end{lemma}
\begin{proof}
Let $uv \in E(G)$. By Lemmas \ref{incidentfivecycles} and \ref{fiveshare} there are two cycles $C,C'$ of length five through $uv$ which do not share any other edge than $uv$. Let $S = \{x_1,x_2,x_3,x_4\} = N(\{u,v\}) \cap (V(C) \cup V(C'))$ where $x_1,x_2$ are adjacent to $u$ and $x_3,x_4$ are adjacent to $v$. If the only edges between the neighbourhoods of the $x_i$ were the edges in $E(N(x_1),N(x_2))$ and $E(N(x_3),N(x_4))$ guaranteed by Lemma \ref{incidentfivecycles} then we would get $\nu(G_S) \leq \nu(G) + 2 + 0 + 0 - 5 = -3$, contradicting assumption \eqref{bigass}. If $e(N(x_{i}),N(x_{i+1})) \geq 2$, where $i \in \{1,3\}$, then we either get a cycle of length four or a cycle of length six through $x_i$ and $x_{i+1}$.

If there are no cycle of length six through the $x_1,x_2$ or $x_3,x_4$ we must instead have an edge in $E(N(x_1)\cup N(x_2) \setminus \{u\},N(x_3)\cup N(x_4) \setminus \{v\})$ which gives a cycle of length six through $uv$.
\end{proof}

\subsection{Proof of main theorem}
\label{conclusion}

We have now shown that if $G$ is a minimal counterexample, either to $\nu$ being non-negative or to $G$ being such that $G \notin \mathcal{G}$ but with $\nu$-value zero, then $G$ must contain cycles of length six by Lemma \ref{sixcyclesG} but also not contain cycles of length six by Lemma \ref{nosixcyclesG}. Hence, there is no minimal counterexample and therefore no counterexample at all. We have therefore shown Theorem \ref{mainthm} to hold. We know that all triangle-free graphs have non-negative $\nu$-value and we that those that have $\nu$-value zero are exactly those in $\{W_5, (2C_7)_{2i}\} \cup \{Ch_k; k \geq 2\} \cup \{BC_k; k \geq 5\}$.

Now, since Theorem \ref{mainthm} are proven, all the properties of Section \ref{theinvariant} that were made under assumption \eqref{bigass} now can be read as general propositions. For example we have fully classified the graphs with $\nu$-value at most two with bivalent vertices (see Property \ref{classprop1}). 

\printbibliography

\end{document}